\documentclass[smallextended]{article}

\usepackage{graphicx}
\usepackage{amsmath}
\usepackage{amssymb}
\usepackage[utf8]{inputenc}
\usepackage[T1]{fontenc}
\usepackage[english]{babel}
\usepackage[all]{xy}
\usepackage{enumitem}
\usepackage[top=25mm, bottom=35mm, left=30mm , right=30mm]{geometry}
\usepackage{amssymb,amsmath,amscd,amsfonts,amsthm,mathrsfs}
\usepackage{geometry}
\usepackage{times}  
\usepackage{fancyhdr}
\usepackage{pinlabel}
\usepackage{array}

\newcommand{\R}{\mathbf R} 
\newcommand{\N}{\mathbf N}
\newcommand{\Z}{\mathbf Z}

\newtheorem{theorem}{Theorem}[section]   
\newtheorem{definition}[theorem]{Definition}
\newtheorem{example}[theorem]{Example}
\newtheorem{corollary}[theorem]{Corollary}
\newtheorem{proposition}[theorem]{Proposition}

\newtheorem{lemma} [theorem]{Lemma}
\newtheorem{remark}[theorem]{Remark}
\newtheorem{fact}[theorem]{Fact}
\usepackage{hyperref}
\hypersetup{
	colorlinks=true,
	linkcolor=blue,
	filecolor=magenta,      
	urlcolor=cyan,
}
\begin{document}
\title{Maximal simply connected Lorentzian surfaces \\with a Killing field and their completeness}
\author{Lilia Mehidi}
\maketitle
\begin{abstract}
In the first part of this paper, we give a global description of simply connected maximal Lorentzian surfaces whose group of isometries is of dimension 1 (i.e. with a complete Killing field), in terms of a $1$-dimensional generally non-Hausdorff manifold (the space of Killing orbits) and a smooth function defined there. 
In the second part, we study the completeness of such surfaces and prove in particular that under the hypothesis of bounded curvature, completeness is equivalent to null completeness. We also give completeness criterions involving the topological structure of the space of Killing orbits, or both the topology of this space and the geometry of the surface. 
\end{abstract}

\section{Introduction}
\subsection{On maximal simply connected Lorentzian surfaces with a Killing field}
The local geometry of a Lorentzian surface with a Killing field is well known. Denote by $K$ the Killing field; locally one of the two null foliations is defined by a vector field $L$ such that $\langle K, L \rangle =1$, and the metric is locally given by $2dxdy+f(x)dy^2$, where $L= \partial_{x}, K = \partial_{y}$, and $f$ is a $y$-independent scalar function. So the metric is locally determined by the norm of the Killing field. The globalization of this picture cannot be done in general; the obstruction to globalizing this picture is the possible presence of Reeb components in the foliation defined by the Killing field or the orthogonal foliation, which amounts to saying that none of the two null foliations is globally transverse to the Killing field (this has been studied in \cite{BM}; we refer to Section \ref{Section préliminaires} of this paper for details). 

In the first part of this paper, we will see that the global structure of a Lorentzian surface $(X,K)$ with a Killing field $K$ can be analyzed in terms of the space $\mathcal{E}_X$ of Killing orbits and a smooth function defined there. This leads us in the first place to the study of the topology and the geometry of this space. 

A first problem is that $\mathcal{E}_X$ can lack even the fundamental properties of a manifold; for example, consider a flat $2$-dimensional torus, and take as Killing field any constant vector field with irrational slope, then the lines of the Killing field are dense in $T^2$ and the space of leaves has the trivial topology. Nothing like this occurs in $\R^2$, where the leaves can be neither closed curves nor dense. To prevent this kind of behaviour of the Killing orbits, we assume that $X$ is simply connected, hence homeomorphic to $\R^2$. In this case, if $K$ is supposed to be non singular, then $\mathcal{E}_X$ is a $1$-dimensional simply connected manifold with a countable base (usually non-Hausdorff): Haefliger-Reeb \cite{MR0089412}  (see \cite[Definition 1 p. 112]{MR0089412} for the definition of simple connectedness in the case of non-Hausdorff manifolds). 

Let $(X,K)$ be a simply connected Lorentzian surface with a Killing field $K$. The Killing field defines a singular foliation of the plane whose singularities are saddle points.  The space of leaves of such a foliation is no longer a simply connected manifold, but it gives rise to a nice manifold $\mathcal{E}_X$ of dimension $1$ such that there exists a local diffeomorphism $\textbf{x} \in C^{\infty}(\mathcal{E}_X,\R)$ (see the last paragraph of Section \ref{Section préliminaires}).  This local diffeomorphism induces a Riemannian structure on $\mathcal{E}_X$ by taking the pullback by $\textbf{x}$ of the Euclidean metric on $\R$. We want to give a description of simply connected Lorentzian surfaces $(X,K)$ with a Killing field $K$ using the two following data: 
\begin{itemize}[parsep=0pt, itemsep=0pt]
	\item the space $\mathcal{E}_X$ of Killing orbits, which is a $1$-dimensional (generally non Hausdorff) orientable manifold, equipped with some combinatorial data that we call "a linking structure" (see Definition \ref{Définition linking structure}), together with a local diffeomorphism $\textbf{x} \in C^{\infty}(\mathcal{E}_X,\R)$ that defines on $\mathcal{E}_X$ a Riemannian structure,
	\item a smooth function defined on it, given by the function induced by the norm of $K$ on $\mathcal{E}_X$.
\end{itemize}
The simple-connectedness is not the end of the problem however. Indeed, in general, the two
data given above do not determine the global structure of the surface. To see this, consider a non-flat torus $(T,K)$ with a Killing field $K$ admitting a null orbit. The flow of $K$ induces a free action of the group $S^1$ on $T$ \cite[Theorem 3.25]{BM}, so in particular, the space of leaves is a circle. Denote by $\widetilde{K}$ the lift of $K$ to the universal cover. There is a local diffeomorphism 
$\textbf{x} \in C^{\infty} (\mathcal{E}_{\widetilde{T}}, \R)$
which is actually a global diffeomorphism since $\mathcal{E}_{\widetilde{T}}$ is Hausdorff of dimension $1$, so the norm of $\widetilde{K}$ on the universal cover can be written $\langle \widetilde{K}, \widetilde{K} \rangle = f \circ \textbf{x}$, where $f \in C^{\infty}(\R,\R)$ is a non constant periodic function that vanishes. Now,  given a smooth non constant periodic function $f$ that vanishes, one can associate to it infinitely many metrics on the torus with non-isometric universal cover. 
\begin{theorem}[Bavard-Mounoud, \cite{BM}]
	The isometry class of the universal cover of a non-flat torus with a Killing field admitting a null orbit is determined by a non constant periodic function $f \in C^{\infty}(\R,\R)$, that vanishes, together with a countable set of periodic marking points on $\R$.
\end{theorem}
This is a consequence of the following important fact: given a simply connected Lorentzian surface $(X,K)$, the presence of a Killing field leads to the existence of local reflections defined in the interior of the connected components of the set $X \smallsetminus \{\langle K,K \rangle =0\}$; these reflections are isometries permuting the null foliations and they are called "\textit{generic reflections}" (see \cite[Proposition 2.5]{BM}).

Indeed, let $\mathcal{E}$ be a $1$-dimensional Hausdorff manifold homeomorphic to a circle, whose universal cover is equipped with a diffeomorphism $\textbf{x} \in C^{\infty}(\widetilde{\mathcal{E}},\R)$. The pullback of the Euclidean metric $d x^2$ on $\R$ by $\textbf{x}$ defines a Riemannian metric on $\widetilde{\mathcal{E}}$. Let $F \in C^{\infty}(\mathcal{E},\R)$ be a non constant smooth function on $\mathcal{E}$ that vanishes, and define $f \in C^{\infty}(\R,\R)$ by $\widetilde{F}=f \circ \textbf{x}$, where $\widetilde{F}$ is the lift of $F$ to the universal cover.  One can define a Lorentzian surface $2dxdy +f(x)dy^2, \,(x,y) \in \R^2$. Since $f$ is periodic, this surface is the universal cover of a torus $(T,K)$ with Killing vector field $K=\partial_{y}$, whose space of leaves is isometric to $(\R, d x^2)$ and whose norm is given by $f$ in the $x$-coordinate. Observe that the null foliation tangent to $\partial_{x}$ here is everywhere transverse to $K=\partial_y$; such a surface  is called "a ribbon" in \cite{BM} (see Definition \ref{Définition ruban/domino} in this paper).  However one can get another isometry class of Lorentzian metric on the torus with a Killing field, in the following way: define a periodic sequence of surfaces $U_i:=(I_i \times \R, 2dxdy +f_i(x)dy^2), i \in S$, where $f_i:=f_{|I_i}$,  such that $I_i \cap I_{i+1}$ is a connected component of $\R \smallsetminus \{f=0\}$, and glue $U_i$ and $U_{i+1}$ using a generic reflection. We obtain again the universal cover of a torus with Killing vector field $K= \partial_{y}$, whose space of leaves is isometric to $(\R, d x^2)$ and whose norm is given by $f$ in the $x$-coordinate. If $S$ has only one element, the surface obtained is a ribbon like before. As soon as $S$ has more than one element, there is no longer a null foliation everywhere transverse  to $K$ in the torus. \\

The problem above is avoided by supposing that the surface is maximal. A connected Lorentzian surface is said to be maximal if it is inextendible, i.e. if it is not (isometric to) a proper open submanifold of some connected Lorentzian surface. In Subsection \ref{Sous-section 1.2 de l'introduction (complétude)} of the introduction, it appears that a non-flat torus $(T,K)$ with a Killing field $K$ admitting a null orbit is actually not maximal.\\ 

Let $(X,K)$ be a simply connected surface with a Killing field $K$. First, we examine the topology of the $1$-dimensional manifold given by the space of leaves; a topology satisfying these constraints will be referred to as the $\mathfrak{T}$ topology (see Definition \ref{Définition topologie T}). This will give in particular a topological characterization of the space of Killing orbits of such surfaces when the set of branch points of $\mathcal{E}_X$ is locally finite.
Then, we give a description of simply connected and maximal Lorentzian surfaces admitting a nontrivial
Killing field. We obtain the following result:
\begin{theorem} \label{Introduction, Proposition: correspondence between maximal surfaces and étalé manifolds}
	There is a bijection between smooth simply connected and maximal Lorentzian surfaces $(X,K)$ admitting a non-trivial complete Killing  field $K$ (up to isometry), and the quadruplets $(\mathcal{E},\mathcal{A},\textbf{x},F)$, up to equivalence (see Theorem \ref{Proposition: correspondance entre surfaces maximales et variétés étalées} for more details), where 
	
	\begin{enumerate}[itemsep=0pt, topsep=0pt, parsep=0pt]
		
		\item $\mathcal{E}$ is a $1$-dimensional smooth manifold with a countable base, and topology $ \mathfrak{T} $,
		\item $\mathcal{A}$ is a linking structure on $\mathcal{E}$,
		\item $ \textbf{x}: \mathcal{E} \to \R $ is a smooth local diffeomorphism, defined up to translation and change of sign, defining on $ \mathcal{E} $ a translation structure,
		\item $ F \in C^{\infty}(\mathcal{E}, \R) $ is an $\mathcal{A}$-inextensible smooth function  such that the set $\{F=0\}$ is composed of 
		
		(a) the closure of the set of branch points, with simple zeros on cycles of branch points,
		
		(b) elements $ I \in \Sigma $ whose closure in $ \mathcal{E} $ is Hausdorff.
	\end{enumerate}
\end{theorem}
Here, $\Sigma$ denotes the set of connected components of the interior of $\mathcal{E} \smallsetminus B$, where $B$ is the set of branch points of $\mathcal{E}$.

If the space of leaves is Hausdorff (for example, if $X$ is the universal cover of a compact surface), there is only one linking structure on it, so this data is only of interest when the space of leaves is non Hausdorff.  \\

Once the correspondence in Theorem \ref{Introduction, Proposition: correspondence between maximal surfaces and étalé manifolds} is obtained, we will give in the second part of this paper some classes of surfaces in which one can say something about geodesic completeness. In particular, we can describe all complete and simply connected  Lorentzian surfaces admitting a Killing field, with bounded sectional curvature. 

\subsection{On the completeness of Lorentzian surfaces with a Killing field}\label{Sous-section 1.2 de l'introduction (complétude)}

\textbf{Compact case: Hopf-Rinow vs Clifton-Pohl.} Hopf-Rinow's theorem states that a Riemannian metric is geodesically complete if and only if the canonically associated distance is complete; in particular any compact Riemannian manifold is complete; the same holds for any Riemannian manifold that is globally homogeneous. There are no analogous conclusions in the Lorentzian case, and it is well known that a compact Lorentzian manifold may be geodesically incomplete. More generally (see \cite[Theorem p. 229]{MR1306563}), if a Lorentzian torus is null complete, then its null foliations contain no Reeb components. In particular, the Clifton-Pohl torus is null incomplete (see Figure 1 below), and all the complete metrics belong to the same connected component of the space of Lorentzian metrics, the one containing the flat metrics.
\begin{figure}[h!] 
	\labellist 
	\small\hair 2pt 
	\endlabellist 
	\centering 
	\includegraphics[scale=0.17]{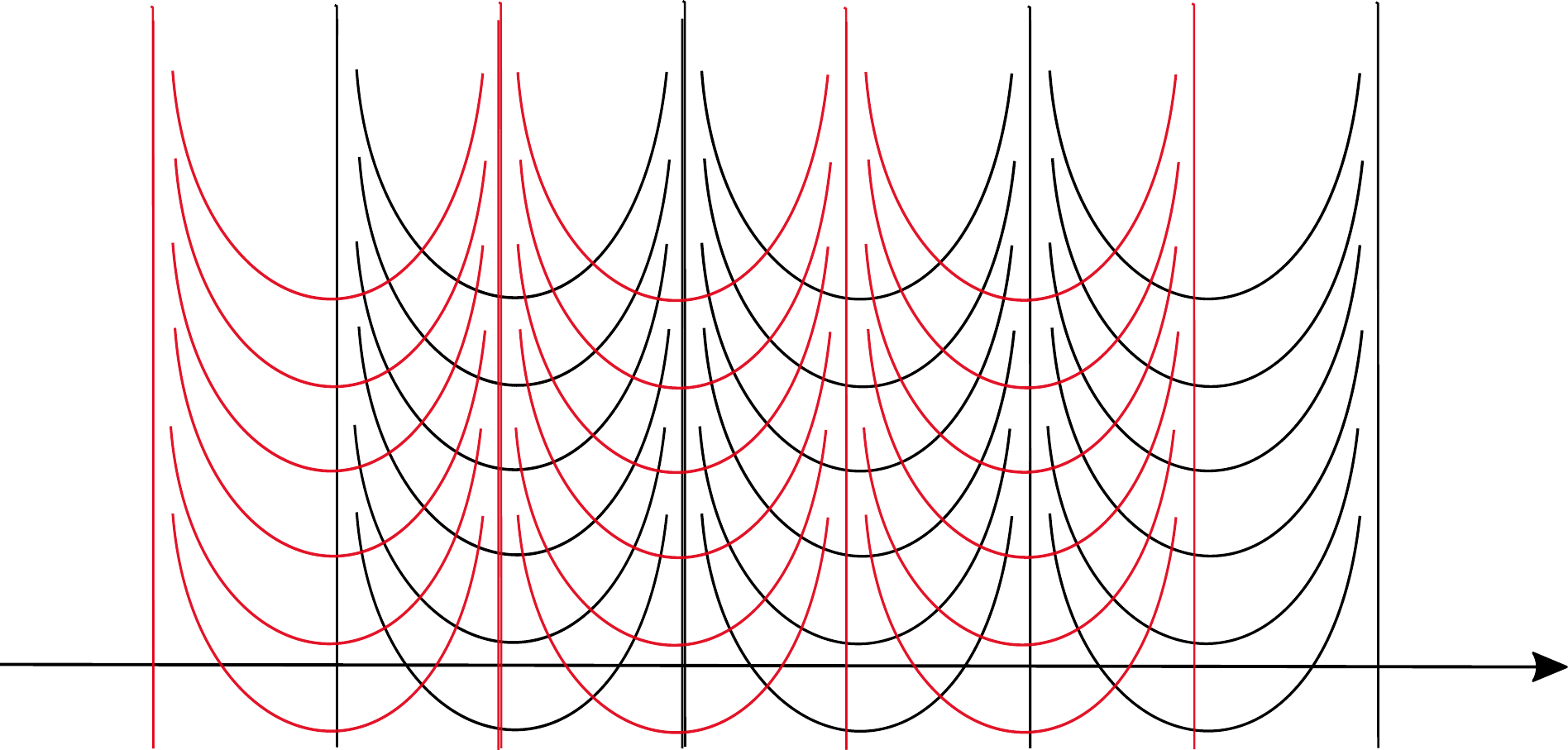} \caption{The null foliations of the Clifton-Pohl torus contain Reeb components; consequently, the Clifton-Pohl torus is null incomplete.}
\end{figure}

In \cite{MR1230550} and \cite{MR1267937}, Romero and Sanchez give different examples and different ways to obtain incomplete Lorentzian metrics on the torus $ T^2 $. At the same time, Guediri and Lafontaine obtained in \cite{MR1310948} examples of incomplete compact and locally homogeneous Lorentzian manifolds. \\
The lack of completeness in the compact case gave rise to the obtaining  of additional conditions which, joint to compactness, would imply completeness of the Lorentzian manifold. Thus, the question of geodesic completeness in the pseudo-Riemannian setting is often viewed from the point of view of $ (G, X) $- structures: under what assumptions are the $ (G, X) $- compact manifolds (where $G$ preserves a pseudo-Riemannian metric) geodesically complete ? When $ G $ preserves a Riemannian metric, all $ (G, X) $- compact manifolds are complete, by the Riemannian Hopf-Rinow theorem. Carrière shows in \cite{MR979369} that any flat Lorentzian metric on $ T^2 $ is complete; these surfaces are locally modeled on Minkowski plane. Klingler generalizes this to Lorentzian manifolds of constant curvature \cite{MR1411352}. Sanchez shows in \cite{MR1376554} that a non-flat $2$-dimensional torus admitting a Killing field with a null orbit is never complete. This becomes clear if one remembers that the universal cover of such a torus is a proper open subset of some maximal Lorentzian surface homeomorphic to $\R^2$, called its "universal extension" (this is a result by Bavard and Mounoud \cite[Proposition 3.3, Theorem 3.25]{BM}). These extensions are proved to be complete in \cite{LM}.  In this paper, we study more generally the completeness of simply connected Lorentzian surfaces admitting a Killing field.

In their paper \cite{MR1230550}, Romero and Sanchez conjecture that the null completeness of a compact Lorentzian manifold implies completeness; which is proved to be true for locally homogeneous manifolds \cite{MR977043}. If we consider tori with one of the two null foliations by circles, then this is "generically true" \cite{MR1306563}, but there are arguments to think that maybe a counterexample can be found. We obtain the following result:
\begin{theorem} \label{Completeness to bounded curvature, introduction}
	A smooth Lorentzian surface with a complete Killing field and bounded curvature is complete if and only if it is null complete.
\end{theorem}
The null completeness of the surface is equivalent to saying that $\mathcal{E}$ satisfies a certain property of "weak " geodesic completeness, defined relative to the linking structure. \\

In the study of geodesic completeness in this paper, special attention is given to geodesics that stay in a maximal ribbon after a certain while. For a fixed ribbon, either they are all complete, or all incomplete, and their completeness is characterized in terms of a condition involving the norm of the Killing field in the given ribbon. With conditions involving this time only the topology of the space of leaves, we give a large class of surfaces where all the geodesics of the surface end in some maximal ribbon, hence in which geodesic completeness is characterized. Among them we have the "small surfaces" (Definition \ref{Définition petites surfaces}) which are maximal surfaces with a finite number of ribbons.

We also give completeness criterions involving both the topological structure of the space of Killing orbits and the geometry of the surface. If $(X,K)$ is a Lorentzian surface with a Killing field $K$, the local diffeomorphism $\textbf{x} \in C^{\infty}(\mathcal{E}_X,\R)$ on the space of leaves induces a Riemannian structure on it by taking the metric $\sigma =\text{d}\textbf{x}^2$. This defines a pseudometric topology on $\mathcal{E}_X$ for which one can consider a stronger completeness condition than the geodesic completeness of the space of leaves, regarded as a Riemannian manifold. In the following results, we say that $\mathcal{E}_X$ is complete if it is complete for this pseudometric topology.
\begin{theorem}
Let $(X,K)$ be a simply connected null complete Lorentzian surface with Killing field $K$. Suppose that the transverse derivative of the norm of $K$ is a bounded function, then $X$ is complete if and only if geodesics orthogonal to $K$ are complete. If in addition  $\mathcal{E}_X$ is complete, then $X$ is complete.
\end{theorem}
\begin{theorem}
Let $(X,K)$ be a simply connected Lorentzian surface with Killing field $K$, such that the norm of $K$ is a bounded function. Suppose that $\mathcal{E}_X$ is complete and contains no infinite chain of branch points, then $X$ is complete. 
\end{theorem}
Here, an infinite chain of branch points is a countable set $\{p_i\}_{i \in I}$ of points in $\mathcal{E}_X$ such that for all $i \in I$, there is $j \in I, j \neq i$, such that $p_i$ and $p_j$ cannot be separated by disjoint neighborhoods. \\

We end this paper by giving various examples illustrating completeness behaviours for the geodesics of a null complete Lorentzian surface, which may appear when omitting certain conditions in the completeness results of the previous paragraphs. We obtain for example a  simply connected surface all of whose geodesics are complete, except timelike geodesics from a certain point.
\\

This paper is organized as follows: in paragraph 2 we introduce the fundamental tools and notions from \cite{BM} dealing with the geometry of compact Lorentzian surfaces with a Killing field, and set up the necessary vocabulary to reading this paper. In paragraph 3, we obtain a description of maximal and simply connected Lorentzian surfaces admitting a Killing field using the space of Killing orbits. 
Paragraph 4 deals with geodesic completeness of these surfaces.

\begin{center}
	\textsc{Acknowledgments}
\end{center}
The author would like to gratefully acknowledge her thesis advisor Christophe Bavard for his insight, his support and encouragements, and the long hours of discussion he devoted to her, without which this article would not have been completed. She also thanks Pierre Mounoud  for drawing her attention to the existence of the example 3.23 of this paper.
\newpage
\tableofcontents
\section{Structure of Lorentzian surfaces admitting a Killing field}\label{Section préliminaires}

All the material exposed in this section, dealing with the geometry of Lorentzian surfaces with a Killing field, has been investigated in \cite{BM}. For the convenience of the reader, we will, as much as possible, try to detail and motivate certain constructions from \cite{BM} which seem important to us, and which would help to acquire a certain familiarity with the objects discussed in this paper.

Let $(X,K)$ be a Lorentzian surface with a Killing field $K$, which we assume to be complete. 

\begin{definition}[ribbons, bands and dominoes] \label{Définition ruban/domino}
	Let $U$ be a subset of $X$ saturated by $K$. Suppose $K$ never vanishes on $U$; we say that $(U,K)$ is \\
	(1) a ribbon if $U$ is open, simply connected  and if one of the null foliations in $U$ is everywhere transverse to $K$. \\
	(2) a band if $U$ is homeomorphic to $[0,1] \times \R$, with $\langle K,K \rangle$ vanishing on the boundary and not vanishing in the interior of $U$. \\
	(3) a domino if $U$ is open, simply connected, and $K$ has a unique null orbit in $U$.
\end{definition}
The closure of the connected components of $X \smallsetminus \{ \langle K , K \rangle =0\}$ are bands. 

\paragraph{Adapted basis.} 

Let $U$ be a ribbon in $X$, and denote by $\mathcal{L}$ the null foliation in $U$ everywhere transverse to $K$. \\

\textbf{Claim:} Any maximal geodesic $\gamma$ tangent to $\mathcal{L}$ (hence transverse to $K$) cuts all the leaves of $K$ in $U$. Besides, it cuts each leaf of $K$ only once. \\

To see this, choose a null-geodesic $\gamma$ that belongs to $\mathcal{L}$, maximal in $U$. The saturation of the geodesic by the flow of $K$ gives a connected open set $U_{\gamma}$. If $\alpha$ is another such geodesic, either $U_{\alpha}$ and $U_{\gamma}$ are equal or are disjoint open subsets of $U$. Define on $U$ the relation: for all $x, y \in U, x \sim y$ if and only if there exists a geodesic $\gamma$ such that $x, y \in U_{\gamma}$. It follows from what we said before that $\sim$ is an open equivalence relation on $U$. We then have $U=U_{\gamma}$, for $U$ is connected. Since $U$ is simply conected, $\gamma$ cuts every leaf of $K$ only once. If $\gamma$ satisfies $\langle \gamma^{'}, K \rangle =1$ (this scalar product is constant by Clairaut's lemma), we obtain coordinates on which the metric reads 
\begin{align}\label{Base adaptée, écriture locale de la métrique}
	2dxdy + f(x)dy^2,\;\; (x,y) \in I \times \R
\end{align}
where $L=\partial_x$ is a null vector field parameterized by $\langle L , K \rangle =1$, and $K=\partial_y$. We shall call an \textbf{adapted basis} in a ribbon a basis field $\{L,K\}$ like above.  
The coordinate denoted by $x$, which is well defined up to translation, will be called the "transverse coordinate", or simply the $x$-coordinate. Thus, the norm of $K$ in the ribbon is given by $f$ in the $x$-coordinate; it vanishes on the null orbits of $K$ contained in $U$.   \\

So we cannot have a global chart on $X$ as in (\ref{Base adaptée, écriture locale de la métrique}) if and only if none of the two null foliations is everywhere transverse to $K$, which amounts to saying that $K$ has two null orbits belonging to different null foliations. 
The latter fact is characterized in \cite{BM} by the presence of a Reeb component in the foliation defined by the Killing field or that of the orthogonal foliation. 

\begin{lemma}\cite[Lemma 2.8]{BM}\label{Lemme: types de bandes}
	Let  $(B,K)$ be a Lorentzian band. Denote by $\mathcal{L}$ the null foliation in $B$ transverse to one of the two connected components of $\partial B$. If $\mathcal{L}$ is everywhere transverse to $K$ in $B$, then the foliations defined by $K$ and $K^{\perp}$ are both suspensions (admit global transversals). Otherwise, one of them is a Reeb component and the other is a suspension. 
\end{lemma}

Thus, we have the following definition, resulting from the lemma:
\begin{definition}
	A Lorentzian band $(B,K)$ is said to be:\\
	(1) of type I: if the foliations defined by $K$ and $K^{\perp}$ are both suspensions. \\
	(2) of type II: if the foliation of $K$ is a suspension and that of $K^{\perp}$  is a Reeb component. \\
	(3) of type III: if the foliation of  $K$ is a Reeb component and that of $K^{\perp}$ is a suspension.
\end{definition} 
The following figures (2, 3 and 4) represent the three types of bands; the foliation of $K$ is represented in black, and the orthogonal foliation in red. In a type I band, the foliations by $K$ and $K^{\perp}$ are suspensions, i.e. admit a transversal curve cutting every leaf; this is equivalent to the fact that the space of leaves is Hausdorff. In a type II (resp. type III) band, the foliation defined by $K^{\perp}$ (resp. $K$) is a Reeb component, so  the space of leaves is non Hausdorff: it is a simple branching (we give a definition of a simple branching in Definition \ref{Définition: branchement d'ordre n}; see also Figures 5 and 6). 
\begin{figure}[h!] 
	\labellist 
	\small\hair 2pt 
	\pinlabel {$\text{Type I band}$} at 125 -12
	\pinlabel {$\text{Type II band}$} at 532 -12
	\pinlabel {$\text{Type III band}$} at 935 -12
	\endlabellist 
	\centering 
	\includegraphics[scale=0.24]{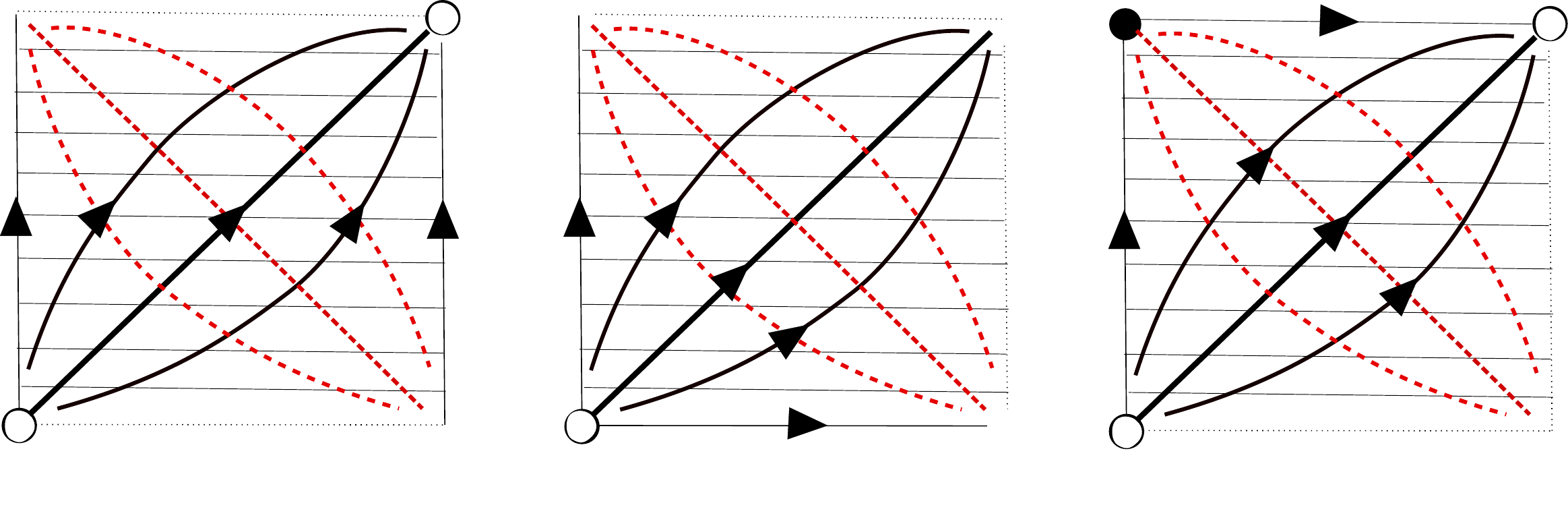} \caption{A type I, type II and type III band.} 
	\label{fig:cobo}
\end{figure} 

\begin{figure}[h!]\label{Figure: The CP torus has 4 bands of type II} 
	\labellist 
	\small\hair 2pt 
	\endlabellist  
	\centering 
	\includegraphics[scale=0.18]{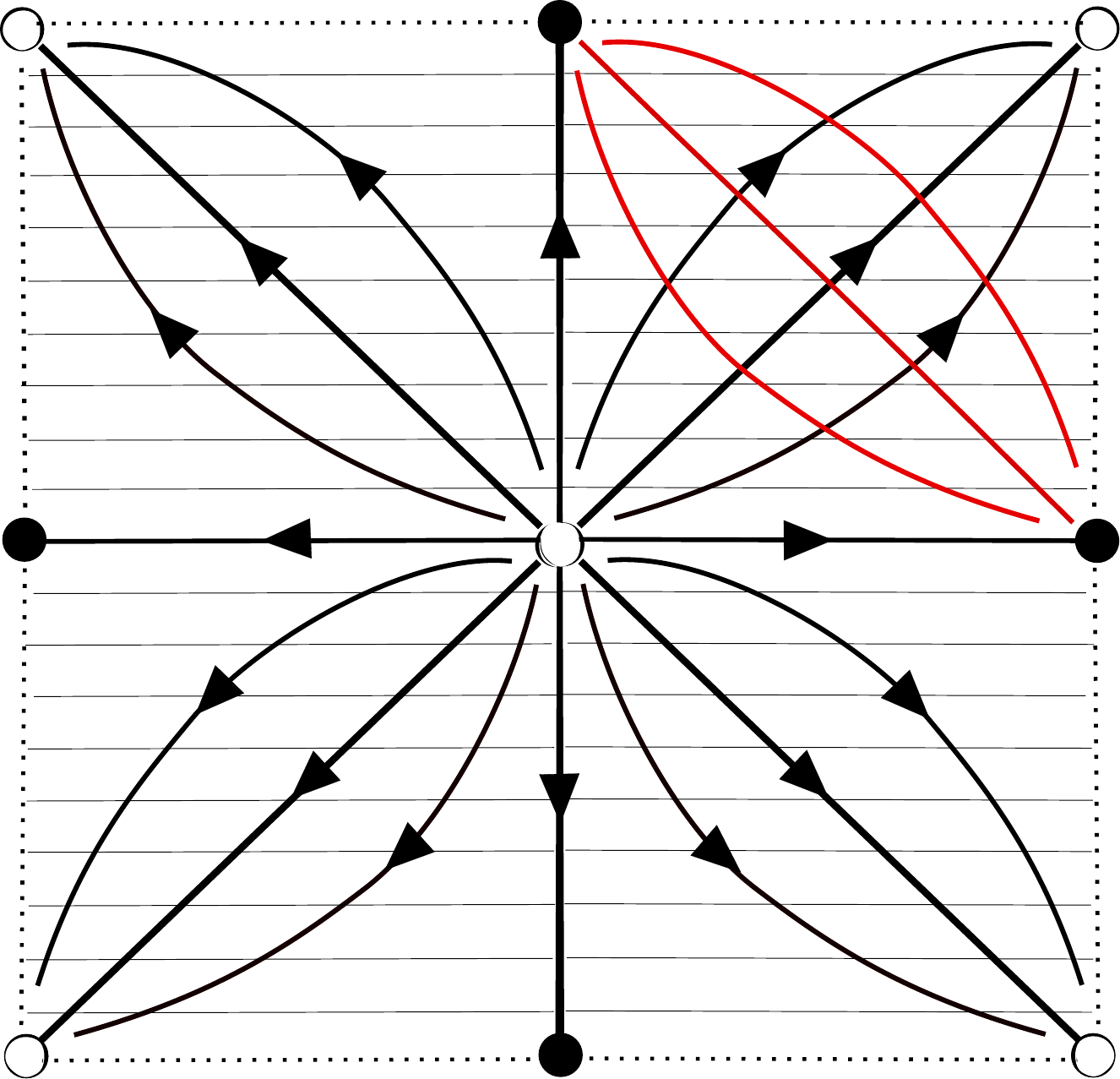} \caption{The Clifton-Pohl torus  has 4 bands of type II} 
	\label{fig:cobo}
\end{figure}

\begin{figure}[h!] 
	\labellist 
	\small\hair 2pt 
	\endlabellist  
	\centering 
	\includegraphics[scale=0.18]{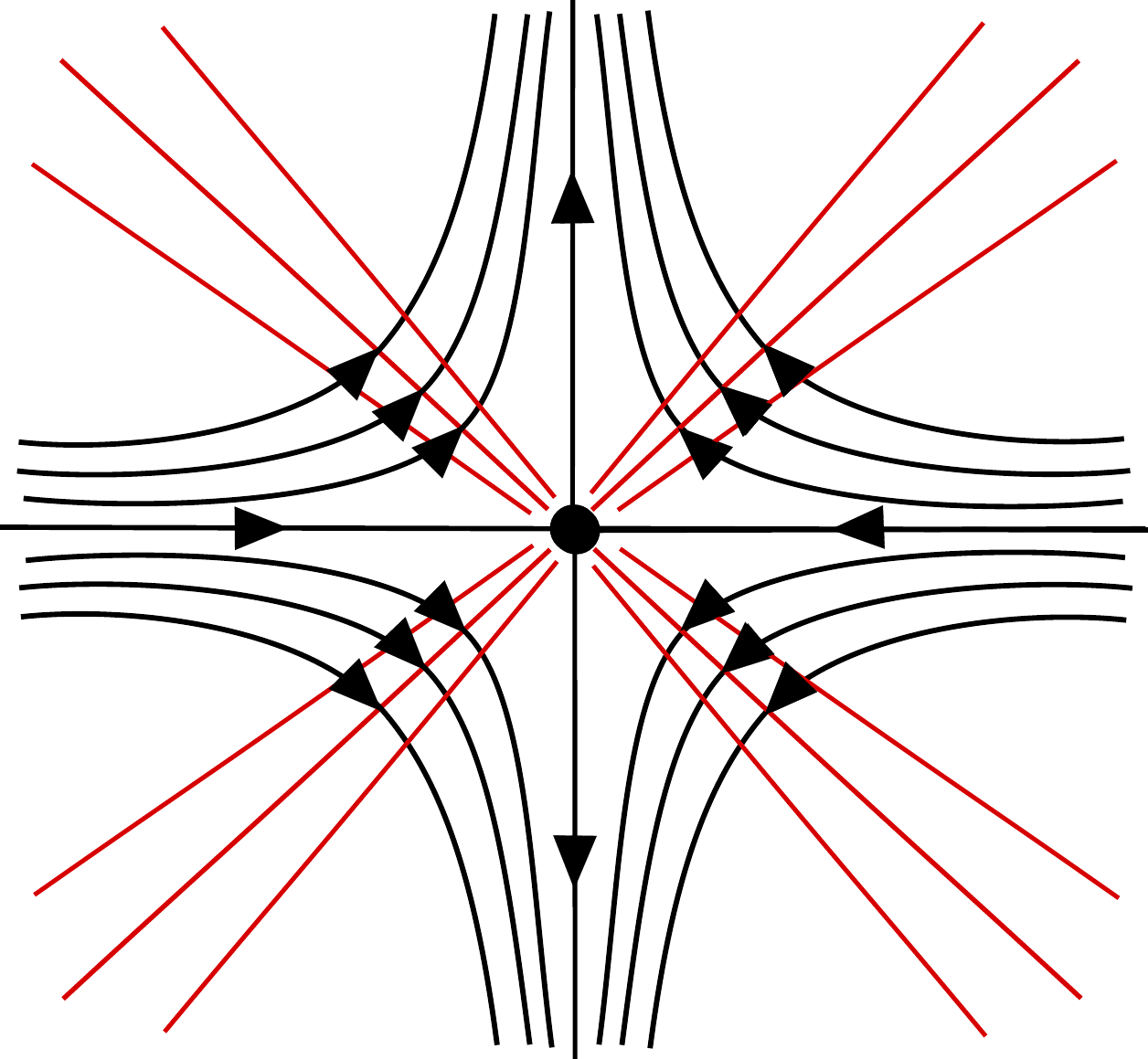} \caption{A saddle with the origin removed contains 4 bands of type III.} 
	\label{fig:cobo}
\end{figure}
\begin{proof}[Proof of Lemma \ref{Lemme: types de bandes}]
	\textbf{Claim 1:} $\mathcal{L}$ is everywhere transverse to $K$ in $B$ if and only if both of the foliations of $K$ and $K^{\perp}$ are suspensions. 
	
\noindent	By definition, $\langle K, K \rangle $ vanishes on the boundary of $B$, but not in the interior, so the boundary of $B$ is made of two null orbits of $K$. Denote by $\partial_1 B$ and $\partial_2 B$ the two connected components of the boundary of $B$. 
	Note that $K$ and $K^{\perp}$ have the same null leaves. Therefore, if $\mathcal{L}$ is everywhere transverse to $K$, then $\mathcal{L}$ is also transverse to $K^{\perp}$ so that the foliations of $K$ and $K^{\perp}$ are both suspensions. We say that $B$ is a type I band. The null leaves in the boundary of a type I band belong to the same null foliation. 
	
	Now, suppose that $\mathcal{L}$ is not globally transverse to $K$ in $B$, then the null leaves of the boundary belong to different null foliations.  Let $\alpha$ be a smooth curve connecting $\partial_1 B$ and $\partial_2 B$. If both of the foliations of $K$ and $K^{\perp}$ are suspensions, then we can take $\alpha$ to be everywhere transverse to both $K$ and $K^{\perp}$. Denote by $E_{\alpha}$ the line bundle over $\alpha$: it is a trivial bundle, so it is homeomorphic to a cylinder in which the null directions determine two disjoint open subsets: one containing the spacelike directions, and the other the timelike directions. The restrictions of $K$ and $K^{\perp}$ to $\alpha$ are smooth sections of $E_{\alpha}$ connecting one null direction of the lightcone to the other. When one is a spacelike section, the other is timelike, and they coincide on the boundary of $E_{\alpha}$. We see that any other continuous section of $E_{\alpha}$ must intersect one of the two sections defined by $K$ or $K^{\perp}$. Therefore, $\alpha$ is necessarily tangent to either $K$ or $K^{\perp}$ at some point of $B$. It follows that one of the two foliations by $K$ or $K^{\perp}$ has a Reeb component (\cite[Proposition 4.3.2]{MR881799}). \\
	
	\textbf{Claim 2:} If the foliation defined by $K$ or $K^{\perp}$ is not a suspension, then it is a Reeb component.
	
\noindent	Indeed, suppose that the foliation defined by $K$ (resp. $K^{\perp}$) has a Reeb component in $B$. A leaf of $\mathcal{L}$ (which is transverse to $\partial_1 B$) cuts all the leaves of $K$ (resp. $K^{\perp}$) in the interior of $B$ transversally. So the foliation of $K$ (resp. $K^{\perp}$) can only have one Reeb component. \\
	
	\textbf{Claim 3:} The two null foliations in $B$ are both suspensions. 
	
\noindent	This is due to the fact that a null geodesic transverse to $\partial_i B$ cuts all the leaves of the null foliation tangent to $\partial_i B$ transversally. \\
	
	\textbf{Claim 4:} If the foliation defined by $K$ (resp. $K^{\perp}$) is a Reeb component, then the orthogonal foliation is a suspension.
	
\noindent	Claim 3 implies the existence of a curve $\beta$ connecting $\partial_1 B$ and $\partial_2 B$, which is everywhere transverse to both the null foliations of $B$. In particular, $\beta$ is either timelike or spacelike. Suppose $K$ is a Reeb component. Then $\beta$ is tangent to $K$  at some point of the interior of $B$. In consequence, in the interior of $B$, $\beta$ has the same causality as $K$, so is everywhere transverse to $K^{\perp}$ in the interior of $B$. Since $\beta$ is also transverse to $K^{\perp}$ on the boundary of $B$, it follows that the foliation of $K^{\perp}$ is a suspension. The proof works exactly the same if we suppose that the foliation of $K^{\perp}$ is a Reeb component instead of that of $K$. 
\end{proof}

\paragraph{The existence of local reflections.}
The presence of a Killing field leads to the existence of local reflections. Let $ U $ be a connected component of $ X \smallsetminus \{\langle K, K \rangle = 0 \} $. In $ U $, the geodesic flow orthogonal to $ K $ can be parameterized by $ \langle K^{\perp}, K \rangle = 0 $, $ \langle K^{\perp}, K^{\perp} \rangle = \pm 1 $. The vector field thus obtained is invariant by $ K $, i.e. $ [K^{\perp}, K] = 0 $. This gives coordinates $ (u, v) \in I \times \R $ on $ U $, in which the metric reads $ \pm du^2 + h(u) dv^2 $, where $ K^{\perp } = \partial_u, K = \partial_v $. We observe that $ (u, v) \mapsto (u, -v) $ is an isometry of $ U $. This reflection sends $ K $ on $ -K $ and permutes the two null foliations. It is called a "\textit{generic reflection}".

So the reflection fixing a non-degenerate geodesic perpendicular to $K$ is an isometry, but unlike the Riemannian case, it is only defined locally when the norm of $ K $ vanishes (see Figure 3 above). In an adapted basis of the form (\ref{Base adaptée, écriture locale de la métrique}), the leaves of the foliation orthogonal to $K$ are defined by $y= G(x) + \beta$, where $G$ is a primitive function of $-1/f$ and $\beta \in \R$ \cite[Proposition 2.5]{BM}. And the generic reflections are given by $\phi(x,y) = (x, 2(G(x)+\beta)-y)$, $(x,y) \in I \times \R$.
\paragraph{An adapted atlas for $(X,K)$.}
If $U$ denotes a ribbon in $X$ and $p \in U$, we can choose a null-geodesic $\gamma$ passing through $p$, maximal in $U$ and transverse to $K$.  On the saturation of the geodesic by the flow of $K$, which is equal to $U$, the metric writes   $$2dxdy + f(x)dy^2,$$ where $L=\partial_x$ is a null vector field parameterized by $\langle L , K \rangle =1$, and $K=\partial_y$.
\begin{remark}(\cite[Lemma 2.25]{BM}).\label{Remarque sur la nature des zéros}
The geodesic parameterization of a null orbit of $K$ is incomplete if and
only if it corresponds to a simple zero of f .
\end{remark}
When $\langle K , K \rangle (p) \neq 0$, there exists another null-geodesic transverse to $K$ and passing through $p$, giving rise to another formula for the metric on an open set $U'$ of $X$. On the intersection $U \cap U'$, the norm of $K$ doesn't vanish: we have by  \cite[Proposition 2.5]{BM} the existence of a generic reflection, i.e. a local isometry fixing a non-degenerate geodesic perpendicular to $K$ and sending $K$ to $-K$. The transition map is given on $U \cap U'$ by composing $(x,y) \mapsto (-x,-y)$ with a generic reflection.  

This gives an atlas of $ X $ minus the zeros of $ K $,  such that on each local chart the metric reads $ 2dxdy + f (x) dy^2 $, with $ K =  \partial_y $, and transition maps are given by generic reflections. This atlas is called an \textbf{adapted atlas} for $(X,K)$. 
\paragraph{Riemannian structure on the space of Killing orbits.}
Let $(X,K)$ be a simply connected Lorentzian surface with Killing field $K$. Consider a positively oriented hyperbolic basis
$(U,V )$ in $\mathfrak{X}(X)$, the space of smooth vector fields on $X$, i.e. $\langle U,U \rangle=\langle V,V \rangle=0$ and $\langle U,V \rangle=1$. Define a volume form $\nu$ by setting $\nu(U,V)=1$; $\nu$ does not depend on the choice of this basis. Define a 1-form $\omega :=i_K \nu$: this form, which defines the foliation of $K$, is closed, hence exact since $X$ is simply connected. So there exists a submersion $\textbf{x} \in C^{\infty}(X,\R)$ (defined up to translation) such that $\omega=d\textbf{x}$, inducing a local diffeomorphism \footnote{\label{ulm}We use the french terminology and say that a manifold $\mathcal{E}$ is "étalé" in $\R$ if there is a local diffeomorphism from $\mathcal{E}$ into $\R$ (Definition \ref{Définition variété étalée}).} between the space of leaves of $K$ and $\R$.  This makes it into an orientable manifold of dimension $1$ (see \cite[Proposition 2.21]{BM}), on which we can define a Riemannian structure by taking the metric $d\textbf{x}^2$.
One can check that the submersion $\textbf{x}$ coincides, up to translation and change of sign, with the $x$-coordinate of any local chart. When $X$ is the universal cover of a non-flat torus, the local diffeomorphism between the space of Killing orbits and $\R$ is a global diffeomorphism, making this space into a Hausdorff manifold of dimension $1$.

\section{Maximal simply connected Lorentzian surfaces with a Killing field}

\subsection{Topological and differential properties of the space of leaves} \label{Section: Espace des feuilles}
Let $(X,K)$ be a simply connected Lorentzian surface admitting a non-trivial Killing field  $K$. 
Denote by $\mathcal{E}_X$ the space of leaves of the foliation defined by $K$. It is an orientable $1$-dimensional Riemannian manifold with a countable base, generally non-Haussdorf. 
In this paragraph, we examine the topology and the geometry of this manifold:
\begin{itemize}
	\item What are the possible branch points in $\mathcal{E}_X$ ? Do they satisfy some combinatorial property (existence of global constraints on the set of branch points) ?  
	\item Completeness of $\mathcal{E}_X$ for the induced pseudometric topology (introduced at the end of this paragraph) ?
\end{itemize}
In the elementary case, i.e. when the norm of $ K $ does not vanish, the space of leaves is homeomorphic to the real line; indeed, since $ X $ is homeomorphic to the plane, null cones determine two $1$-dimensional foliations all of whose leaves are homeomorphic to the line, and which are moreover everywhere transverse to $ K $ since the norm of the latter does not vanish. The space of leaves can then be parameterized by a maximal null geodesic. {\em In the following, $ X $ is assumed to be non-elementary}. 

\paragraph{The branch points of the space of leaves} A point $p$ on $\mathcal{E}_X$ is said to be a "branch point" if there is a point $q \neq p$ such that $p$ and $q$ cannot be separated by disjoint neighborhoods in $\mathcal{E}_X$. Denote by $ B $ the set of branch points, and $ \Sigma $ the set of connected components of the interior of $ \mathcal{E}_X \smallsetminus B $. 
\begin{definition}\label{Définition: point de branchement simple}
	A branch point $p$ of $\mathcal{E}_X$ is said to be a "simple branch point" if the set $V_p$ of points $q \neq p$ such that $p$ and $q$ cannot be separated, satisfies one of the two following properties:\\
	- $V_p$ contains only one point;\\
	- $V_p$ contains exactly two distinct points that can be separated.
\end{definition}
\begin{definition}\label{Définition: branchement d'ordre n}
	1) A simple branching is the union of a pair of simple branch points  	$(p_i)_{i \in \Z/2\Z}$ such that for $i \in \Z/2\Z$, $p_{i+1} \in V_{p_i}$, and the element $\sigma \in \Sigma$ such that $\forall i \in \Z/2\Z,\, p_i \in \bar{\sigma}$. We denote it by $(p_0,p_1,\sigma)$.\\
	2) A branching of order $n \in \N_{\geq 2}$ is a cycle of order $n$ of simple branchings $\{(p_i,p_{i+1},\sigma_i^{i+1}), i \in \Z/n\Z\}$, where the family $\{\sigma_i^{i+1}$, $i \in \Z/n\Z\}$ contains pairwise distinct elements. 
\end{definition}
\begin{figure}[h!] 
	\labellist 
	\small\hair 2pt 
	\pinlabel $p_1$ at 586 231 
	\pinlabel $p_2$ at 749 238
	\pinlabel $p_3$ at 675 126
	\endlabellist 
	\centering 
	\includegraphics[scale=0.25]{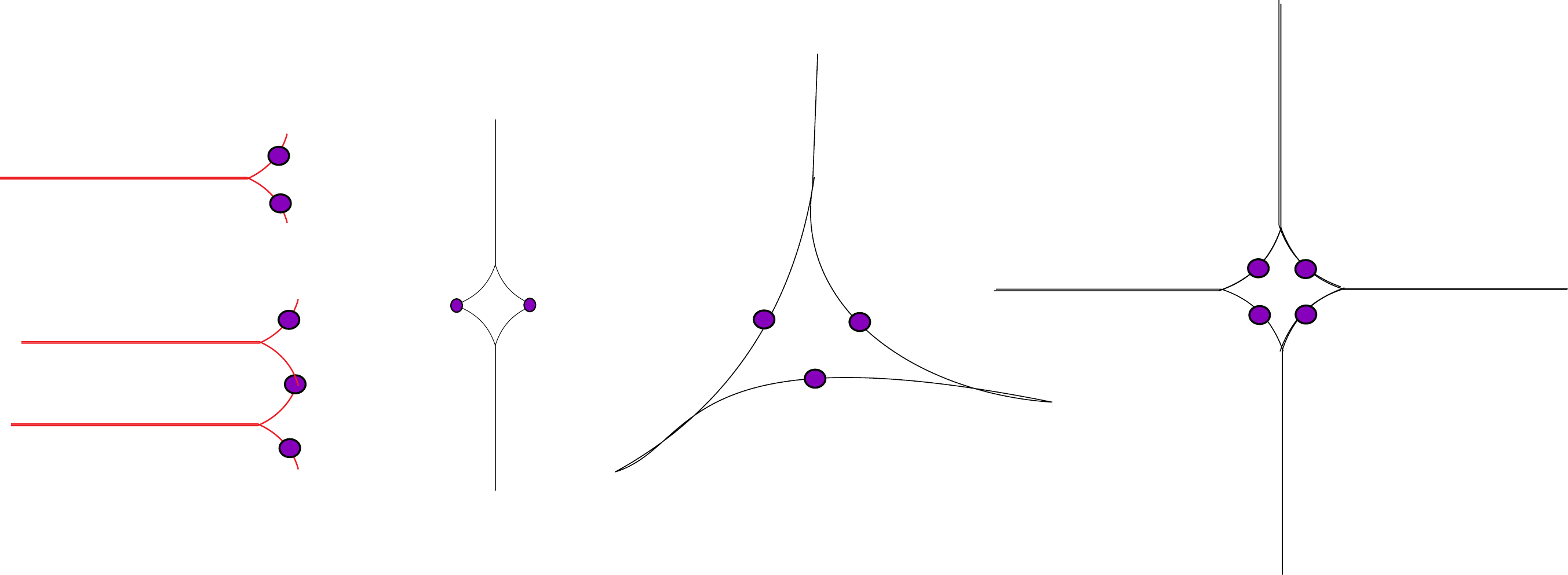} \caption{A simple branching; a chain of simple branchings;\; branchings of order $2$, $3$ and $4$} 
	\label{fig:cobo}
\end{figure} 

We also define an "{\em infinite (resp. finite) chain of simple branchings}", for which the simple branch points -in infinite (resp. finite) number- are indexed by $\Z$ (resp. a finite subset of $\Z$) instead of $\Z/n\Z$.\\

The possible branch points in the space of leaves are described in \cite[Paragraph 4.1]{BM}. First, we give an exposition of what is known in \cite{BM} about this manifold, before we can fully describe its topology. The branch points of $ \mathcal{E}_X $ correspond precisely to the null orbits of $ K $ bordering the squares of $ X $; in other words, to the boundaries of the connected components of $ \{\langle K, K \rangle \neq 0 \} $ in $ X $. 
\begin{definition}\label{Definition null band - null component}
	Let $(X,K)$ be a simply connected Lorentzian surface with Killing field $K$. A {\em null band} of $X$ is a connected component of the interior of the set $\{<K,K>=0\}$. The space of leaves of a null band will be referred to as a {\em null component} of $\mathcal{E}_X$. The null components have Hausdorff closure.
\end{definition}
\textbf{Notation:} Let $(X,K)$ be a simply connected Lorentzian surface with a Killing field $K$. We denote by $ \Sigma_1 $ (resp. $\Sigma_0$) the subset of $ \Sigma $ of elements with non Hausdorff (resp. Hausdorff) closure.\\   

We define the following property ($\mathcal{P}$) on $\mathcal{E}_X$:\\
If $\sigma \in \Sigma_1$, then $\bar{\sigma}$ is a closed interval with exactly one or two pairs of non-separated points in the boundary.\\

When $ X $ is a \textbf{maximal} surface, $\mathcal{E}_X$ satisfies the property ($\mathcal{P}$), and $\Sigma$ is the union of

\begin{itemize}
	\item intervals $I_i, i \in I$, such that $\{\langle K,K \rangle \neq 0 \} = \amalg_{i \in I} I_i$ and
	$$
	\overline{I_i}^{\mathcal{E}_X} = \left\{
	\begin{array}{ll}
	\includegraphics[scale=0.17]{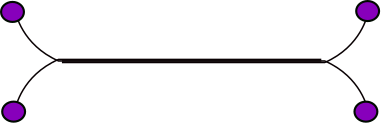} & \mbox{The space of leaves of a square} 
	\\
	\mbox{or} & \\
	\includegraphics[scale=0.17]{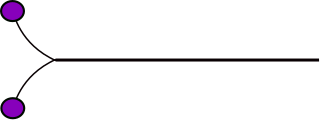} & \mbox{The space of leaves of a type III band near infinity}
	\end{array}
	\right.
	$$
	
	\item intervals $J_j, j \in J$, such that $\displaystyle \overset{\circ}{\overbrace{\{\langle K,K \rangle =0\}}}= \amalg_{j \in J} J_j$, and 
	$$
	\overline{J_j}^{\mathcal{E}_X} = \left\{
	\begin{array}{ll}
	\includegraphics[scale=0.18]{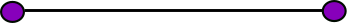} & \mbox{The closure of a null component} 
	\\
	\mbox{or} & \\
	\includegraphics[scale=0.18]{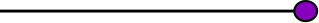} & \mbox{The closure of a null component near infinity}
	\end{array}
	\right.
	$$
\end{itemize}

\begin{figure}[h!] 
	\labellist 
	\small\hair 2pt 
	\endlabellist 
	\centering 
	\includegraphics[scale=0.22]{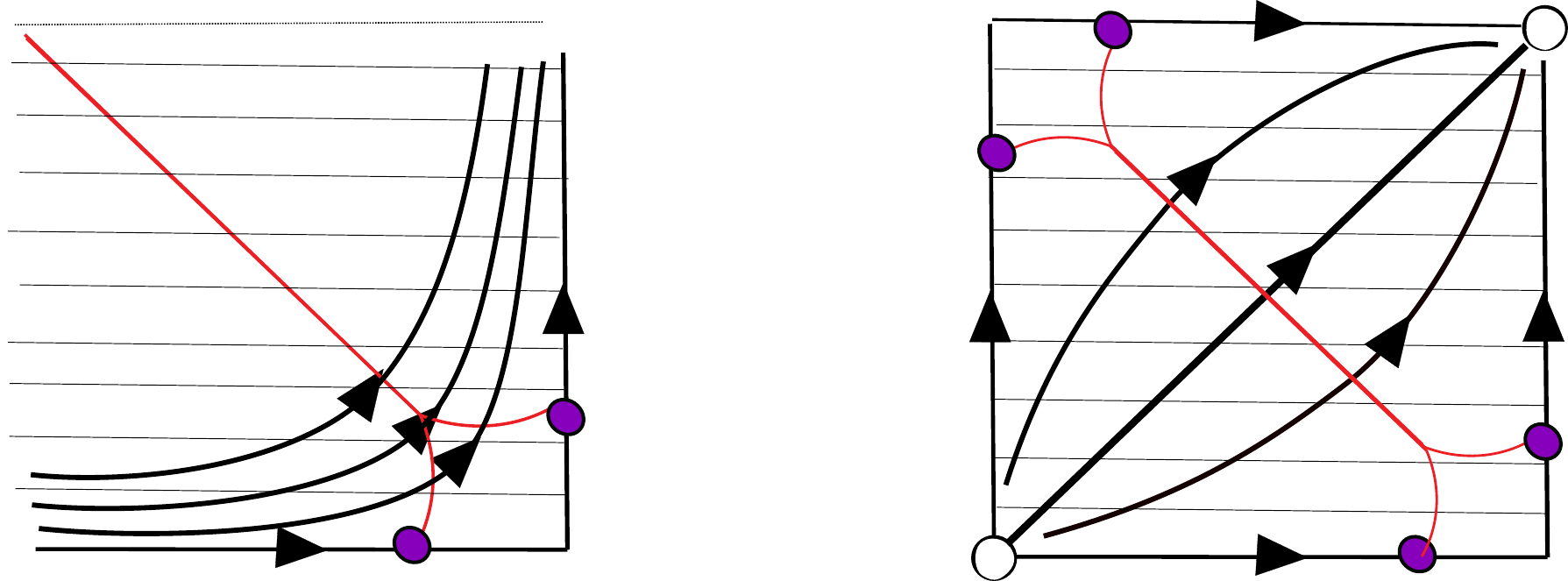} \caption{The space of leaves of a type III band (on the left) and a square (on the right)} 
	\label{fig:cobo}
\end{figure} 

This is equivalent to saying that when $ X $ is maximal, and if $ U $ is a connected component of $ X \smallsetminus \{\langle K, K \rangle = 0 \} $, then the closure of $ U $ in $ X $ minus the zeros of $ K $ is either a Lorentzian square or a type III band; this is proved in \cite[Proposition 2.20]{BM}. In particular, when $X$ is maximal, an element of $\Sigma_0$ is the space of leaves of a null band of $X$, and an element of $\Sigma_1$ is the space of leaves of a connected component of $ X \smallsetminus \{\langle K, K \rangle = 0 \} $. Note that if $ X $ is not maximal, an element of $\Sigma_0$ is the space of leaves of a type I band or a type II band, half a band, or a null band of $X$.

These segments can accumulate, and the set of branch points $B$ is not closed in general. 

Consider a null orbit of $ K $ which is the border between two adjacent squares in $X$. Denote by $ f $ the function induced by the norm of $K$ in the ribbon containing this orbit. The following figure shows how the segments corresponding to squares are connected to each other in the space of leaves, depending on whether $ f $ changes sign or not.

\begin{figure}[h!]
	\labellist 
	\small\hair 2pt 
	\pinlabel {$f \text{\;\;changes sign}$} at 860 350 
	\pinlabel {$f \text{\;\;does not change sign}$} at 900 21 
	\endlabellist 
	\centering 
	\includegraphics[scale=0.21]{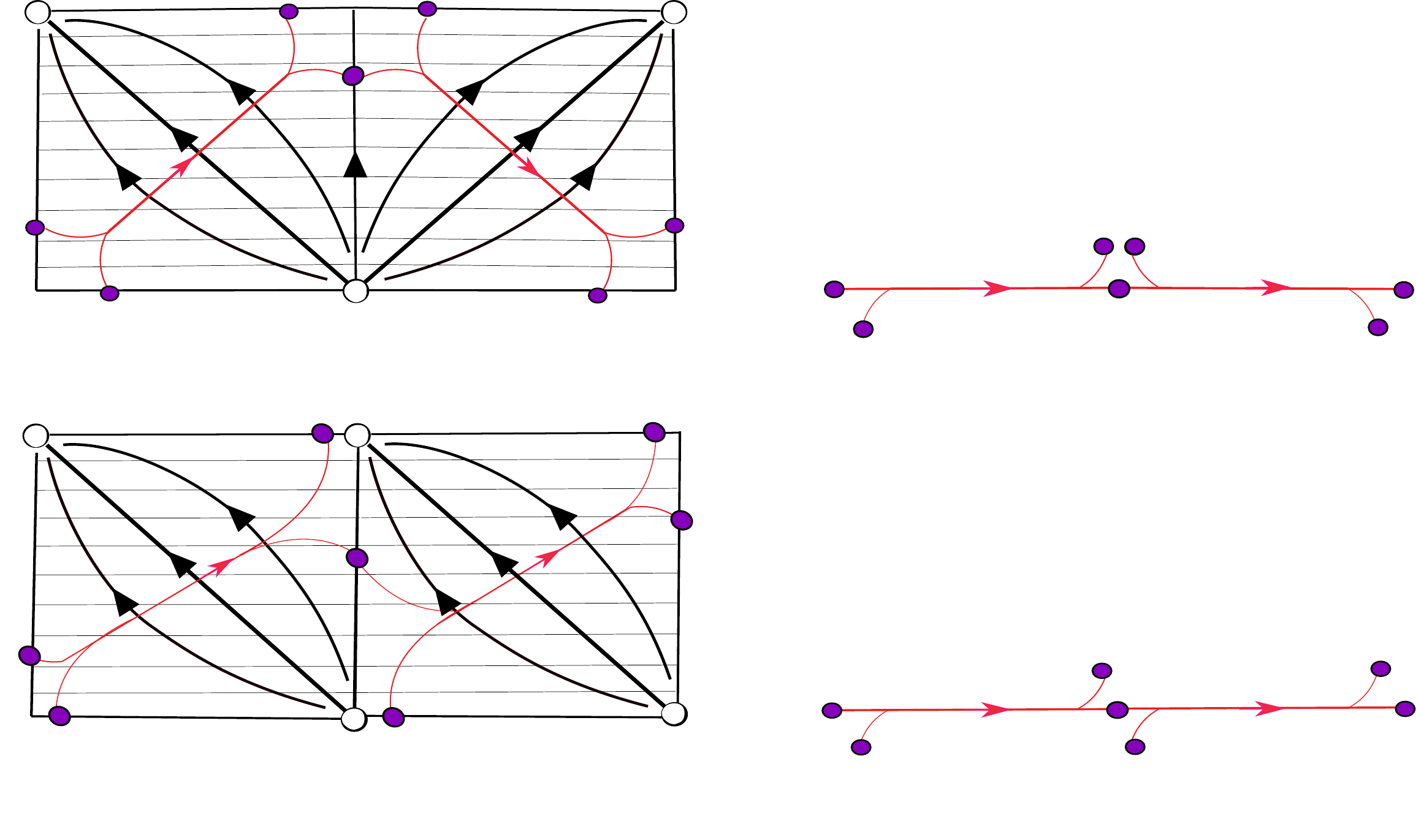} \caption{Connecting two squares in the space of leaves} 
	\label{fig:cobo}
\end{figure}\label{figure: raccordement}
\begin{figure}[h!] 
	\labellist 
	\small\hair 2pt 
	\endlabellist 
	\centering 
	\includegraphics[scale=0.24]{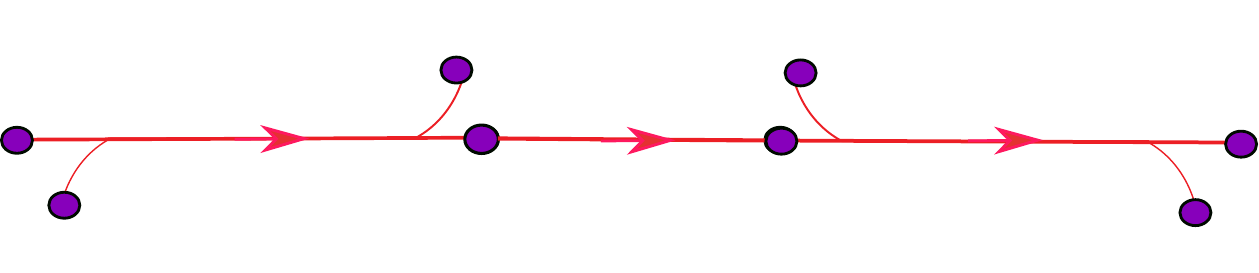} \caption{Presence of a null-component in the space of leaves of a maximal surface} 
	\label{fig:cobo}
\end{figure} 
\newpage
From a topological point of view, the spaces of leaves represented in Figure 7 are the same. The only consideration of the space of leaves is therefore not sufficient to characterize the foliation of $K$. For example, to a manifold with $ 3 $ simple branch points, one can associate two topological conjugation classes of non-oriented foliations of the plane (see \cite{MR0336755}; in this paper, Godbillon gives a topological classification of foliations of the plane whose space of leaves have a finite number of branch points).
 
\begin{figure}[h!] 
	\labellist 
	\small\hair 2pt 
	\endlabellist 
	\centering 
	\includegraphics[scale=0.31]{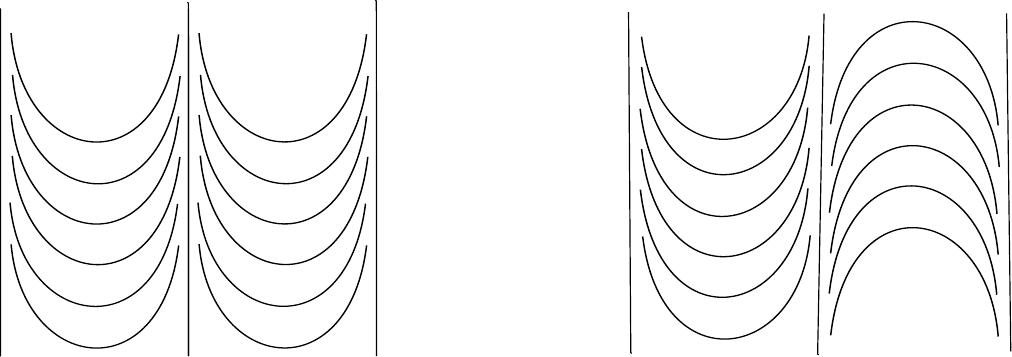} \caption{Two topological conjugation classes of foliations of the plane associated to a manifold with $3$ simple branch points}
\end{figure} 

\begin{definition}\label{Définition variété étalée}
	A smooth manifold $ \mathcal{E} $ of dimension $ 1 $ is said to be "étalé in $ \R $" \, if there is a local diffeomorphism of $ \mathcal{E} $ into $ \R $.	 
\end{definition}

Given an oriented manifold of  dimension $ 1 $ étalé in $ \R $, we distinguish in the following definition two types of non-separation for the branch points. 
\begin{definition}
	Let $ \mathcal{E} $ be a manifold of dimension $ 1 $ étalé in $ \R $ by a local diffeomorphism $ \textbf{x} \in C^{\infty}(\mathcal{E}, \R) $. Let $ p $ and $ q $ be two non-separated points of $ \mathcal{E} $. We say that $p$ and $q$ are not right-separated (resp. left-separated) and we write $ p \, \mathcal{R}_r \, q $ (resp. $p\, \mathcal{R}_l\, q$) if, given two neighborhoods $ V_p $ and $ V_q $ of $ p $ and $ q $ respectively, $\textbf{x}(V_p \cap V_q)$ is on the right (resp. on the left) of $x_0:=\textbf{x}(p)=\textbf{x}(q)$. 
\end{definition}
It is not difficult to see that $ \mathcal{R}_r $ and $ \mathcal{R}_l $ are equivalence relations on $B$. For every $ p \in B$, denote by $ [p]_r $ (resp. $ [p]_l $) the equivalence class of $ p $ for the relation $ \mathcal{R}_r $ (resp. $ \mathcal{R}_l $).

Let $ (X, K) $ be a simply connected non-elementary Lorentzian surface with a non-trivial Killing field $ K $. The space of leaves $ \mathcal{E}_X $ is a (connected)  $1$-dimensional manifold étalé in $\R$, whose non-separated points are simple branch points. Therefore, the equivalence classes defined above have cardinality $ 1 $ or $ 2 $.

A finite maximal chain of simple branchings in $\mathcal{E}_X$ contains two points $p_1$ and $p_2$ such that  $\# [p_i]_r=1$ or $\# [p_i]_l=1$, $i \in \{1,2\}$. These points will be called the boundaries of the finite chain. 
\begin{fact}\label{Remarque: point séparé à droite --> non isolé}
	If $p$ is a branch point of $\mathcal{E}_X$ such that $\# [p]_r = \# [p]_l=2$, then the corresponding null orbit of $K$  is isolated in the closed set $\{\langle K,K \rangle =0\}$, for it borders two bands in $X$; the converse is true if $ X $ is assumed to be maximal. 
\end{fact}
\begin{proof}
	Let $(X,K)$ be a maximal Lorentzian surface with a Killing field $K$. Consider $ U $ a Lorentzian domino contained in $ X $. In an adapted basis, the metric in $U$ reads $2dx dy + f(x) dy^2$, with $f \in C^{\infty}(I,\R)$. Translating $x$ if necessary, set $ x_0 = 0 $ in the unique zero of $ f $ in $ I $. Fix  $U_0 = ]-m,0[\times \R$ a component of $U\smallsetminus \{\langle K,K \rangle=0\}$, where $ m \in \R$ is the length of the segment $\mathcal{E}_{U_0}$.
	We have another local chart in $U_0$ induced by the other null foliation, where the metric reads $2 dx' dy' + \hat{f}(x') dy'^2$, with $x'=-x$ and $\hat{f}(x')=f(x)$ for $x'\in ]0,m[$. It is clear that $ \hat{f} $ extends into a function $ g $ defined over an interval $]m',m[$, $m'<0$, since $ \hat{f} $ defines the same germ at $  0 $ as $ f $ at $  0 $.
	We can then define a proper extension of $ U $ by gluing a ribbon $ R_g $ along $ U_0 $ using a reflection of $ U_0 $. Thus, if $ X $ contains a domino $ U $, it also contains, by maximality, an extension of $ U $ of the previous form. Denote by $p$ the branch point in $\mathcal{E}_X$ corresponding to the (isolated) null orbit of $K$ in $U$, this yields $\# [p]_r = \# [p]_l=2$ (see the Figure 10 below).
	\begin{figure}[h!] 
	\labellist 
	\small\hair 2pt 
	\endlabellist 
	\centering 
	\includegraphics[scale=0.15]{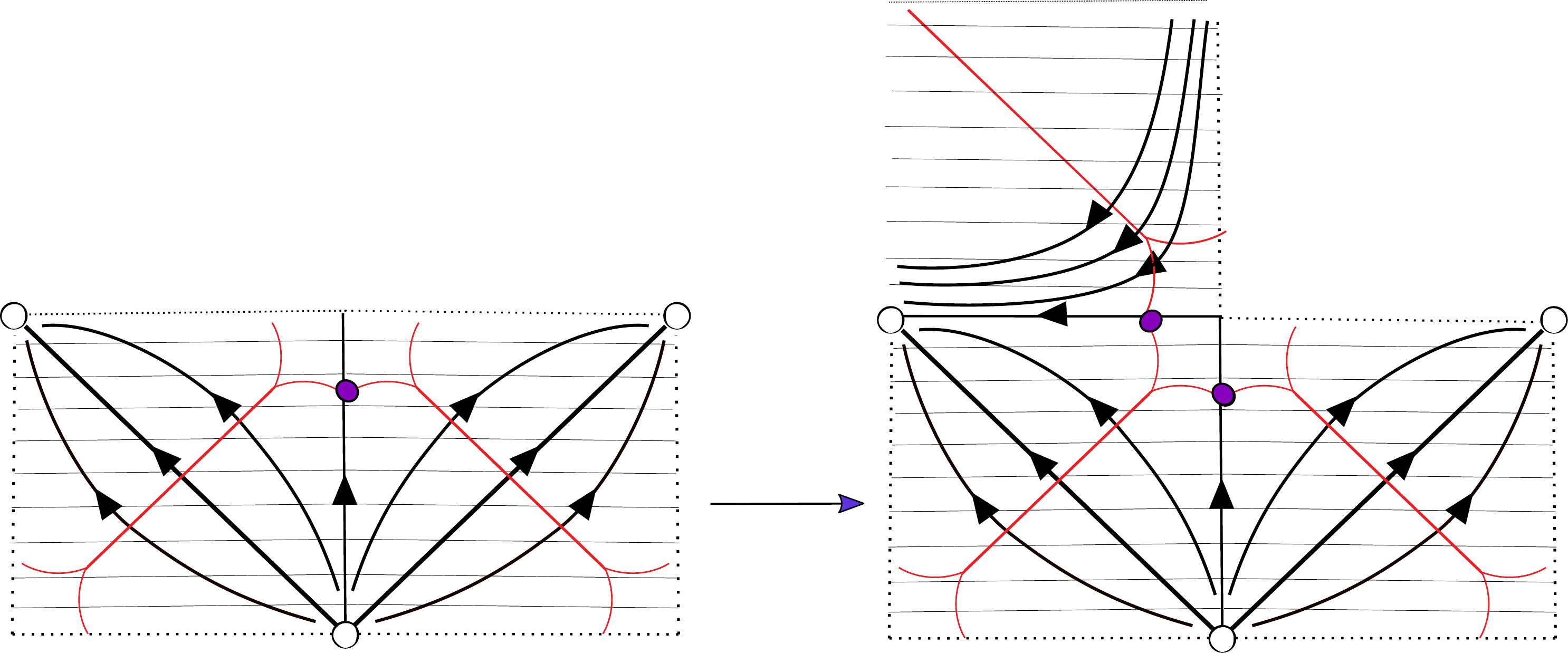} \caption{Extension of a domino (on the left) by means of a local reflection. So when $X$ is maximal, the branch point $p$ corresponding to the null orbit of $K$ in the domino satisfies $\# [p]_r = \# [p]_l=2$.}
\end{figure}	
\end{proof}

If $X$ is maximal, and if $ p \in B$ satisfies $ \# [p]_r = 1 $ or $ \# [p]_l = 1 $, then the corresponding null orbit $ l_p $ in $ X $ is not isolated in the closed set $\{\langle K, K \rangle = 0 \}$. So if the norm of $K$ is identically zero from one side of $l_p$, then $ p $ is in the boundary of a null component of $ \mathcal{E}_X $. If $ \# [p]_r = 1 $ (resp. $ \# [p]_l = 1 $) and if the norm of $K$ is not identically zero on the right (resp. on the left) of $ p $, then in the ribbon containing $ l_p $ there are bands which accumulate on $ l_p $, so that $ p $ is a limit point of $B$. In particular, we have the following fact: 
\begin{fact}\label{Fact le bord d'une chaine finie se trouve dans une null component}
If $X$ is maximal and $B$ is locally finite, and $ p \in B$ satisfies $ \# [p]_r = 1 $ or $ \# [p]_l = 1 $, then $p$ is in the boundary of a null component.
\end{fact}

\paragraph{The simply connected metric space associated to the space of leaves}

Le us first introduce some vocabulary related to non-Hausdorff spaces. 
\begin{definition}
	A \textbf{pseudometric space} $(E,\delta)$ is a set $E$ together with a non-negative real-valued function $\delta : E \times E \to \R^{+}$ (called a pseudometric) such that for every $x, y, z \in E $,
	\begin{itemize}[itemsep=0pt, topsep=0pt, parsep=0pt]
		\item $\delta (x,x) =0$,
		\item $\delta(x,y)=\delta(y,x)$ (symmetry),
		\item $\delta (x,z) \leq \delta (x,y) + \delta(y,z) $ (triangle inequality).
	\end{itemize}
\end{definition} 

\noindent In a pseudometric space, one may have $\delta(x,y)=0$ for $x \neq y$. \\

The \textbf{pseudometric topology} is the topology generated by the open balls ${\displaystyle B_{r}(p)=\{x\in E \mid \delta(p,x)<r\},}$ which form a basis for the topology.

The definition of Cauchy sequences for metric spaces carry over to pseudometric spaces in the same way.
\begin{definition}
	A pseudometric space in which every Cauchy sequence has a limit is a complete pseudometric space.	
\end{definition} 
Given a pseudometric space $(E,\delta)$, one may define an equivalence relation on $E$ by setting for all $x,y \in E, x \sim y $ if $\delta(x,y)=0$. The topological space $E/\sim$ is the \textbf{Kolmogorov quotient}. This space can be made into a metric space by setting $d(\pi(x),\pi(y)) :=  \delta(x,y)$, for all $x, y \in E$, where $\pi: E \to E/\sim$ is the quotient map.  $(E/\sim,d)$ is called the metric space induced by the pseudometric space $( E , \delta )$; the quotient topology on $E/\sim$ coincides with the metric topology. \\ 

Let $(X,K)$ be a simply connected Lorentzian surface with a Killing field $K$. We can define the following binary relation on $ \mathcal{E}_X $: for $ p, q \in  \mathcal{E}_X, p \mathcal{R} q $ if and only if $ p $ and $ q $ cannot be separated. Obviously, $ \mathcal{R} $ is not an equivalence relation for it is not transitive; we then define $ \widetilde{\mathcal{R}} $ as the equivalence relation generated by $ \mathcal{R} $.

The Riemannian structure on the space of leaves $\mathcal{E}_X$ induces a pseudometric on this space that we denote by $\delta$. It is defined  in the following way: for two points $p$ and $q$ on $ \mathcal{E}_X$, $\delta(p,q)$ is the minimum length among the curves joining these points. What we get is not a metric but a pseudometric: it satisfies the axiom of symmetry and the triangle inequality, but not the axiom of separation. Denote by $\mathcal{T}_1$ the quotient topology on $\mathcal{E}_X$ (regarded as the quotient of the topological space $X$) and $\mathcal{T}_2$ the pseudometric topology. $\mathcal{T}_1$ is a strictly finer topology than $\mathcal{T}_2$: for non Hausdorff $\mathcal{E}_X$, $\mathcal{T}_2$ is not Kolmogorov.  The Kolmogorov quotient of $(\mathcal{E}_X, \mathcal{T}_2)$ is a metric space (in particular Hausdorff), and the equivalence relation $p  \sim q$ if and only if $\delta(x,y)=0$ coincides with $  \widetilde{\mathcal{R}} $. 


\begin{definition}
	We define $ \mathcal{G}(\mathcal{E}_X): = \mathcal{E}_X / \widetilde{\mathcal{R}} $ to be the Kolmogorov quotient of $(\mathcal{E}_X, \mathcal{T}_2)$. 
\end{definition}
Let $(X,K)$ be a simply connected Lorentzian surface with a Killing field $K$. If there is no zero of $ K $,  then $K$ induces a non singular foliation of the plane, and $ \mathcal{E}_X $ is simply connected by \cite[Proposition 1 p. 121]{MR0089412}. If $p \in X$ is a zero of $K$, then $p$ is a saddle point of $K$ connecting $4$  null orbits of $K$ (see the proof of  \cite[Proposition 2.20]{BM}), so that $\mathcal{E}_X$ is a $4$ order branching in a neighborhood of $p$. In this case, $\mathcal{E}_X$ is not simply connected (see \cite[Definition 1 p. 112]{MR0089412} for the definition of simple connectedness in the case of non-Hausdorff manifolds). \\

In the sequel, a simple cycle of geodesics of $\mathcal{E}_X$ is a sequence of geodesics  $\gamma_0, \gamma_1, \gamma_2, . . . ,\gamma_n = \gamma_0$, $n \geq 3$, such that $\gamma_i \cap \gamma_j \neq \emptyset$ if and only if $|i-j|=1$, with $ i,j \in \Z/n\Z$. 
\begin{proposition}\label{Caractérisation topologique de la variété des feuilles}
	Let $ (X, K) $ be a simply connected Lorentzian surface with a non-trivial Killing field $ K $. Then $ \mathcal{G}(\mathcal{E}_X)$ is a real tree, and $ \mathcal{G}(\mathcal{E}_X)$, with the quotient topology of $\mathcal{T}_1$, is Hausdorff. 
\end{proposition}
\begin{proof}
	 By definition, $\mathcal{G}(\mathcal{E}_X)$ is the metric space induced by the pseudometric space $(\mathcal{E}_X, \delta)$. So all we have to prove is that given any two points of it, there is a unique (injective) path joining them.  Let $\alpha: [0,1] \to \mathcal{E}_X$ be a $\mathcal{T}_2$-continuous injective path such that $\alpha(0) \sim \alpha(1)$. Since $ \mathcal{E}_X$ is path connected, there is a $\mathcal{T}_2$-continuous path $r$ joining $\alpha(0)$ and $\alpha(1)$. The union $\alpha \cup r$ can be covered by a finite number of geodesics, such that every geodesic is the space of leaves of a ribbon in $X$. This defines a simple cycle of geodesics of $\mathcal{E}_X$ corresponding to ribbons, and such a cycle is necessarily a  branching of order $ 4 $ (this is a consequence of \cite[Lemma 3.13]{BM}). This proves that $\alpha(t) \sim \alpha (0), \forall t \in [0,1]$, implying that $\mathcal{G}(\mathcal{E}_X)$ is a real tree. The topology on $\mathcal{G}(\mathcal{E}_X)$ induced by $\mathcal{T}_1$ is finer than the metric topology (and it may be strictly finer than the latter), so $\mathcal{G}(\mathcal{E}_X)$ is again Hausdorff for the quotient topology of $\mathcal{T}_1$. 
\end{proof}
From all the previous study, we can write the following conclusion: let $(X,K)$ be a simply connected Lorentzian surface with a non-trivial Killing vector field $K$. Then:
\begin{itemize}
	\item $\mathcal{E}_X$ is a (connected) étalé manifold of dimension $1$ with a countable base (usually non Hausdorff), whose non-separated points are simple branch points,
	\item Finite order branchings of $\mathcal{E}_X$ are of order $4$,
	\item $\mathcal{G}(\mathcal{E}_X)$ is a real tree,
	\item If $X$ is besides maximal, then $\mathcal{E}_X$ satisfies the property ($\mathcal{P}$).
\end{itemize}
\begin{remark}\label{Remarque: condition to be a real tree}
	Let $ \mathcal{E} $ be a manifold of dimension  $ 1 $, étalé in $ \R $, such that the branch points of $\mathcal{E}$ are simple branch points. Then $\mathcal{G}(\mathcal{E}_X)$ is a real tree if and only if the only simple cycles of geodesics of $\mathcal{E}$ are  branchings of order $n$. 
\end{remark}
\begin{remark}
$(\mathcal{E}_X, \mathcal{T}_2)$ is in general not locally compact. If $B$ is $\mathcal{T}_2$-locally finite, then $\mathcal{E}_X$ is $\mathcal{T}_2$-locally compact. The converse may not be true. 

\end{remark}

\begin{remark}\label{Remarque sur les zéros de f _structure différentiable}
	If $ (X, K) $ is a Lorentzian surface with Killing field $ K $, the function $ F \in C^{\infty}(\mathcal{E}_X, \R) $ induced by the norm of $K$ vanishes at all the branch points of $ \mathcal{E}_X $, and satisfies the following property: if $ F $ is of rank $ 1 $ (resp. of rank $ 0 $) at a point $ p $ of $ \mathcal{E}_X $, then it is also of rank $ 1 $ (resp. of rank $ 0 $) at every point $ q $ such that $\delta(p,q)=0$.  This is a consequence of the fact that there is a local isometry (the local reflections) in every connected component $U$ of $\{\langle K,K \rangle \neq 0\}$ that permutes  branch points $p$ and $q$ of $\overline{\mathcal{E}_U}$ satisfying $\delta(p,q)=0$. So on a fixed simple branching, the  null orbits of $K$ corresponding to the branch points are of the same nature (complete or not),  by Remark \ref{Remarque sur la nature des zéros}.
\end{remark}
\begin{remark}
	In \cite{MR0089412} p. 115, the authors define a differentiable structure on the simple branching so that any $ C^{\infty} $ function  defined on it is of rank $ 0 $ at the origin. They also give examples of manifolds of dimension $ 1 $ admitting a differentiable structure of class $ C^{\infty} $ such that all the differentiable functions on these manifolds are reduced to constants. Since  there is a function $ \textbf{x}: \mathcal{E}_X \to \R $ which is a local diffeomorphism, thus of rank $ 1 $ everywhere, then  these pathological properties are discarded in the case of $ \mathcal{E}_X $ manifolds; moreover, if $ H \in C^{\infty}(\mathcal{E}_X, \R) $ then $H$ satisfies the property given  in Remark \ref{Remarque sur les zéros de f _structure différentiable}.
\end{remark}
\subsection{Distinguished geodesics and linking structures}

In the sequel,  we say that  $\phi: (X_1,K_1) \rightarrow (X_2,K_2)$ is an isometry between two Lorentzian surfaces $(X_1,K_1)$ and $(X_2,K_2)$ admitting a Killing field,  if $\phi$ is an isometry sending $K_1$ to $K_2$.

We know that in any orientable Lorentzian surface, there exists two null-foliations. Let $(X,K)$ be a simply connected Lorentzian surface with a Killing field $K$; the two null-foliations define two disjoint families of ribbons covering $X$, such that two any ribbons in the same family are disjoint, and two ribbons belonging to different families are either disjoint or intersect in an open band. 
 
Recall that $ \mathcal{E}_X $ is a smooth manifold étalé in $ \R $ by a local diffeomorphism $ \textbf{x} \in C^{\infty} (\mathcal{E}_X, \R) $; in particular, $ \mathcal{E}_X $ is orientable, oriented by $ d \textbf{x} $. 
\begin{definition}\label{Définition géodésique admissible de X}
	A geodesic of $ \mathcal{E}_X $ is called "\textbf{distinguished}" if it is an oriented maximal geodesic which is the space of leaves of a ribbon in $X$. 
\end{definition}
This gives two disjoint families of distinguished geodesics of $\mathcal{E}_X$, such that the intersection of two of them is either empty or a connected component of the interior of $\mathcal{E}_X \smallsetminus B$.\\

The orientation of $ K $ induces an order on the pairs of points $ \{p_1, p_2 \} $, where $ p_1 $ and $ p_2 $ are two non-separated points of $ \mathcal{E}_X $. We will denote by $ (p_1, p_2) $, $ p_1 <p_2 $, the ordered pair with origin $ p_1 $ and end point $ p_2 $. Let $ \gamma $ be a distinguished geodesic, and let $ \sigma \in \Sigma $ contained in $ \gamma $, whose closure in $ \mathcal{E}_X $ is a non-Hausdorff segment whose boundary is given by two ordered pairs $ (p_1, p_2) $ and $ (q_1, q_2) $. 
The geodesic is characterized by the fact that it passes through $ p_1 $ and $ q_2 $ (or $ p_2 $ and $ q_1 $), for any $\sigma \in \Sigma $ contained in $ \gamma $, whose closure in $ \mathcal{E}_X $ is as before.
\paragraph{Distinguished geodesics of an étalé manifold} 
\begin{definition}\label{Définition topologie T}
	Let $ \mathcal{E} $ be a manifold of dimension  $ 1 $ with a countable base, étalé in $ \R $. We say that $ \mathcal{E} $ has topology $ \mathfrak{T} $ if it satisfies in addition the following properties:
	
	(1) Non-separated points are simple branch points.
	
	(2) Finite order branchings are of order $ 4 $.
	
	(3) $\mathcal{E}$ satisfies the property ($ \mathcal{P} $).
	
	(4) $ \mathcal{G}(\mathcal{E}) $ is a real tree.
\end{definition}
\textbf{Notation:} We define $\Sigma_2$ as the subset of $\Sigma_1$, such that if $\sigma \in \Sigma_2$, then $\bar{\sigma}$ is a non-Hausdorff segment with two pairs of non-separated points in the boundary.

Let $ \mathcal{E} $ be a smooth manifold with a topology $ \mathfrak{T} $. For any element $ \sigma \in \Sigma_2 $, giving only the topological manifold $ \mathcal{E} $ does not allow to distinguish the non-separated points in the boundary of $\sigma$, and therefore to characterize distinguished geodesics  by the way in which they cross these points. We then need an additional data on $ \mathcal{E} $:
\begin{definition}\label{Définition linking structure}
	Let $\mathcal{E}$ be a smooth manifold with a topology $\mathfrak{T}$. A \textbf{linking structure} $\mathcal{A}$ on $\mathcal{E}$ is the giving of two families $\mathcal{F}_1$ and $\mathcal{F}_2$ of pairwise disjoint maximal geodesics satisfying the following conditions:
	\begin{enumerate}[itemsep=0pt, topsep=0pt, parsep=0pt]
		\item $\mathcal{F}_1 \cup \mathcal{F}_2 = \mathcal{E}$ and $\mathcal{F}_1 \cap \mathcal{F}_2 = \emptyset$,
		\item $\forall \gamma_1 \in \mathcal{F}_1, \forall \gamma_2 \in \mathcal{F}_2, \gamma_1 \cap \gamma_2 = \emptyset \textrm{\;\;or\;\;} \gamma_1 \cap \gamma_2 \in \Sigma_1$.
	\end{enumerate}
	The second condition will be referred to as the \textit{transversality condition}.
	
	A maximal geodesic of $ \mathcal{E} $ is said to be \textbf{distinguished} \textbf{relative to $\mathcal{A}$} if it belongs to one of the two families. 
\end{definition}

\textbf{Notation:} We denote by $G_{\mathcal{A}}$ the set of all distinguished geodesics of $\mathcal{E}$ relative to $\mathcal{A}$.

Note that since  $ \mathcal{E}$ is étalé in $\R$, a distinguished geodesic  is diffeomorphic to $\R$.\\

In what follows, we give a combinatorial way to defining a linking structure on $\mathcal{E}$. Take $ \sigma \in \Sigma_2 $, and denote by $ \partial_1 \sigma $ (resp. $ \partial_2 \sigma $) the left boundary (resp. right boundary) of $ \sigma $ composed of two non-separated points. If $\mathcal{E} \simeq \mathcal{E}_X$, we know that there are two distinguished geodesics of $ \mathcal{E}_X $ containing $ \sigma $, and they define a bijection $ \mu_{\sigma}: \partial_1 \sigma \rightarrow \partial_2 \sigma $. For $ \sigma \in \Sigma_2 $, there are two possible bijections from $ \partial_1 \sigma = \{a, a '\} $ to $ \partial_2 \sigma = \{b, b' \} $. Fixing a bijection $ \phi_{\sigma} $ determines two distinct paths on the closure of $\sigma $ whose intersection is $ \sigma $; they are defined by $ [a \rightarrow \phi_{\sigma}(a)] $ and $ [a'\rightarrow \phi_{\sigma}(a')] $.
\begin{remark}\label{Définition combinatoire des géo distinguées}
	Let $ \mathcal{E} $ be a smooth manifold with a topology $ \mathfrak{T} $. A  linking structure   on $\mathcal{E}$  induces a bijection $\phi_{\sigma}: \partial_1 \sigma \rightarrow \partial_2 \sigma $ for every element $ \sigma \in \Sigma_2 $. And a  \textbf{distinguished} geodesic of $\mathcal{E}$ is a maximal geodesic such that for all $ \sigma \in \Sigma_2$, the geodesic follows a path defined by $ \phi_{\sigma} $.
\end{remark}
\begin{lemma}
	Let $\mathcal{E}$ be a smooth manifold with topology $\mathfrak{T}$. Giving a linking structure on $\mathcal{E}$ is equivalent to giving a bijection $\phi_{\sigma}: \partial_1 \sigma \to \partial_2 \sigma$, for all $\sigma \in \Sigma_2$, such that the set $G:=\{ \gamma$, where $\gamma$ is a maximal geodesic such that $ \forall \sigma \in \Sigma_2, \gamma$ follows a path defined by $\phi_{\sigma}\}$ covers $\mathcal{E}$. 
\end{lemma}
\begin{proof}
What we have to show is how one can define a linking structure on $\mathcal{E}$ when given these bijections on every $\sigma \in \Sigma_2$, in such a way that the set of distinguished geodesics coincides with $G$. We define an equivalence relation $\sim$ on $G$ in the following way: let $\alpha,$ $\beta \in G$, and choose two points $p$ and $p'$ on $ \alpha$ and $\beta$ respectively. There exists a piecewise smooth path in $\mathcal{E}$ with no cycle, joigning $p$ and $p'$; denote it by $r$. Since $G$ covers $\mathcal{E}$ and $r$ is compact, $r$ can be covered by a finite number of geodesics of $G$; denote by $\{\gamma_0, ..., \gamma_n \}$ the sub-family of $G$ covering $r$ such that: $\gamma_0 = \alpha$, $\gamma_n = \beta$, and $\gamma_i \cap \gamma_j \neq \emptyset$ if and only if $ |i-j|=1$ (when $r$ is fixed, this family is well defined). Set $\alpha \sim \beta$ if $n$ is an even number. The equivalence relation provides a partition of $G$ into two disjoint equivalence classes $\mathcal{F}_1$ and $\mathcal{F}_2$ satisfying $\mathcal{F}_1 \cup \mathcal{F}_2 = \mathcal{E}$. We claim that the parity of $n$ does not depend on the choice of the path $r$, $p$ and $p'$. To see this, suppose there is another path $r'$, with no cycle, joining $p$ and $p'$, and let $\{\gamma'_0, ..., \gamma'_m \}$  be the sub-family of $G$ covering $r'$, such that $\gamma'_0 = \alpha$, $\gamma'_m = \beta$, and $\gamma'_i \cap \gamma'_j \neq \emptyset$ if and only if $ |i-j|=1$. It is clear that the family $\{\gamma_0,...,\gamma_n, \gamma'_0,..., \gamma'_m \}$ defines a closed path in $\mathcal{E}$; let us first assume that each geodesic in  $\{\gamma_0,...,\gamma_n, \gamma'_1,..., \gamma'_{m-1} \}$ appears only once. One can define by induction a decomposition of this family into sub-families $f_j:=\{\gamma_{i_j}, \gamma_{i_j+1},..., \gamma_{i_{j+1}}; \gamma'_{i'_j}, \gamma'_{i'_j+1}, ..., \gamma'_{i'_{j+1}} \}, 0 \leq j \leq k-1$, with $k \in \N$, $i_0=0, i'_0=1, i_k=n, i'_k=m-1$,  such that each of them is a simple cycle of geodesics of $G$. This can be done in the following way: at each step, $f_j$ is the subset containing $\{\gamma_{i_j}, \gamma'_{i'_j} \}$ and satisfying: $\gamma'_{i'_j} \cap (\gamma_{i_j +1} \cup ... \cup \gamma_{n}) = \emptyset$, $\gamma'_l \cap (\gamma_{i_j} \cup ... \cup \gamma_{n}) = \emptyset$, with $i'_j+1 \leq l \leq i'_{j+1}-1$, $\gamma'_{i'_{j+1}} \cap (\gamma_{i_j} \cup ... \cup \gamma_{i_{j+1}-1}) = \emptyset$, and $\gamma_{i_{j+1}} \cap \gamma'_{i'_{j+1}} \neq \emptyset$. 


Since $\mathcal{G}(\mathcal{E})$ is a tree, each cycle $f_j$ is necessarily a finite order branching (Remark \ref{Remarque: condition to be a real tree}), hence one of order $4$. So if we count the number of geodesics in the total family, we have $4k-2(k-1)+2=2k+4$ geodesics. Moreover, we know by hypothesis that there are $m+n+2$ of them, so we get $m+n=2k+2$, hence $m+n$ is an even number. It follows that $m$ and $n$ have the same parity. It appears from this proof that the conclusion does not depend  on the choice of the points $p$ and $p'$ either. 

Now suppose that $\{\gamma_0,...,\gamma_n, \gamma'_1,..., \gamma'_{m-1} \}$ is no longer injective. Let $(s_j)_{1\leq j \leq l}$ and $(s'_j)_{1\leq j \leq l}$ be two maximal subsets of $\{1,..,n \}$ and $\{1,..,m \}$ respectively, such that $\gamma_{s_j} = \gamma'_{s'_j}$, for all $1 \leq j \leq l$. Consider the  family
$ \begin{Bmatrix} \gamma_{s_j} & ... & \gamma_{s_{j+1}} \\ \gamma'_{s'_j} &...& \gamma'_{s'_{j+1}} \end{Bmatrix} $, for $1 \leq j \leq l-1$.
If such a family contains only elements of the sequences $(s_j)$  and $(s'_j)$, then it has an even cardinal number; otherwise, it is as above and then has also an even number of elements. Finally, since two consecutive families have an even number of elements in common, we conclude that the global family has an even cardinal number, which implies that $m+n$ is even and proves our claim.


Now we have to see that if two distinct geodesics belong to the same family $\mathcal{F}_i$, $i=1,2$, then they are necessarily disjoint.  Assume this is not the case; this gives, by definition of $\mathcal{F}_i$, a cycle in $G$ with an odd cardinal number, which is impossible. Finally, the transversality condition (see Definition \ref{Définition linking structure}) is a consequence of the fact that there is no cycle in $\mathcal{E}$ other then $4$-order branchings; this finishes the proof. 


\end{proof}
Now let $ h: \mathcal{E} \rightarrow \mathcal{E}^{'} $ be a homeomorphism between two manifolds with a topology $ \mathfrak{T} $. Let $ \mathcal{A} $ (resp. $ \mathcal{A}'$) be a linking structure on $ \mathcal{E} $ (resp. $ \mathcal{E}^{'} $);
this gives on each element $ \sigma \in \Sigma_2 $ (resp. $ \sigma'\in \Sigma_2^{'} $) a bijection $ \phi_{\sigma} $ (resp. $ \phi'_{\sigma'} $). The homeomorphism $ h $ sends the distinguished geodesics of $ \mathcal{E} $ into those of $ \mathcal{E}^{'} $ if and only if for all $ \sigma \in \Sigma_2 $ and $ \sigma'\in \Sigma_2^{'} $ such that $ h(\sigma) = \sigma'$, we have
\begin{align}\label{bijections_conjugaison}
\phi'_{\sigma'}=h \textrm{\,o\,} \phi_{\sigma} \textrm{\,o\,} h^{-1}.
\end{align}
\begin{remark}
	The set $G$ in the previous lemma does not necessarily  cover $\mathcal{E}$, so this requirement on $G$ cannot be omitted  (see Example \ref{Exemple: G_A ne recouvre pas E} below). 
\end{remark}
\begin{example}\label{Exemple: G_A ne recouvre pas E}
	In the following example (see Figure 11 below), we have a sequence of branch points $(p_n)_n$ converging along a real line to a point $p$ such that $ [p]_l = \{p\} $.
We define on $\mathcal{E}$ a linking structure such that a distinguished geodesic  from $\mathcal{E}$ contains exactly one branch point on the horizontal line. We associate a color from the set $\{$ green, red $\}$ to every distinguished geodesic, in such a way that two geodesics with a common intersection in $\Sigma_1$ have two different colors. The branch point from the horizontal line then takes the color (green or red) of the geodesic to which it belongs. This  gives an alternating sequence of colors for the points $(p_n)$. We see that there is no distinguished geodesic that goes through $p$, for such a geodesic forces the branch points on a certain neighborhood of $p$ to be of the same color.
	It follows that if there exists a distinguished geodesic  through $p$, the linking structure is determined in a neighborhood of $p$.
	\begin{figure}[h!]\label{Géodésiques admissibles, sigma}
		\labellist 
		\small\hair 2pt 
		\pinlabel {$p$} at 1735 120
		\endlabellist 
		\centering 
		\includegraphics[scale=0.2]{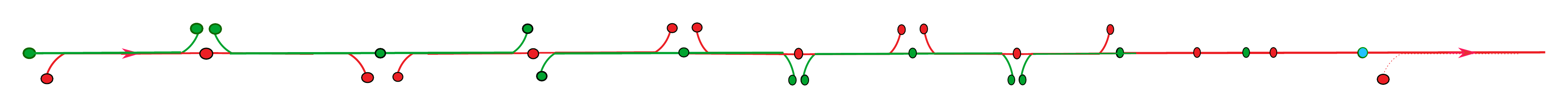} \caption{A non-covering set $G$}
	\end{figure} 
\end{example}

\begin{example}
	In the following example, the bands \textbf{A} and \textbf{B} are non isometric bands, as well as \textbf{A} and \textbf{C}. We have two surfaces $ (X_1, K_1) $ and $ (X_2, K_2) $ such that the manifolds $ \mathcal{E}_{X_1} $ and $ \mathcal{E}_{X_2} $ are homeomorphic, and the homeomorphism $ h: \mathcal{E}_{X_1} \rightarrow \mathcal{E}_{X_2} $ acts on the unique element $ \sigma \in \Sigma_2 $ by fixing the two non separated points on one side and permuting those on the other side. This homeomorphism preserves the translation structure and the two functions $ F_1 $ and $ F_2 $ induced by the norm of the Killing fields. Therefore, it appears that giving only the topological manifold $ \mathcal{E} $ and the pair $ (\textbf{x}, F) $ is not sufficient to determine the class of isometry of the surface.
	\begin{figure}[h!]\label{Géodésiques admissibles, sigma}
		\labellist 
		\small\hair 2pt 
		\pinlabel {\textbf{A}} at 502 845
		\pinlabel {\textbf{A}} at 766 525 
		\pinlabel {\textbf{B}} at 344 118
		\pinlabel {\textbf{A}} at 1500 834
		\pinlabel {\textbf{B}} at 1790 525 
		\pinlabel {\textbf{A}} at 1355 74
		\pinlabel {\textbf{C}} at 39 412 
		\pinlabel {\textbf{C}} at 960 393
		\endlabellist 
		\centering 
		\includegraphics[scale=0.12]{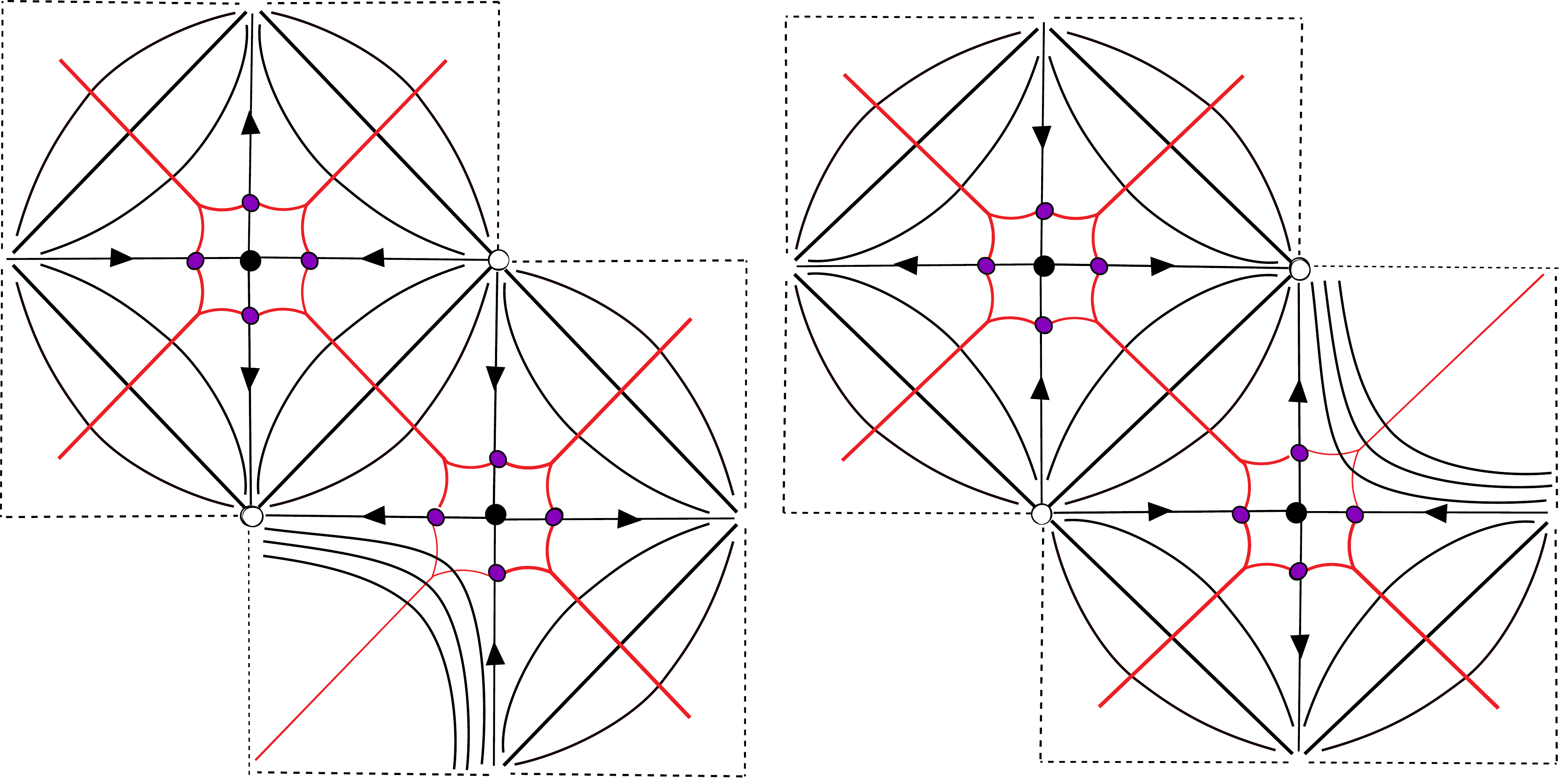} \caption{Two non-isometric Lorentzian surfaces with a Killing field, such that there exists a homeomorphism between the space of leaves which preserves the Riemannian structure on the space of leaves and the norm of the Killing field.}
	\end{figure}
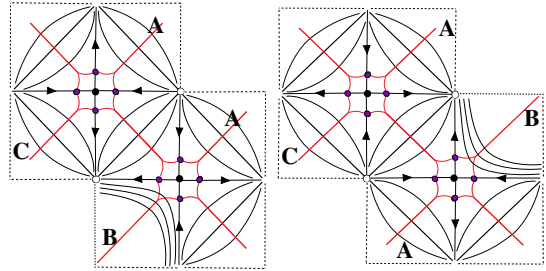 	
\end{example}
\subsection{Correspondence between maximal surfaces and étalé manifolds}
In this section, we give a description of simply connected and maximal Lorentzian surfaces admitting a non-trivial Killing field. It is proved in \cite[Proposition 3.1]{BM} that a Lorentzian surface $(X,K)$ is maximal among surfaces admitting a Killing field if and only if $X$ is maximal in the usual sense. It appears in the proof of this proposition that such a surface is maximal if and only if all its maximal ribbons are inextendible. 

\begin{definition}
Let $\mathcal{E}$ be a $1$-dimensional smooth manifold with topology $\mathfrak{T}$, and let $\mathcal{A}$ be a linking structure on it.  A function $F \in C^{\infty}(\mathcal{E}, \R)$ is said to be $\mathcal{A}$-inextendible if $\displaystyle F_{|_{\gamma}}$, the restriction of $F$ to $\gamma$, is maximal for all $\gamma \in G_{\mathcal{A}}$.  
\end{definition}

If $ \mathcal{E} $ is a $1$-dimensional manifold with a countable base, recall that the set of branch points of $ \mathcal{E} $ is denoted by $ B $, and $ \Sigma$ is the set of connected components of the interior of $ \mathcal{E} \smallsetminus B $ (it is a countable set). The elements of $\Sigma$ are of two types: those with Hausdorff closure in $ \mathcal{E} $, denoted by $\Sigma_0$, and those with non-Hausdorff closure, denoted by $\Sigma_1$.

Giving a simply connected maximal Lorentzian surface $ (X, K) $ with a Killing field $ K $ defines a quadruplet ($ \mathcal{E},  \mathcal{A}, \textbf{x}, F $) , where $ \mathcal{E} = \mathcal{E}_X $, $\mathcal{A}$ is defined by distinguished geodesics of $\mathcal{E}_X $, $ \textbf{x} \in C^{\infty}(\mathcal{E}_X,\R) $ is a local diffeomorphism induced by the transverse coordinate of $ X $ up to translation and  change of sign, and $ F \in C^{\infty}(\mathcal{E}_X,\R) $ is the function induced by the norm of $K$. On every distinguished geodesic of $\mathcal{E}_X$, $\textbf{x}$ is a diffeomorphism into an interval $I$ of $\R$; this induces a function $f \in C^{\infty}(I,\R)$ such that $F = f \circ \textbf{x}$. The action of the $ D_2 $ dihedral group on $ \textbf{x} $ by translation and  change of sign induces a  right  action of $ D_2 $ which consists in replacing $ f(\textbf{x}) $ by $ f(\epsilon \textbf{x} + b) $, for $ (\epsilon, b) \in \{\pm 1 \} \times \R $, on every distinguished geodesic of $\mathcal{E}_X$.

An isometry $\phi: (X_1,K_1) \rightarrow (X_2,K_2)$  induces a diffeomorphism $h: \mathcal{E}_{X_1} \rightarrow \mathcal{E}_{X_2}$ such that
\begin{align}\label{quadruplets: conjugaison par un difféomorphisme}
\mathcal{A}_1=h \textrm{\,o\,} \mathcal{A}_2,\, \textbf{x}_1=\textbf{x}_2 \textrm{\,o\,} h + \alpha, \,\alpha \in \R, \textrm{\,and\,} F_1=F_2 \textrm{\,o\,} h.
\end{align} 
In the following, two quadruplets $(\mathcal{E}_1, \mathcal{A}_1, \textbf{x}_1, F_1)$ and $(\mathcal{E}_2, \mathcal{A}_2, \textbf{x}_2, F_2)$ are said to be equivalent if there exists a diffeomorphism $h: \mathcal{E}_1 \to \mathcal{E}_2$ satisfying (\ref{quadruplets: conjugaison par un difféomorphisme}).
\begin{theorem}\label{Proposition: correspondance entre surfaces maximales et variétés étalées}
	To give a simply connected and maximal Lorentzian surface $(X,K)$ (up to isometry), admitting a non-trivial Killing  field $K$, is equivalent to giving a quadruplet $(\mathcal{E}, \mathcal{A}, \textbf{x}, F)$ (up to equivalence), where
\begin{enumerate}[itemsep=0pt, topsep=0pt, parsep=0pt]
	\item $\mathcal{E}$ is a $1$-dimensional smooth manifold with topology $ \mathfrak{T} $ (see Definition \ref{Définition topologie T}),
	\item $\mathcal{A}$ is a linking structure on $\mathcal{E}$,
	\item $ \textbf{x}: \mathcal{E} \to \R $ is a smooth function defined up to translation and change of sign, which is a local diffeomorphism, defining on $ \mathcal{E} $ a translation structure,
	\item  $ F \in C^{\infty}(\mathcal{E}, \R) $ is an $\mathcal{A}$-inextendible function such that the closed set $\{F=0\}$ is composed of 
	
	(a) the closure of the set of branch points of $ \mathcal{E} $, with simple zeros on $4$-order branchings.
	
	
	(b) elements of $ I \in \Sigma $ whose closure in $ \mathcal{E} $ is Hausdorff.

\end{enumerate}
\end{theorem}
Before we start the proof, we need to state the following lemma.
\begin{lemma}\label{Corollaire hors contexte}
	Let $ f \in C^{\infty}(\R, \R) $ be a function with a unique zero at $ 0 $. Let $ G $ be a primitive function of $ -1/f $ on $]-\infty,0[$. We have $ \displaystyle \lim_{x \to 0} G(x) = \infty $. 
\end{lemma}
\begin{proof}
	We use the fact that if $ f $ is a $C^{\infty}$ function that vanishes at $ 0 $, then $ f(x)/x $ can be extended to a $C^{\infty}$ function.
\end{proof}
\begin{proof}[\textbf{Proof of Theorem \ref{Proposition: correspondance entre surfaces maximales et variétés étalées}}]
Ley $ (X, K) $ be a simply connected maximal Lorentzian surface with a Killing field $ K $; it is clear that the function induced by the norm of $K$ satisfies condition (a), and also (b) by maximality of $X$.\\  
Conversely, let $ \mathcal{E} $ be a $1$-dimensional manifold with topology $ \mathfrak{T} $, and $ \textbf{x} \in C^{\infty}(\mathcal{E},\R) $ a local diffeomorphism, and $ F \in  C^{\infty}(\mathcal{E},\R) $ a function that satisfies conditions (a) and (b). We will show that there exists a unique Lorentzian metric on the plane -up to isometry- admitting a Killing field such that: (i) the space of Killing orbits is homeomorphic to $ \mathcal{E} $ and the distinguished geodesics of the metric coincide with those defined by the linking structure $\mathcal{A}$ on $\mathcal{E}$, (ii) the norm of the field is given by $ F $, (iii) and the Riemannian structure induced on $ \mathcal{E} $ by the metric coincides with that given by $ d \textbf{x}^2 $. \\
Let $\mathcal{F}_1$ and $\mathcal{F}_2$ be the two sets of distinguished geodesics of $\mathcal{E}$ given by the linking structure. They define two disjoint families of charts covering $ \mathcal{E} $.
 For all $ \mathfrak{a} \in \mathcal{F}_1 $ and $ \mathfrak{b} \in \mathcal{F}_2 $, $ \mathfrak{a} \cap \mathfrak{b} = \emptyset $ or $ \mathfrak{a} \cap \mathfrak{b} = \sigma \in \Sigma_1 $.
	
	\begin{itemize}
		\item Define a parametrization of $ \Sigma_1 $ by a subset $ A $ of $ \Z $ so that if $ (\alpha_1, \alpha_2, \alpha_3, \alpha_4) $ is an indexing of a cycle of $ 4 $ elements $ \sigma_i $ ($ i \in \Z/4\Z $) of $ \Sigma_1$ around a $ 4 $ order branching, such that $ \sigma_i $ and $ \sigma_{i + 1} $ belong to the same distinguished geodesic of $ \mathcal{E} $ for all $ i \in \Z/4\Z $, then we have the relation $ \alpha_1 - \alpha_2 = \alpha_4 - \alpha_3 $. Such a parametrization exists since, by hypothesis, there are no cycles of $4$-order branchings in $ \mathcal{E} $.

		\item Every element of $ \Sigma _1 $ is the intersection of a unique element of $ \mathcal{F}_1 $ and a unique element of $ \mathcal{F}_2 $. There is an indexing of the geodesics in $ \mathcal{F}_1 \cup \mathcal{F}_2 $ by a set $ I \amalg J $, where $ I $ and $ J $ are subsets of $ \Z $, such that for all $ i \in I, j \in J, \alpha \in A $, if $ \mathfrak{a}_i \in \mathcal{F}_1 $, $ \mathfrak{b}_j \in \mathcal{F}_2 $ and $ \mathfrak{a}_i \cap \mathfrak{b}_j = \sigma_{\alpha} \in \Sigma_1 $, then $ i + j = \alpha $. This defines a subset $ S $ of $ \Z^2 $ such that the correspondance that maps $ (i, j) \in S $ to $ \alpha \in A $ such that $ i + j = \alpha $ is bijective . We obtain it in the following way: we fix an element $ \sigma_{\alpha_0} \in \Sigma_1 $ and we write $ \alpha_0 = i_0 + j_0 $, $ i_0 \in \Z, j_0 \in \Z $; we denote by $ \mathfrak{a}_{i_0} $ the geodesic of $ \mathcal{F}_1 $ containing $ \sigma _{\alpha_0} $ and $ \mathfrak{b}_{j_0} $ the geodesic of $ \mathcal{F}_2 $ containing $ \sigma_{\alpha_0} $. This determines the indices of all distinguished geodesics; indeed, if $ c $ is a geodesic of $ \mathcal{F}_1 \cup \mathcal{F}_2 $, there is a finite path of distinguished geodesics $ c_0, ..., c_n $ such that $ c_0 = c $ and $ c_n = \mathfrak{a}_{i_0} $, satisfying $ c_k \cap c_{k + 1} \neq \emptyset $. This makes it possible to define an index for $ c $ using the relation $ i + j = \alpha $ if and only if $ \mathfrak{a}_i \cap \mathfrak{b}_j = \sigma_{\alpha} $. We show next that the indexing of $ c $ does not depend on the chosen path, using the relation $ \alpha_1 - \alpha_2 = \alpha_4 - \alpha_3 $ for a $ 4 $ order cycle of simple branch points.
	\end{itemize}
	Let $ (H, K) $ be a ribbon such that $ \mathcal{E}_H = \mathfrak{a} \in \mathcal{F}_1 $. Assuming that the Riemannian structure induced by the metric coincides with that given by $ d \textbf{x}^2 $, and that the norm of the Killing field is given by $ F $, then the metric on $ H $ in an adapted basis reads $ g = 2 \epsilon dx dy + f(x) dy^2 $, with $ \epsilon = \pm 1 $, $ K = \partial_y $ and $F=f \circ \textbf{x}$. On a ribbon $ (V, K) $ corresponding to an element $ \mathfrak{b} \in \mathcal{F}_2 $ such that $ \mathfrak{a} \cap \mathfrak{b} = \sigma_{\alpha} \in \Sigma_1 $, we have $ g = -2 \epsilon dx'dy' + f(x') dy'^2 $, with $ K = \partial_y '$. Let's call $ C_{\alpha}^H $ (resp. $ C_{\alpha}^V $) the open square of $ H $ (resp. $ V $) such that $ \mathcal{E}_{C_{\alpha}^H} = \mathcal{E}_{C_{\alpha}^V} = \sigma_{\alpha} $; the two squares are glued by means of the map $ x'= x, y' = y- \epsilon G_{\alpha}(x) $, where $ G_{\alpha} $ is a primitive function of $ -1/f_{\alpha} $ on the interval $ \textbf{x}(\sigma_{\alpha}) $. \\

	For all $ \alpha \in A $, $ \sigma_{\alpha} = \mathfrak{a}_i \cap \mathfrak{b}_j $ for a unique pair $ (i, j) \in S $. For every $ \alpha \in A $, we choose a primitive function $ G_{\alpha} $ of $ -1/f_{\alpha} $. Let $ (X_0, K) $ be the Lorentzian surface with a Killing field $ K $ defined as the quotient space of the surface $ (\amalg H_i) \displaystyle \amalg_{(i, j) \in S}(\amalg V_j) $, where $ (H_i, K) $ is the ribbon $ \textbf{x}(\mathfrak{a}_i) \times \R $ endowed with the metric $ 2 \epsilon dx dy + f_i(x) dy^2 $ with $ K = \partial_{y} $, and $ (V_j, K) $ the ribbon $ \textbf{x}(\mathfrak{b}_j) \times \R $ with metric $-2\epsilon dx dy + f_j (x) dy^2$ with $ K = \partial_{y} $, by the equivalence relation $ \sim $ defined by
	$$p \sim q \textrm{\;si\;} p=q \textrm{\;ou\;} (p,q) \in C_{\alpha}^{H_i} \times C_{\alpha}^{V_j}, (i,j) \in S, \alpha = i+j \in A \textrm{\;et\;} q=\phi_{\alpha}(p),$$
	where $\phi_{\alpha}$ is the reflection of $C_{\alpha}^{H_i}$ defined by $G_{\alpha}$.
	
	Since $G_{\mathcal{A}}$ covers $\mathcal{E}$, the surface we obtain is the total space of a line bundle over $ \mathcal{E} $. It is connected by connectedness of $ \mathcal{E} $, and Hausdorff by Lemma \ref{Corollaire hors contexte} and Proposition 5 \cite{MR0219084}.

If $ \mathcal{E} $ is simply connected, then $ X_0 $ is an acyclic surface with a countable base (\cite[Proposition 3]{MR0219084}). And it is known that the only simply connected surfaces with a countable base and without boundary are the sphere $ S^2 $ and the plane $ \R^2 $ (this is a classical theorem in surface theory, see for example \cite[Theorem 3.2.2]{MR1658468}); finally $ X_0 $ is homeomorphic to the $ \R^2 $ plane. In addition, $ X_0 $ does not contain a zero of the Killing field, and the choice of $ \epsilon $ in the definitioh of the metric determines a choice of $ K $ or $ -K $. If $ \mathcal{E} $ contains branchings of order $ 4 $, the surface above such a branching is a quasi-saddle. Let $ (G_{\alpha_1}, G_{\alpha_2}, G_{\alpha_3}, G_{\alpha_4}) $ be the primitive functions fixed around a branching of order $4$; we can always modify $ G_4 $ so that we can add a saddle point (see \cite[Proposition 2.32]{BM}); since there are no cycles of $4$-order branchings in $ \mathcal{E} $, we can do this for every branching of order $4$ in a coherent way. It is easy to see that the surface thus obtained is simply connected.

Finally, the fact that $ f $ is $\mathcal{A}$-inextensible implies that every maximal ribbon in the surface is inestensible, hence the surface is maximal; if $ f $ is not extensible, the topology $ \mathfrak{T} $ on the space of leaves implies that the obtained surface still satisfies a weak property of maximality: every half-band is contained in a type III band.

To finish the proof, we have to see that  given two equivalent quadruplets $(\mathcal{E}_1, \textbf{x}_1,\mathcal{A}_1,F_1)$ and 	$(\mathcal{E}_2, \textbf{x}_2,\mathcal{A}_2,F_2)$, the corresponding surfaces obtained above are isometric. Denote by $h$ the homeomorphism defining this equivalence, and by $(X_1,K_1)$ and $(X_2,K_2)$ the two Lorentzian surfaces.  On every ribbon of $X_1$ and $X_2$, define local coordinates as above, by choosing on each of them an origin for the $y$-coordinate. It is possible to do this globally on each surface in a coherent way  (\cite[Proposition 2.32]{BM}). Fix two ribbons $(U_1,K_1)$ and $(U_2,K_2)$ in $X_1$ and $X_2$ respectively, such that $h$ sends the distinguished geodesic above $U_1$ to the one above $U_2$. The metric on $U_1$ reads $2 \epsilon_1 dx_1 dy + f_1(x_1) dy^2$ in local coordinates, where $\epsilon_1 = \pm 1$ corresponds to a choice of an orientation on the Killing field. The metric on $U_2$ reads  $2 \epsilon_2 dx_2 dy + f_2(x_2) dy^2$, with $\epsilon_2 = \pm 1$. 
The isometry then sends $(x_1,y)$ to $(x_2, \epsilon_1 \epsilon_2 y)$, such that $\textbf{x}_2^{-1}(x_2) = h \circ \textbf{x}_1^{-1}(x_1)$.  It is globally well defined, and this achieves the proof.  
\end{proof}
\begin{remark}
	Let $\mathcal{E}$ be a $1$-dimensional smooth manifold with topology $\mathfrak{T}$, and let $\textbf{x} \in C^{\infty}(\mathcal{E},\R)$ be a local diffeomorphism. To say that a function $F \in C^{\infty}(\mathcal{E}, \R)$ is $\mathcal{A}$-extendible, for some linking structure $\mathcal{A}$ on $\mathcal{E}$,  is in general not equivalent to saying that $(\mathcal{E},\textbf{x},F)$ is extendible to a bigger manifold $(\widetilde{\mathcal{E}}, \widetilde{\textbf{x}},\widetilde{F})$ with $\widetilde{\mathcal{E}}$ having a topology $\mathfrak{T}$. In Example \ref{Exemple: l'importance du choix de linking structure}, one may define a function $F$ and a linking structure for which $F$ is extendible, and another one for which $F$ is not extendible. 
\end{remark}
\begin{remark}
Let $ (E_f^u, K) $ be a universal surface. If $ \mathcal{A} $ and $  \mathcal{A}^{'} $ are two linking structures on the space of leaves $ \mathcal{E}_{E^u_f} $, then there is a homeomorphism $ h $ of $ \mathcal{E}_{E^u_f} $ such that $ \mathcal{A}^{'}~=~h $~o~$ \mathcal{A} $~o~$ h^{- 1} $, and which preserves the transverse coordinate as well as the norm of the Killing field. We then find the fact that the isometry class of $ E^u_f $ is determined by $ f $.
\end{remark}
\begin{remark}(\textbf{Lorentzian surfaces with a Killing field and principal bundles})\\
	We can see in the proof of Theorem \ref{Proposition: correspondance entre surfaces maximales et variétés étalées} that to any simply connected Lorentzian surface $ (X, K) $ admitting a non-trivial Killing field $ K $, one can associate a principal line bundle over the space of leaves $ \mathcal{E}_X $.
	Conversely, given a manifold $ \mathcal{E} $ of dimension $1$ with a countable base and topology $ \mathfrak{T} $, étalé in $ \R $, together with a principal line bundle $ \mathcal{L} = (L, p, \mathcal{E}) $ over $ \mathcal{E} $, one can define exactly two Lorentzian metrics $ \pm g $ on the plane admitting a Killing field, such that the principal bundle associated to these metrics coincide with $ \mathcal{L} $.
\end{remark}

\paragraph{Topological characterization of the space of leaves} 

\begin{corollary}\label{Corollaire 1 (cas complet): caractérisation top de l'espace des feuilles}
	Let $\mathcal{E}$ be a $1$-dimensional smooth manifold with topology $\mathfrak{T}$, and let $\textbf{x} \in C^{\infty}(\mathcal{E},\R) $ be a local diffeomorphism. Suppose there exists a linking structure $\mathcal{A}$ such that $\mathcal{A}$-distinguished geodesics are complete. Then there is a smooth, simply connected and complete Lorentzian surface $(X,K)$ with $\mathcal{E}_X \simeq \mathcal{E}$.
\end{corollary}
\begin{proof}
	By Theorem \ref{Proposition: correspondance entre surfaces maximales et variétés étalées} above and Theorem \ref{Complétude à courbure bornée} in the next paragraph, it suffices to prove the existence of a function $F \in C^{\infty}(\mathcal{E},\R)$, with bounded transverse second derivative, such that the set $\{F=0\}$ is as in Theorem \ref{Proposition: correspondance entre surfaces maximales et variétés étalées}. 
	Let $\gamma \in G_{\mathcal{A}}$. We know that $\textbf{x}_{|\gamma}: \gamma \to \R$ is a diffeomorphism. Let $S_{\gamma}$ be the closed subset of $\R$ composed of:\\
	(a) $\textbf{x}_{|\gamma}(B)$, where $B$ is the set of branch points of $\mathcal{E}$,\\
	(b) $\textbf{x}_{|\gamma}(\sigma)$,  $\forall \sigma \in \Sigma_0$,\\
	(c) $\textbf{x}_{|\gamma}$, $\forall p \in \mathcal{E} \smallsetminus B$  a limit point of $B$.\\
	Consider the functions $\phi_0$ and $\phi_1$ defined by:  	
	\[
	\phi_0(x)  =	\left\lbrace
		\begin{array}{ll}
		0 & \mbox{if\;}x \leq 0 \\
		e^{-1/x} & \mbox{if\;} x>0
	\end{array}
	\right.
	\qquad ; \qquad
	\phi_1(x)  =	\left\lbrace
	\begin{array}{ll}
	0 & \mbox{if\;\;}x \leq 0 \mbox{\;or\;} x \geq 1\\
	e^{1/x(x-1)} & \mbox{if\;\;} x \in ]0,1[
	\end{array}
	\right.
	\]
    These are $C^{\infty}$ functions with bounded derivatives. We write $\R \smallsetminus S_{\gamma} = \displaystyle \bigcup_{n \in \N} I_n$ as the union of  pairwise disjoint open intervals of $\R$, with possibly one or two unbounded intervals. If $I_n=]a_n,b_n[$ is bounded, set $f_n(x)= \epsilon(n) e^{-1/\lambda_n} \phi_1((x-a_n)/\lambda_n)$, with $\lambda_n = b_n-a_n$ and $\epsilon(n)=\pm 1$. This function vanishes exactly on $\R \smallsetminus I_n$. If $I_n=]a_n, + \infty[$ (resp. $I_n=]-\infty, a_n[$) is unbounded, set $f_n(x) = \epsilon(n) \phi_0(x-a_n)$ (resp. $\phi_0(a_n-x)$). On a fixed distinguished geodesic, the branch points belonging to a branching of order $4$ are isolated. Denote by $B_1=(b_j)_{j \in \Z}$ the (countable) subset of $B$ corresponding to these branch points. Let $p_j=\textbf{x}(b_j) \in \textbf{x}(B_1 \cap \gamma)$: $p_j$ delimits two intervals $I_{n_j}$ and $I_{n_{j+1}}$. For all $j \in \Z$, we assume that the functions $(f_n)_n$ satisfy the following property: if $f_{n_j} \leq 0$ then $f_{n_{j+1}} \geq 0$. It's easy to check that the function defined by $\hat{f}:=\sum_{n \in \N} f_n$ is a $C^{\infty}$ function with bounded derivatives, and that it vanishes exactly on $S_{\gamma}$. \\
	Now, fix  any $C^\infty$ function $\phi_2$ with bounded derivatives, such that $\phi_2(x)=0$ if $x \in ]-\infty,-1[ \cup \{0\} \cup$\\$ ]1,+\infty[$, and $\phi'(0)=1$.  For any $j \in \Z$, let $U_j=B(p_j,r_j)$, $r_j >0$, the open ball of center $p_j$ and radius $r_j$, such that $U_j \cap \textbf{x}(B) = \{p_j\}$. Define a $C^{\infty}$ function $h_j$ as follows: $h_j(x)=\epsilon(j) e^{-1/r_j} \phi_2((x-p_j)/r_j)$, with $\epsilon(j) =1$ if $f_{n_j} \leq 0$, and $\epsilon(j) = -1$ otherwise. The function $h:= \sum_{j \in \Z} h_j$ is a $C^{\infty}$ function with bounded derivatives, and with a simple zero at $p_j, \forall j \in \Z$. Set $f_{\gamma}:= \hat{f}+ h$. This is a $C^{\infty}$ function with bounded derivatives, that vanishes exactly on $S_{\gamma}$, and with simple zeros on $\textbf{x}(B_1)$. Let $\mathcal{F}_1$ be one of the two families of $G_{\mathcal{A}}$ containing $\gamma$, we define $f$ by induction on each  geodesic of $\mathcal{F}_1$ in such a way that on two  geodesics $\gamma_1$ and $\gamma_2$ that share a  $4$-order branching of $\mathcal{E}$, the functions $ f_{\gamma_1} \circ \textbf{x} $ and $f_{\gamma_2} \circ \textbf{x}$ can be glued into a $C^\infty$ function on $\mathcal{E}$. We obtain this way a function $F$ defined on $\mathcal{E}$, which is $C^{\infty}$  on the open sets belonging to $\mathcal{F}_1$. It is easy to check that  $F_{|\gamma} \circ \textbf{x}^{-1}$  is also a $C^{\infty}$ function on $\R$, $\forall$  $\gamma \in \mathcal{F}_2$, with bounded derivatives. So $F$ is a $C^{\infty}$ function on $\mathcal{E}$, with bounded derivatives such that the set $\{F=0\}$ is as in Theorem \ref{Proposition: correspondance entre surfaces maximales et variétés étalées}, which ends the proof.
	
\end{proof}

\begin{corollary}\label{Corollaire 3 (cas maximal): caractérisation top de l'espace des feuilles}
	Let $\mathcal{E}$ be a $1$-dimensional smooth manifold with topology $\mathfrak{T}$, and let $\textbf{x} \in C^{\infty}(\mathcal{E},\R) $ be a local diffeomorphism. Suppose that the set of branch points is locally finite, and that $\textbf{x}$ is unbounded on elements of $\Sigma_0$ near infinity. Then there is a smooth, simply connected and maximal Lorentzian surface $(X,K)$ with $\mathcal{E}_X \simeq \mathcal{E}$.
\end{corollary}
\begin{proof}
	Let $\mathcal{A}$ be any linking structure on $\mathcal{E}$.
	Let $\gamma \in \mathcal{F}_1$, and let $S_{\gamma}$ be as in the proof of Corollary \ref{Corollaire 1 (cas complet): caractérisation top de l'espace des feuilles} above. Here, $\textbf{x}$ is a diffeomorphism between $\gamma$ and an interval $I$ of $\R$. If $I = \R$, we construct $f$ on $I$ as in the proof of Corollary \ref{Corollaire 1 (cas complet): caractérisation top de l'espace des feuilles}. If $I=].,a[$ is bounded from one side,  the hypothesis that $B$ is locally finite implies that $S_{\gamma} \cap ]a-1,a[$ is finite. 	Furthermore, we assumed that $\textbf{x}$ is unbounded on elements of $\Sigma_0$ near infinity, so we can construct $f$ on $I$ in such a way that it is unbounded on $]a-1,a[$, hence maximal on $I$. We  do this on every distinguished geodesic  of $\mathcal{F}_1$, so that the determined function $F$ defined on $\mathcal{E}$ is smooth. The obtained function is maximal on every element of $\mathcal{F}_1$. Now, if $\gamma \in \mathcal{F}_2$ with bounded interval $\textbf{x}(\gamma)$, since $B$ is locally finite, the geodesic $\gamma$ contains finitely many components of $\Sigma$; it follows that $f$ is unbouded on $\textbf{x}(\gamma)$, hence maximal. The function $F$ we obtain is $C^{\infty}$ and maximal on every distinguished geodesic of $\mathcal{E}$, and it defines by Theorem \ref{Proposition: correspondance entre surfaces maximales et variétés étalées} a smooth, simply connected and maximal Lorentzian surface.   
	
\end{proof}

\begin{corollary}\label{Corollaire 2 (cas complet): caractérisation top de l'espace des feuilles}
	Let $\mathcal{E}$ be a $1$-dimensional smooth manifold with topology $\mathfrak{T}$. Suppose that the set of branch points is locally finite. Then there is a smooth, simply connected and complete Lorentzian surface $(X,K)$ with $\mathcal{E}_X \simeq \mathcal{E}$.
\end{corollary}
\begin{proof}
Let $\mathcal{A}$ be any linking structure on $\mathcal{E}$. By definition of the $\mathfrak{T}$ topology, there exists a smooth function $\textbf{x}: \mathcal{E} \to \R$ which is a local diffeomorphism. Define $\lambda := \min_{\sigma \in \Sigma} |\sigma|$, where $|\sigma|$ is the length of $\sigma$.  Since the set $B$ is locally finite, $\lambda >0$. We consider the new function $\textbf{x}_1:= 1/\lambda\; \textbf{x}$, which is a local diffeomorphism from $\mathcal{E}$ to $\R$ such that the length of $\sigma$ for this new function is bigger than $1$, for all $\sigma \in \Sigma$. Now, again using the fact that $B$ is locally finite, we can modify $\textbf{x}_1$ on each $\mathcal{A}$-distinguished geodesic so that the elements of $\Sigma$ near infinity will be of infinite length.  We get  a local diffeomorphism $\textbf{x}_1^{\infty} \in C^{\infty}(\mathcal{E},\R)$ such that all $ \mathcal{A}$-distinguished geodesics are complete. Corollary  \ref{Corollaire 1 (cas complet): caractérisation top de l'espace des feuilles} ensures the existence of a smooth, simply connected and complete Lorentzian surface $(X,K)$ such that $\mathcal{E}_X \simeq \mathcal{E}$.  
\end{proof}

\section{Completeness of Lorentzian surfaces with a Killing field}\label{Section: complétude}
\subsection{Critical Clairaut constant associated to a semi-ribbon}
In what follows, we denote by $R_f=(R,K)$ the surface $(R=I \times \R, 2dxdy+ f(x)dy^2), (x,y) \in I \times \R$, with a Killing field $K = \partial_y$. It is called the "ribbon associated to $f$".\\

Let $(X,K)$ be a Lorentzian surface with a non-trivial Killing field $K$. We assume that it is null complete. Recall that the set of connected components of the interior of $\mathcal{E}_X \smallsetminus B$, where $B$ is the set of branch points of $\mathcal{E}_X$, is denoted by $\Sigma$. 

A maximal geodesic $\gamma$ may have two different behaviors: either it leaves any maximal ribbon contained in $X$, or it remains in a maximal ribbon provided $t$ goes close enough to the limit. We denote by $t_{\infty}^{+} \in \R \cup \{\infty\}$ the upper bound of the domain of $\gamma$. In the second case, we shall consider two behaviors again; one may assume that $\gamma$ is defined and contained in the ribbon for $t \in [0,t_{\infty}^{+}[ $, by translating the geodesic parameter if necessary. Set $ I^{+} := \{x (\gamma (t)), t \in [0,t_{\infty}^{+}[ \} $; either $I^{+}$ is bounded, in which case the geodesic remains in a band for $t \geq 0$, or $I^{+}$ is unbounded.\\

Denote by $T$ the unit vector field tangent to $\gamma$, and $N$ the vector field along $\gamma$ orthogonal to $T$, such that the basis $(T,N)$ is positively oriented. Set $$K=C T + \beta N.$$
Then $C=\epsilon \langle T,K \rangle$ is a constant called the Clairaut constant (see \cite{MR3039767}, p. 3), and $\beta=-\epsilon \langle K,N \rangle$ is a solution of the Jacobi equation.  Notice that if $K(p), p \in \gamma$, is not degenerate, $\beta(p)=0$ if and only if $\gamma$ is tangent to $K$ at $p$.

\begin{lemma}(\cite[Lemma 2.12]{LM}).\label{asymptote à K}
	Let $ \gamma $ be a non-null semi-geodesic that remains in a band for $t$ close enough to $t^+_{\infty}$, with bounded interval $I^{+}$. Then:
	\begin{itemize}[topsep=0pt, itemsep=0pt, parsep=0pt]
		\item either $\gamma$ is invariant by some element of the flow; in particular, $\gamma$ is entirely contained in that band;
		\item or it asymptotically approaches a critical orbit of $K$. If $\gamma$ is not orthogonal to $K$, this orbit is either timelike or spacelike, depending on the type of $\gamma$; otherwise, it is a null orbit of $K$. 
	\end{itemize} 
\end{lemma}
A critical orbit of $K$ is an orbit corresponding to a critical point of the function $\langle K,K \rangle$ (these orbits are geodesics).
\begin{proposition}(\cite[Proposition 2.14]{LM}).\label{Complétude dans une bande}
	A non-null semi-geodesic $\gamma$ that lies in a band after a certain while with bounded $I^+$ is complete.
\end{proposition}

Now completeness reduces to non-null geodesics that a) either remain in a maximal ribbon as $t$ goes to  $t_{\infty}^{+}$ with unbounded $I^{+}$, or b) leaves any maximal ribbon. 

\begin{lemma}(\cite[Lemma 2.15]{LM}). \label{Behavior of geodesics, paragraph completeness}
	Let $\gamma$ be a non-null geodesic not perpendicular to $K$. Assume $\gamma$ cuts a null orbit of $K$, denoted by $\mathfrak{l}$. Then \\  
	(i)	 $\gamma$ does not cross a type II band containing $\mathfrak{l}$;  \\
	(ii) if $\beta$ does not vanish, $\gamma$ lies in the maximal ribbon containing  $\mathfrak{l}$;\\
	(iii) if $\gamma$ is tangent to $K$ in the band containing $\mathfrak{l}$, then $\gamma$ leaves the maximal ribbon containing $\mathfrak{l}$. 
\end{lemma}
In what follows, we may suppose that $x^{'} >0$ on a geodesic $\gamma$ transverse to $K$, by changing $K$ to $-K$ in the local chart, if necessary.
\begin{corollary}\label{Corrolaire: caractériser géodésiques traversant tout un ruban}
	1) If $\gamma_{\epsilon}$ remains in a maximal ribbon $R_f$ with unbounded $I^{+}$, then 
	\begin{itemize}[topsep=0pt, itemsep=0pt, parsep=0pt]
		\item either $C^2 > \displaystyle \sup_{x \in I^{+}} \epsilon f(x)$,
		\item or $C^2 = \displaystyle \sup_{x \in I^{+}} \epsilon f(x)$ and  $m:=\displaystyle \sup_{x \in I^{+}} \epsilon f(x) $ is not a critical point of $f$ on $I^{+}$. 
	\end{itemize}
	2) If $\gamma_{\epsilon}$ cuts a null orbit $\mathfrak{l}$ of $K$, of coordinate $x_0$, contained in a maximal ribbon $R_f$, and if $C^2 >  \displaystyle \sup_{x \geq x_0} \epsilon f(x)$, then $\gamma$ remains in $R_f$ with unbounded $I^{+}$. 
\end{corollary}
Before proving the corollary, note that 2) is a partial converse of 1): the property in the following remark also holds; it follows from \cite[Corollary 4.7]{LM}. 
\begin{remark}
	If $\gamma_{\epsilon}$ cuts a null orbit of $K$, of coordinate $x_0$, contained in a maximal ribbon $R_f$, and if $C^2 = \displaystyle \sup_{x \geq x_0} \epsilon f(x)$ and $m:=\displaystyle \sup_{x \geq x_0} \epsilon f(x) $ is not a critical value of $f$ for $x \geq x_0$, then $I^{+}$ is unbounded.
\end{remark}
\begin{proof}[\textbf{Proof of Corollary \ref{Corrolaire: caractériser géodésiques traversant tout un ruban}}]	
	1) It follows from assertion (iii) of Lemma \ref{Behavior of geodesics, paragraph completeness} that if $\gamma$ remains in a maximal ribbon  $R_f$ after a certain while, then $C^2 > \epsilon f(x)$, for all $x \in I^{+}$. Hence $C^2 \geq \displaystyle \sup_{x \in I^{+}} \epsilon f(x)$. 
	
	Now, if $m:=\displaystyle \sup_{x \in I^{+}} \epsilon f(x) < \infty$ is a critical value of $f$ on $I^{+}$, we claim that a geodesic $\gamma_C$ with $C^2 = m$ cannot cross this critical orbit transversally, for otherwise the formula $C^2 - \beta^2 = \epsilon f$ applied at a point that belongs to the orbit yields $m - \beta_0^2 = m$, with $\beta_0^2 >0$, which is impossible. It follows that $\gamma$ does not cross the whole ribbon. We conclude that if $I^{+}$ is unbounded and if $C^2 = \displaystyle \sup_{x \in I^{+}} \epsilon f(x)$, then $m$ is not a critical value of $f$ on $I^{+}$.\\
	
	2) The assumption $C^2 > \displaystyle \sup_{x \geq x_0} \epsilon f(x)$ implies, using (ii) Lemma \ref{Behavior of geodesics, paragraph completeness}, that $\gamma$ lies in the maximal ribbon containing $\mathfrak{l}$, after a certain while. Suppose, contrary to our claim, that $I^{+}$ is bounded; Lemma \ref{asymptote à K} implies that $\gamma$ asymptotically approaches a leaf of $K$, on which the norm of $K$ is $C^2$, a contradiction. 
\end{proof}
The following lemma characterizes completeness for geodesics contained in a maximal ribbon for $t$ large enough, with unbounded interval $I^{+} = x(\gamma(t)_{t \geq 0} )$. Note that if a timelike geodesic (resp. spacelike) remains in a ribbon $ R_f $ with unbounded $ I^{+}$, then  $ f $ is necessarily bounded below (resp. above) on $\R^{+}$, for otherwise any geodesic of the ribbon  would eventually leave the ribbon or asymptotically approach a leaf of $K$ (this is a consequence of Corollary \ref{Corrolaire: caractériser géodésiques traversant tout un ruban}, 1).
 
\begin{lemma}[\textbf{Semi-ribbon completeness criterion}]	\label{Complétude dans un ruban}
	Let $(X,K)$ be a simply connected null complete Lorentzian surface admitting a Killing field $K$. 
	Let $R_f$ be a maximal ribbon in $X$. Then, every geodesic contained in the semi-ribbon $R_{f_{|\R^{+}}}$ with unbounded interval $I^{+}$ is complete if and only if for all $\epsilon \in \{\pm1\}$ such that  $ \sup_{\R^{+}} \epsilon f < +\infty$, 
	\begin{itemize}[topsep=0pt, itemsep=0pt, parsep=0pt]
		\item there exists $M >0$ such that $\mu (\{-\epsilon f \leq M\} \cap \R^{+}) = \infty$ ($\mu$ is the Lebesgue measure).  \\
		or
		\item for all $M>0$, $\mu (\{-\epsilon f \leq M\} \cap \R^{+}) < + \infty$, and 
		\begin{align}\label{condition analytique complétude dans un ruban}
		\exists \alpha >0, \int_{\{-\epsilon f \geq \alpha\} \cap \R^{+}} \frac{dx}{\sqrt{- \epsilon f(x)}} = \infty.
		\end{align}
	\end{itemize}   
	In particular, if $ f$ is bounded, these geodesics are complete.
\end{lemma}
In the sequel, the condition in Lemma \ref{Complétude dans un ruban} will be referred to as the (SRC) condition.
\begin{proof} We start with the following observation:\\
	\textbf{Observation:} if a geodesic $ \gamma_{C_0} $ which remains in the semi-ribbon is complete, then any geodesic $ \gamma_C $ of the semi-ribbon such that $ | C | \leq | C_0 | $, is complete too. This follows from the fact that the time spent in the semi-ribbon is a decreasing function of $ C $.
	
	Fix $\gamma_{\epsilon}$ a geodesic  contained in $R_f$ for $x \geq x_0, x_0 \in \R$. The geodesic is complete if and only if the integral $t_{\infty}^{+} = \int_{x_0}^{\infty} \frac{dx}{\sqrt{C^2 -\epsilon f(x)}}$, where $C$ is the Clairaut constant of $\gamma$, diverges. 
	The first part of the Lemma follows from the following inequality  $$ \int_{\{-\epsilon f \leq M\} \cap \{x \geq x_0 \}} \frac{dx}{\sqrt{C^2-\epsilon f(x)}} \geq \frac{\mu \{\{- \epsilon f \leq M\} \cap \{x \geq x_0 \}\}}{\sqrt{C^2 +M}}.$$
	Assume now that for all $M>0$, $\mu (\{-\epsilon f \leq M\} \cap \R^{+}) < + \infty$. \;\;\; (*)\\
	Since $\gamma$ is contained in $R_f$ for $x \geq x_0$, Corollary \ref{Corrolaire: caractériser géodésiques traversant tout un ruban} yields $C^2 \geq \displaystyle \sup_{x \geq x_0} \epsilon f(x)$. If $C^2 > \displaystyle \sup_{x \geq x_0} \epsilon f(x)$, then $C^2 \geq \displaystyle \sup_{x \geq x_0} \epsilon f(x) + \eta$, where $\eta >0$, and the inequality  $$\int_{\{-\epsilon f \leq C^2 \} \cap \{x \geq x_0 \}} \frac{dx}{\sqrt{C^2- \epsilon f(x)}} \leq \frac{\mu\{\{-\epsilon f \leq C^2 \} \cap \{x \geq x_0 \}\}}{\sqrt{\eta}}$$
	implies the convergence of the integral on the left.\\
	Therefore,  $\int_{x_0}^{+\infty} \frac{dx}{\sqrt{C^2 -\epsilon f(x)}}$ diverges if and only if $\int_{\{-\epsilon f \geq C^2\} \cap \{x \geq x_0 \}} \frac{dx}{\sqrt{C^2 -\epsilon f(x)}}$ does too. 
	Now write  $$\int_{\{-\epsilon f \geq C^2\} \cap \R^{+}} \frac{dx}{\sqrt{2}\sqrt{-\epsilon f(x)}} \leq \int_{\{-\epsilon f \geq C^2\} \cap \R^{+}} \frac{dx}{\sqrt{C^2 -\epsilon f(x)}} \leq \int_{\{-\epsilon f \geq C^2\} \cap \R^{+}} \frac{dx}{\sqrt{-\epsilon f(x)}} .$$
	These inequalities show that $\gamma$ is complete if and only if $\int_{\{-\epsilon f \geq C^2\} \cap \R^{+}} \frac{dx}{\sqrt{-\epsilon f(x)}} = \infty$, which is equivalent to the condition of the Lemma when (*) is satisfied.
	
	\vspace{0.2cm}
	Finally, if for all $ M> 0 $, $\mu (\{-\epsilon f \leq M\} \cap \R^{+}) < + \infty$, and if the condition (\ref{condition analytique complétude dans un ruban}) is satisfied, then for all $x_0 \in \R$, the geodesics such that $C^2 > \displaystyle \sup_{x \geq x_0} \epsilon f(x)$ are complete, and their completeness leads also to that of geodesics such as $ C^2 = \displaystyle \sup_{x \geq x_0} \epsilon f(x) $, using the above observation. Hence we obtain the completeness of all the geodesics that remain in the semi-ribbon. Assume now that condition (\ref{condition analytique complétude dans un ruban}) is not satisfied; for all $x_0 \in \R$ and $ C^2> \displaystyle \sup_{x \geq x_0} \epsilon f(x) $, there is a geodesic $ \gamma $ remaining in the semi-ribbon with unbounded interval $I^{+}$, and whose Clairaut constant is $ C $ (see Corollary \ref{Corrolaire: caractériser géodésiques traversant tout un ruban}, 2), and we have shown that such a geodesic is then incomplete. This ends the proof.
\end{proof}
\begin{corollary}[\textbf{Completeness of a null complete saddle}]\label{Complétude d'une selle L-complète}
	Let $(S,K)$ be a null complete saddle defined by giving a branching of order $4$, together with  $4$ functions $f_i, i \in \Z/4\Z$, defined on $(\sigma_i)_{i \in \Z/4\Z}$ de $\Sigma$, such that $f_i$ and $f_{i+1}$ have the same germ at the point $b_i = \bar{\sigma}_i \cap \bar{\sigma}_{i+1}$. Then $S$ is complete if and only if for all $i \in \Z/4\Z$, $f_i$ satisfies the (SRC) condition of Lemma \ref{Complétude dans un ruban}. 
\end{corollary}
\begin{corollary}[\textbf{Completeness of universal extensions}] Let $X=E^u_f$ (in particular, analytic surfaces), and assume $X$ is null complete, i.e. $f$ is defined over $\R$. Then $X$ is complete if and only if $R_{f_{|\R^+}}$ and $R_{f_{|\R^-}}$ satisfy the (SRC) condition in Lemma \ref{Complétude dans un ruban}.
\end{corollary}
\begin{proof}
If $ \gamma $ is a geodesic in the universal extension which leaves any maximal ribbon, then $ \gamma $ is invariant by an isometry of the universal extension,  acting on $\gamma$ by a translation of the geodesic parameter (see \cite[Lemma 4.8]{LM}), proving that $\gamma$ is complete. It is called an "invariant" or "periodic geodesic" in \cite{LM}. Now, for geodesics remaining in a maximal semi-ribbon with unbounded $I^+$, the geodesic completeness is equivalent to the (SRC) condition, which ends the proof.
\end{proof}
\begin{corollary}\label{Complétude des extensions universelles des tores}
The extension $E^u_f$ associated to a Lorentzian torus $(T,K)$ is  complete. 
\end{corollary}
\begin{definition}[\textbf{Critical Clairaut constant}]\label{C critique, définition}
According to the observation in the proof of Lemma \ref{Complétude dans un ruban}, we can define,
\begin{align*}
C^*:=\inf&\{|C|; \textrm{\;there exists a geodesic\;}\,\gamma_C \textrm{\;with Clairaut constant\;} C, \textrm{\;contained in \;} R_{f_{|\R^{+}}} \textrm{\;with}\\ 
&\textrm{unbounded\;} I^+, \textrm{\; such that\;} \gamma_C \textrm{\;is incomplete}\}.
\end{align*}
\end{definition} 
\begin{remark}\label{C critique, remarque SRC}
1) The critical Clairant constant depends on $K$ (on the norm of it). \\
2) One can also define $C^*_{\epsilon}$, for $\epsilon \in \{\pm 1\}$, by considering only $\epsilon$-type geodesics in the Definition \ref{C critique, définition} above.\\
3) The (SRC) condition in Lemma \ref{Complétude dans un ruban} is equivalent to $C^* = + \infty$.
\end{remark}
\begin{proposition}\label{C critique, proposition comportement}
	Let $(X,K)$ be a simply connected null complete Lorentzian surface admitting a Killing field $K$. 	Let $R_f$ be a maximal ribbon in $X$, and $C^*$ the critical Clairaut constant associated to the semi-ribbon  $R_{f_{|\R^+}}$. Then
	\begin{enumerate}[topsep=0pt, itemsep=0pt, parsep=0pt]
		\item if $C^* = +\infty$, then all geodesics remaining in the semi-ribbon with unbounded $I^+$ are complete,
		\item if $C^* < + \infty$, then $C^*= C^*_{\epsilon} < + \infty$\, for some $\epsilon \in \{\pm 1\}$, with $C^*_{\epsilon} = \displaystyle \inf_{x_0 \in \R^{+}} \max \{\sup_{x \geq x_0} \epsilon f, 0\}$, and $C^*_{-\epsilon}=+\infty$. In this case, all $\epsilon$-type geodesics remaining in the semi-ribbon with Clairaut constant $|C|>C^*_{\epsilon}$ are incomplete. 
	\end{enumerate}
\end{proposition}
\begin{proof}
This comes from the proof of Lemma \ref{Complétude dans un ruban} and Remark \ref{C critique, remarque SRC}, 3), above.
\end{proof}
\begin{remark}
1) When  $C^*= C^*_{\epsilon} < + \infty$, then $\epsilon$-type geodesics with Clairaut constant  $C=C^*_{\epsilon}$, remaining in the semi-ribbon with unbounded $I^{+}$, could be complete or incomplete (see Remark \ref{C critique selle} below). \\
2) If $C^* < + \infty$ and if the semi-ribbon has an infinite band, i.e. $f_{|\R^+}$ has constant sign $\eta \in \{\pm 1\}$ at infinity, then $ C^*=C^*_{-\eta} = 0$.
\end{remark}

\begin{remark}\label{C critique selle}
For a null complete saddle, it follows from Proposition \ref{C critique, proposition comportement} above that either the saddle is complete, or $\epsilon$-type geodesics, for some $\epsilon \in \{\pm 1\}$, not orthogonal to $K$ and remaining in some semi-ribbon are all incomplete, i.e. $C^*_{\epsilon}=0$ for this semi-ribbon. In Paragraph \ref{section: examples and non-examples} we have an example of an incomplete saddle with incomplete orthogonal geodesics (see Example \ref{Exemple 2.3.19}), and an example of an incomplete saddle in which spacelike and timelike geodesics orthogonal to $K$ are complete and the others are incomplete (see Example \ref{exemple selle incomplète et orthogonal complète}). 
\end{remark}
\subsection{Small surfaces}
\begin{definition}[\textbf{Small surfaces}]\label{Définition petites surfaces}
	We call  "a small surface" any maximal and simply connected Lorentzian surface, with a Killing field, containing a finite number of ribbons.
\end{definition}
\begin{proposition}
	Let $\mathcal{E}$ be a $1$-dimensional manifold with topology $\mathfrak{T}$. There is a small surface $(X,K)$ such that $\mathcal{E}_X \simeq \mathcal{E}$ if and only if $B$ is finite.
\end{proposition}
\begin{proof}
	Consider $(X,K)$ a small surface. Since $X$ has a finite number of ribbons, then it has necessarily a finite number of (null or non-null) bands. Indeed, if $R$ is a ribbon with infinitely many non-null bands, then any maximal surface containing $R$ has infinitely many ribbons. So any ribbon in $X$ has a finite number of non-null bands, hence a finite number of null bands as well. This implies that $\mathcal{E}_X$ has a finite number of branch points. The converse is a consequence of Corollary \ref{Corollaire 2 (cas complet): caractérisation top de l'espace des feuilles}.
\end{proof}

Let $\mathcal{E}$ be a $1$-dimensional manifold with topology $\mathfrak{T}$. When $\mathcal{E}=\mathcal{E}_X$, we defined a null component of $\mathcal{E}_X$ to be the space of leaves of a null band in $X$ (Definition \ref{Definition null band - null component}). They have Hausdorff closure, and when $X$ is maximal, they are the only elements of $\Sigma$ with Hausdorff closure. 
\begin{definition}\label{Definition Y-piece, generalized Y-piece}
	1) A $\mathcal{Y}$-manifold is a $1$-dimensional manifold homeomorphic to a simple branching line, i.e. to the quotient space of two copies of the real line $\R \times \{a\}$ and $\R \times \{b\}$ with the equivalence relation ${\displaystyle (x,a)\sim (x,b){\text{ if }}x<0}$.  \\
	2) A generalized $\mathcal{Y}$-manifold is a $1$-dimensional manifold homeomorphic to the quotient space of a finite number of copies of the real line $\R \times \{a_1\}$,...,$\R \times \{a_n\}$ with the equivalence relation ${\displaystyle (x,a_i)\sim (x,a_{i+1}){\text{ if }} (-1)^{i+1} x<0}$. \\
	3) A $4$-order branching line is a $1$-dimensional manifold homeomorphic to the quotient space of $4$ real lines $\R \times \{a_i\}, i \in \Z/4\Z$, with the equivalence relation ${\displaystyle (x,a_i)\sim (x,a_{i+1})}$  if $(-1)^{i+1} x<0, \forall i \in \Z/4\Z$.  
\end{definition}
\begin{figure}[h!] 
	\labellist 
	\small\hair 2pt 
	\endlabellist 
	\centering 
	\includegraphics[scale=0.15999]{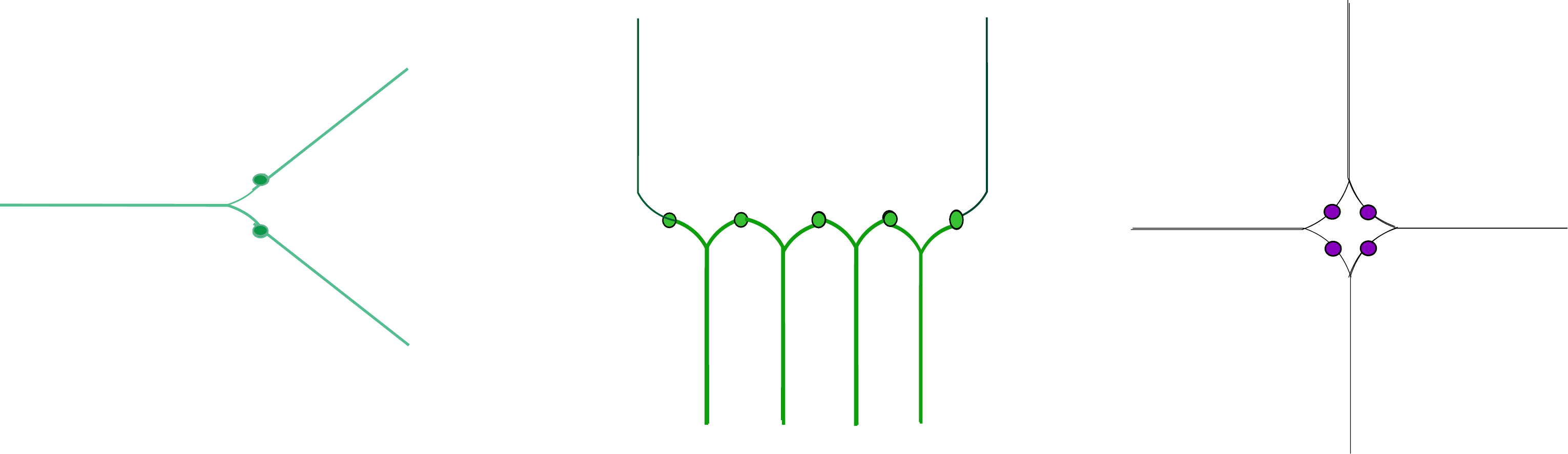} \caption{a $\mathcal{Y}$-piece;\; a generalized $\mathcal{Y}$-piece;\; a $4$-order branching line}
\end{figure} 

\textbf{Notation:} The family of $1$-dimensional manifolds homeomorphic to a generalized $\mathcal{Y}$-manifold or to a $4$-order branching line will be denoted by $\mathfrak{F}$.
\begin{definition}
	Let $F, F' \in \mathfrak{F}$; we say that two elements $\sigma \in \Sigma, \sigma' \in \Sigma'$  are of the same type if $\sigma \in \Sigma_1$ and $\sigma' \in \Sigma'_1$ or $\sigma \in \Sigma_0$ and $\sigma' \in \Sigma'_0$.
\end{definition}
Let $\mathcal{E}$ be a $1$-dimensional manifold with topology $\mathfrak{T}$. An open subset of $\mathcal{E}$ is called a generalized $\mathcal{Y}$-piece of $\mathcal{E}$ if it is homeomorphic to a generalized $\mathcal{Y}$-manifold and is maximal in $\mathcal{E}$ for this property. A generalized $\mathcal{Y}$-piece of $\mathcal{E}$ is then the union of a finite chain of simple branchings and the elements $\sigma, \sigma' \in \Sigma$, such that $p \in \bar{\sigma}$ and $p' \in \bar{\sigma'}$, where $p$  and $p'$ are the boundaries of the finite chain, and $\bar{\sigma} $ and $\bar{\sigma'}$ are Hausdorff. Similarly, the $4$-order branchings of $\mathcal{E}$ as defined in Definition \ref{Définition: branchement d'ordre n} are the maximal $4$-order branching lines contained in $\mathcal{E}$.

\begin{proposition}[Decomposition of small surfaces]
Let $(X,K)$ be a small surface, then $\mathcal{E}_X$ is the gluing of a finite number of generalized $\mathcal{Y}$-pieces and $4$-order branchings of $\mathcal{E}$ along elements of $\Sigma$ of the same type. Conversely, let $\mathcal{E}$ be a $1$-dimensional manifold obtained as a finite gluing of generalized $\mathcal{Y}$-manifolds and $4$-order branching lines in such a way that:\\
i) any two $F, F' \in \mathfrak{F}$ are glued along $\sigma \in \Sigma$ and $\sigma' \in \Sigma'$  of the same type;\\
ii) $ \mathcal{G}(\mathcal{E})$ is a tree. \\
If $\mathcal{E}$ has only simple branch points, then it is the space of leaves of a small surface.
\end{proposition}
\begin{proof}
The manifold $\mathcal{E}_X$ has a finite number of branch points, so $\mathcal{G}(\mathcal{E})$ has a finite number of vertices $\{s_1,..,s_n\}$.  For all $i \in \{1,..,n\}$, any simple branching with branch points in the class $s_i$ is contained in (the same) $4$ order branching, or in a finite chain of simple branchings; since $X$ is maximal, if $p$ is in the boundary of this chain, then $p$ is the boundary of a null component  (Fact \ref{Fact le bord d'une chaine finie se trouve dans une null component}). This means that branch points in the class $s_i$ are contained in a $4$-order branching line or in a generalized $\mathcal{Y}$-piece of $\mathcal{E}_X$. This gives a finite number $n+1$  of generalized $\mathcal{Y}$-pieces and $4$-order branchings in $\mathcal{E}_X$ glued along elements of $\Sigma$ of the same type. 
To do the converse, note that $\mathcal{E}_X$, where $X$ is maximal,  satisfies the property $(\mathcal{P})$. 
In consequence, if $\mathcal{E}$ is a finite gluing of generalized $\mathcal{Y}$-manifolds and $4$-order branching lines such that $\mathcal{E}= \mathcal{E}_X$, where $X$ is a small surface, then $\mathcal{E}$ satisfies condition i) (and ii)) of the proposition. Now, such a construction gives a $1$-dimensional manifold with topology $\mathfrak{T}$, which has a finite number of branch points, so it is the space of leaves of a small surface.

\end{proof}
\begin{corollary}[\textbf{Completeness of small surfaces}] Let $ (X, K) $ be a small surface. 
Then $ X $ is complete if and only if all the ribbons contained in $X$ satisfy the (SRC) condition. In particular, small surfaces with bounded function $F$ are complete.
\end{corollary}
\subsection{Completeness results with controlled geometry}
Let $(X,K)$ be a Lorentzian surface with a non-trivial Killing field $K$. Recall that the set of connected components of the interior of $\mathcal{E}_X \smallsetminus B$, where $B$ is the set of branch points of $\mathcal{E}_X$, is denoted by $\Sigma$. \\

The proofs of the results is this subsection are postponed to the end of the subsection.
\begin{definition}
Let $(X,K)$ be a Lorentzian surface with a non-trivial Killing field $K$. Let $F \in  C^{\infty}(\mathcal{E}_X,\R)$. The transverse derivative of $F$ at $p \in \mathcal{E}_X$ is the derivative of the function $f \in C^{\infty}(I,\R)$ at $\textbf{x}(p)$, where $F_{|\gamma}=f \circ \textbf{x}$ and $I=\textbf{x}(\gamma)$, and $\gamma$ is some distinguished geodesic through $p$.
\end{definition} 

\begin{theorem}\label{Complétude à courbure bornée}
	Let $(X,K)$ be a simply connected Lorentzian surface with Killing field $K$. Assume $X$ has bounded curvature. Then $X$ is complete if and only if it is null complete. 
\end{theorem}
\begin{corollary}\label{Complétude à courbure bornée, corollaire}	
	A Lorentzian surface with bounded sectional curvature whose group of isometries is of dimension $1$ is complete if and only if it is null complete.
\end{corollary}
Let $(\sigma_i)_{i \in I}$ be a sequence of pairewise distinct elements of $\Sigma$; such a sequence is called "normal" if \\
- the $\sigma_i$'s are contained in branchings of  $\mathcal{E}_X$ of order  $4$,\\
- for all $i \in I$, $\sigma_i$ and $\sigma_{i+1}$ belong to the same $4$ order branching, and are not adjacent in it, they are then opposite to each other, \\
- $(\sigma_i)_{i \in I}$ is maximal.  

Every non-null maximal geodesic orthogonal to $K$ defines a normal sequence in $\Sigma$; and conversely, to each such a sequence, one can associate a unique (up to the action of the flow) non-null maximal geodesic orthogonal to $K$.  
\begin{theorem}\label{Complétude à dérivée bornée}
	Let $(X,K)$ be a simply connected Lorentzian surface with Killing field $K$. Assume that $F$ has bounded transverse derivative. If $X$ is null complete, then geodesics not orthogonal to $K$ are complete. Therefore $X$ is complete if and only if it is null complete and the geodesics orthogonal to the Killing field are complete. The last condition reads: for any normal sequence  $(\sigma_i)_{i \in I}$ of $\Sigma$, we have 
	\begin{align}\label{Complétude à dérivée bornée, condition sur les géod orthogonales}
	\sum_{i} \int_{\sigma_{i}} \frac{dx}{\sqrt{| f_i(x)|}} = \infty.
	\end{align}
\end{theorem}
In Example \ref{Exemple dérivée bornée 1er}, we obtain a null complete saddle with bounded transverse derivative of $F$, and incomplete orthogonal geodesics joining two points. This shows that the condition (\ref{Complétude à dérivée bornée, condition sur les géod orthogonales}) in Theorem \ref{Complétude à dérivée bornée} above is not empty. Moreover, Example \ref{exemple selle incomplète et orthogonal complète} of an incomplete saddle with complete null geodesics and complete orthogonal geodesics shows that the condition on $F$ is not superfluous.
\begin{corollary}\label{Complétude à dérivée bornée et espace des feuilles complet}
Let $(X,K)$ be a simply connected null complete Lorentzian surface with Killing field $K$. Assume that $F$ has bounded transverse derivative and that $\mathcal{E}_X$ is complete (for the pseudometric topology), then $X$ is complete.  
\end{corollary}
\begin{remark}
Actually, the condition on the space of leaves in Corollary \ref{Complétude à dérivée bornée et espace des feuilles complet} can be weakened into:  $\sum_i |\sigma_i| = + \infty$ for every normal sequence $(\sigma_i)_i$ in $\Sigma$.
\end{remark}

We can now prove theorems \ref{Complétude à courbure bornée} and \ref{Complétude à dérivée bornée}.

\begin{proof}[\textbf{Proof of Theorem \ref{Complétude à courbure bornée}}]
	Assume that $ X $ is null complete and has bounded curvature, say by a constant $N>0$; we show that $X$ is complete. We first prove that geodesics that remain in a maximal ribbon are complete. Fix a maximal ribbon $R_f$ and a geodesic $\gamma$ of type $\epsilon$ lying in the ribbon with unbounded $I^{+}$. The null completeness of the surface implies that $f$ is defined over $\R$. Set $g(x):=C^2-\epsilon f(x)$; the hypothesis on curvature implies that $g^{''}$ is bounded, hence for all $x \in I^{+}$,
	\begin{align}\label{Courbure bornée: inégalité 1 sur f, ici seulement}
	|g(x)| \leq A(1+x^2),
	\end{align}
	where $A >0$ is a constant that depends on the constant $N$ and on the ribbon. This yields $$\int_{I^+} \frac{dx}{\sqrt{g(x)}} \geq \frac{1}{\sqrt{A}}\int_{I^+} \frac{dx}{\sqrt{1+x^2}} = \infty,$$ hence $\gamma$ is complete. 
	
	
	Assume now that $\gamma$ leaves any maximal ribbon. Denote by $\Delta_{x_i}$ the piece of $\mathcal{E}_X$ delimited by two consecutive zeros of $\beta$, on which $\gamma$ is thus contained in a maximal ribbon $R_{f_i}$, and denote by $|\Delta_{x_i}|$ its length.  We apply the argument on curvature again to get
	\begin{align}\label{Courbure bornée: inégalité 2 sur f, ici seulement}
	|g(x)| \leq N \cdot |\Delta_{x_i}|^2, \;\forall x \in \Delta_{x_i},
	\end{align}
	where $N>0$ is the constant defined above. Therefore, the time $\gamma$ takes to cross every piece is bounded below by a constant $1/\sqrt{N}$ independent of the given piece. Since $\gamma$ croses infinitely many such pieces of $\mathcal{E}_X$, the completeness follows.     
\end{proof}

\begin{proof}[\textbf{Proof of Theorem \ref{Complétude à dérivée bornée}}]
	Here we suppose that the transverse derivative of $F$ is bounded by a constant $N>0$. 
	Let $\gamma$ be a geodesic contained in a ribbon $R_f$ after a certain while. We have
	\begin{align}
	|C^2 - \epsilon f(x)| \leq A'(1+|x|),
	\end{align}
	where $A'$ is a constant that depends on the constant $N$ and on the ribbon. Hence 
\begin{align*}
	\int_{I^{+}} \frac{dx}{\sqrt{C^2- \epsilon f(x)}} \geq \frac{1}{\sqrt{A'}} \int_{I^{+}} \frac{dx}{\sqrt{1+|x|}} = \infty,
\end{align*}
 which proves that $\gamma$ is complete. 
	
	Assume now that $\gamma$ leaves any maximal ribbon. Denote by $\Delta_{x_i}$ the piece of $\mathcal{E}_X$ delimited by two consecutive zeros of $\beta$, on which $\gamma$ delimits a ribbon $R_{f_i}$, with $f_i$ defined over an interval $J_i$ of length $|\Delta_{x_i}|$, and $t_i$ the time $\gamma$ takes to cross every such piece. On each ribbon $R_{f_i}$, the hypothesis on the derivative of $f$ gives $|C^2 - \epsilon f_i(x)| \leq N \cdot |\Delta_{x_i}|$; this yields in each ribbon
	\begin{align}\label{Minorer le temps, ici seulement}
	t_i = \int_{\Delta_{x_i}} \frac{dx}{\sqrt{C^2- \epsilon f(x)}} \geq \frac{\sqrt{|\Delta_{x_i}|}}{\sqrt{N}} .
	\end{align}
Suppose that $\gamma$ is incomplete; it follows from (\ref{Minorer le temps, ici seulement}) that the series 	$\displaystyle \sum_{i} \sqrt{|\Delta_{x_i}|}$ converges. 
In what follows, we show that $\gamma$ is necessarily a geodesic orthogonal to $K$. Set $d := \displaystyle \sum_{i} t_i < \infty$. Define $\hat{f}_i(t) := f_i$ o $x(t)$ and $\hat{g}_i(t) := C^2 - \epsilon f_i$~o~$x~(t)$, $t \in J_i$. We have $|\hat{f}_i^{'}(t)| = |f_i^{'}(x) \cdot x^{'}(t)| = |f_i^{'}(x)| \sqrt{C^2 -\epsilon f_i(x(t))}$. By assumption, $f^{'}$ is bounded as well as $\displaystyle \sum_{i} \sqrt{|\Delta_{x_i}|}$, this implies that $\hat{f}_i^{'}$ is bounded on each $I_i$ by a constant $N'>0$ independent  of $i$. Now let $(d_i)_i$ be a sequence converging to $d$, such that for all $i$, $d_i \in J_i$. Let $(y_i)_i$ be a sequence of zeros of $\hat{f}_i$ converging to $d$, such that for all $i$, $y_i \in J_i$, and let $(z_i)_i$ be a sequence of zeros of $\beta$ (hence of $\hat{g}_i$) such that for all $i$, $z_i \in J_i$. In one hand $|\hat{f}_i(d_i) - \hat{f}_i(y_i)|~\leq~N'~(d_i - y_i)$,  in the other hand $|\hat{g}_i(d_i)~-~\hat{g}_i(z_i)|~\leq~N'~(d_i - z_i)$. Since $(d_i - y_i)_i$ and $(d_i - z_i)_i$ tend to $0$,  the sequence $(f_i$ o $x(d_i))$ tends to $0$ and $C^2$, which forces $C=0$. 

It follows that $X$ is complete if and only if the geodesics orthogonal to $K$ are complete. A geodesic orthogonal to $K$ that leaves every maximal ribbon leaves every ribbon through a saddle point, so that the pieces $\Delta_{x_i}$ it crosses define what we called in Theorem \ref{Complétude à dérivée bornée} a normal sequence  of elements of $\Sigma$, and the completeness reads as in (\ref{Complétude à dérivée bornée, condition sur les géod orthogonales}). 

\end{proof}

\begin{proof}[\textbf{Proof of Corollary \ref{Complétude à courbure bornée, corollaire}}]
	This corollary is an immediate consequence of Theorem \ref{Complétude à courbure bornée}. 
\end{proof}
\begin{proof}[\textbf{Proof of Corollary \ref{Complétude à dérivée bornée et espace des feuilles complet}}]
Using Theorem \ref{Complétude à dérivée bornée}, all we need to check is the completeness of orthogonal geodesics. Every non-null geodesic orthogonal to $K$ defines a normal sequence $(\sigma_i)_i$ in $\Sigma$. It appears from the proof of Theorem \ref{Complétude à dérivée bornée} that if  $\sum_i |\sigma_i| = + \infty$ (which is ensured here by the completeness of $\mathcal{E}_X$), then the series defined in (\ref{Complétude à dérivée bornée, condition sur les géod orthogonales}) diverges, so that the geodesic is complete. 
\end{proof}


\subsection{Completeness results with topological conditions on the space of leaves}
Null completeness is equivalent to the completeness of the distinguished geodesics of the space of leaves regarded as a Riemannian manifold. With sufficiently general conditions on this space, we define a large class of surfaces where geodesic completeness is characterized. It will appear in particular that the non-compact case contains many examples of complete surfaces. In order to write these conditions, we introduce in the following some definitions and vocabulary on the space of leaves.


Let $\mathcal{E}$ be a $1$-dimensional Riemannian manifold, with topology $\mathfrak{T}$. The first following definition (due to Spivak \cite{MR532830}, p. 23) applies to any locally compact topological space, and the second one to any Riemannian manifold. 
\begin{definition}\textbf{(Topological end).}
A topological end of $\mathcal{E}$ is a map $e$  which associates to each compact (not necessarily Hausdorff) subset ${\displaystyle K\subset  \mathcal{E}}$ a connected component of ${\displaystyle X\smallsetminus K}$ in such a way that $ \forall K_{1}\subset K_{2}:e (K_{2})\subset e(K_{1})$.
\end{definition}
\begin{definition}\textbf{(Geometric end).}
A geometric end of $\mathcal{E}$ is an end defined by a geodesic ray (a smooth maximal semi-geodesic) of $\mathcal{E}$, i.e. if $e$ is the application that defines this end, then $ e(K) $ contains the end of the geodesic ray for any compact subset $K$ of $\mathcal{E}$.
\end{definition}
\begin{figure}[h!] 
	\labellist 
	\small\hair 2pt 
	\endlabellist 
	\centering 
	\includegraphics[scale=0.14]{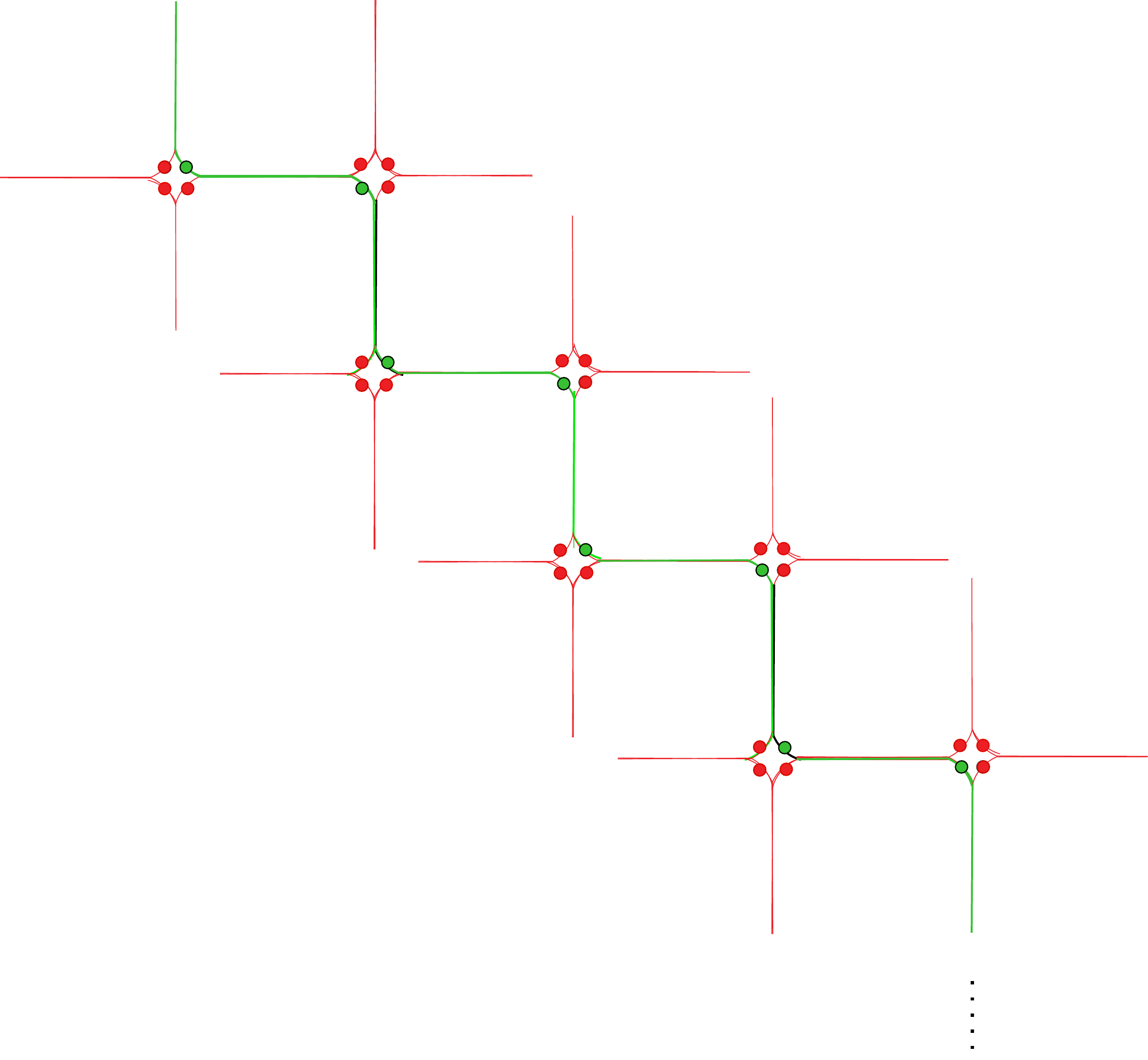} \caption{Example ($\mathcal{G}$): All ends of ($\mathcal{G}$) are geometric ends}
\end{figure} 

\paragraph{Definition of ends in terms of piecewise maximal semi-geodesics}

\begin{definition}\label{Définition des bouts avec les géodésiques brisées}
Let $\mathcal{E}$ be a $1$-dimensional manifold, with topology $\mathfrak{T}$. A \textbf{ray} in $\mathcal{E}$ is a (continuous) piecewise smooth maximal semi-geodesic in $\mathcal{E}$ in which any point of $B$ appears at most once.\\
Two rays are equivalent if their intersection is not contained in any compact subset of $\mathcal{E}$. 


\end{definition}
In the sequel, a singular point of a ray in $\mathcal{E}$ will be called a \textbf{cusp}. 
\begin{remark}
Since a compact subset of $\mathcal{E}$ contains finitely many $\mathcal{Y}$-pieces, any ray of $\mathcal{E}$ has finitely many cusps in a compact subset; therefore, a ray leaves any compact subset of  $\mathcal{E}$, and defines an end of this space: to each compact subset $K$ of $\mathcal{E}$, $e(K)$ is the connected component of $\mathcal{E} \smallsetminus K$ containing the end of the ray. 
\end{remark} 
\begin{fact}\label{Observation sur la nature des bouts --> chaines}	
	Let $\mathcal{E}$ be a $1$-dimensional Riemannian manifold with topology $\mathfrak{T}$.
	Since $\mathcal{G}(\mathcal{E})$ is a tree, if a point is chosen in $\mathcal{E}$, then each end of $\mathcal{E}$ contains a unique ray (up to equivalence) starting from this point; therefore, the ends may be placed in one to one correspondance with these rays. 
\end{fact}
\begin{figure}[h!]
	\labellist 
	\small\hair 2pt 
	\endlabellist 
	\centering 
	\includegraphics[scale=0.08]{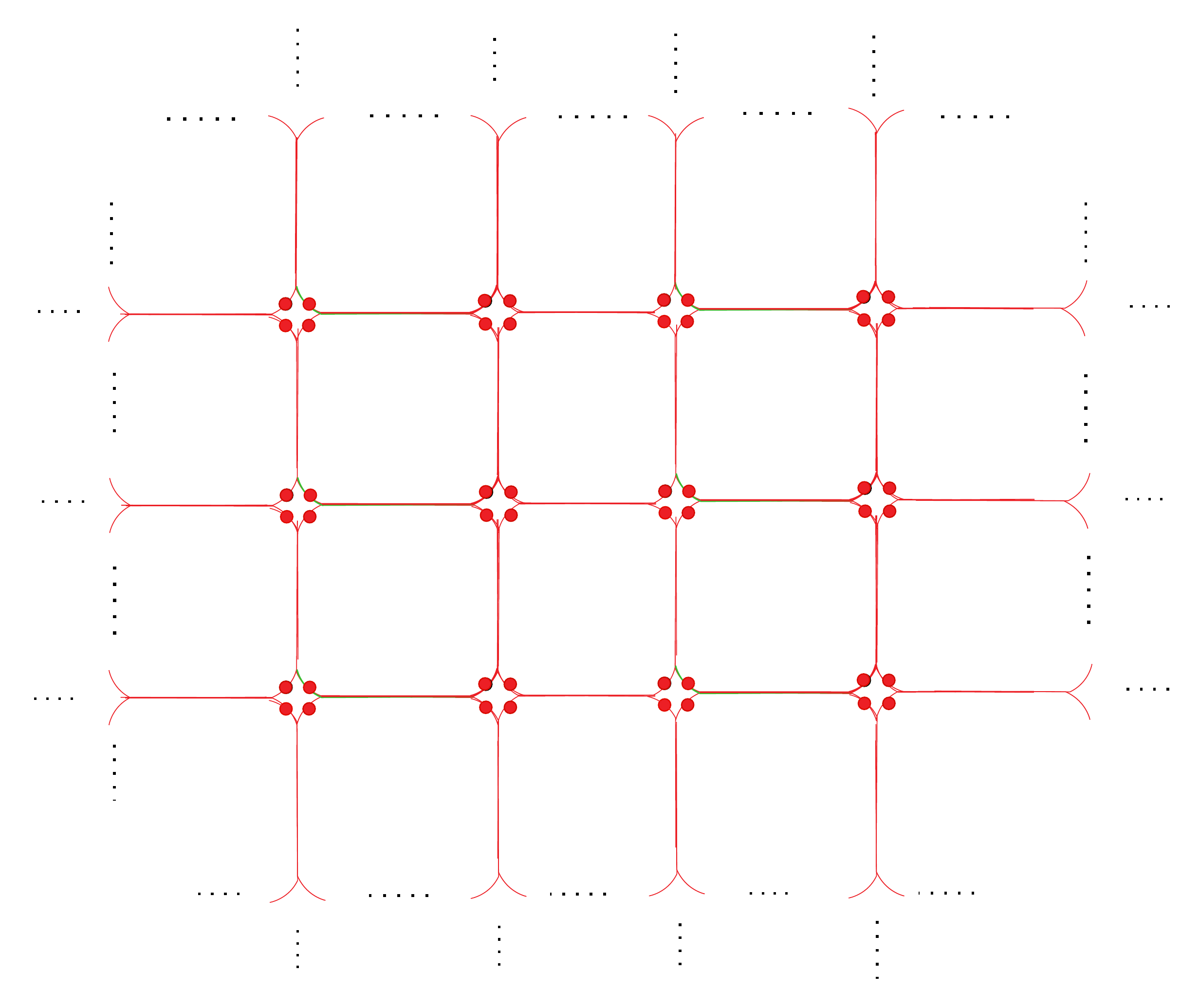} \caption{Infinite grid of $4$-order branchings: it has only one end, which is a geometric end defined by infinitely many non-equivalent rays}
\end{figure}

In the figure above, we see that the Fact \ref{Observation sur la nature des bouts --> chaines}  does not hold. Here, $\mathcal{G}(\mathcal{E})$ is not a tree.
\begin{definition}
Let $\mathcal{E}$ be a $1$-dimensional manifold, with topology $\mathfrak{T}$.\\
1) A \textbf{chain} in $\mathcal{E}$ is a semi-maximal path of $4$-order branchings of $\mathcal{E}$.\\ 
2) A \textbf{normal chain} is a maximal path of $4$-order branchings, such that the elements of $\Sigma$ interior to the path define a normal sequence of elements of $\Sigma$. \\
We denote by $N_{\infty}$ a semi-finite normal chain of $\mathcal{E}$, and by $N_{2}$ a normal chain of cardinal number$~2$.
\end{definition}

	\begin{figure}[h!] 
		\labellist 
		\small\hair 2pt 
		\endlabellist 
		\centering 
		\includegraphics[scale=0.175]{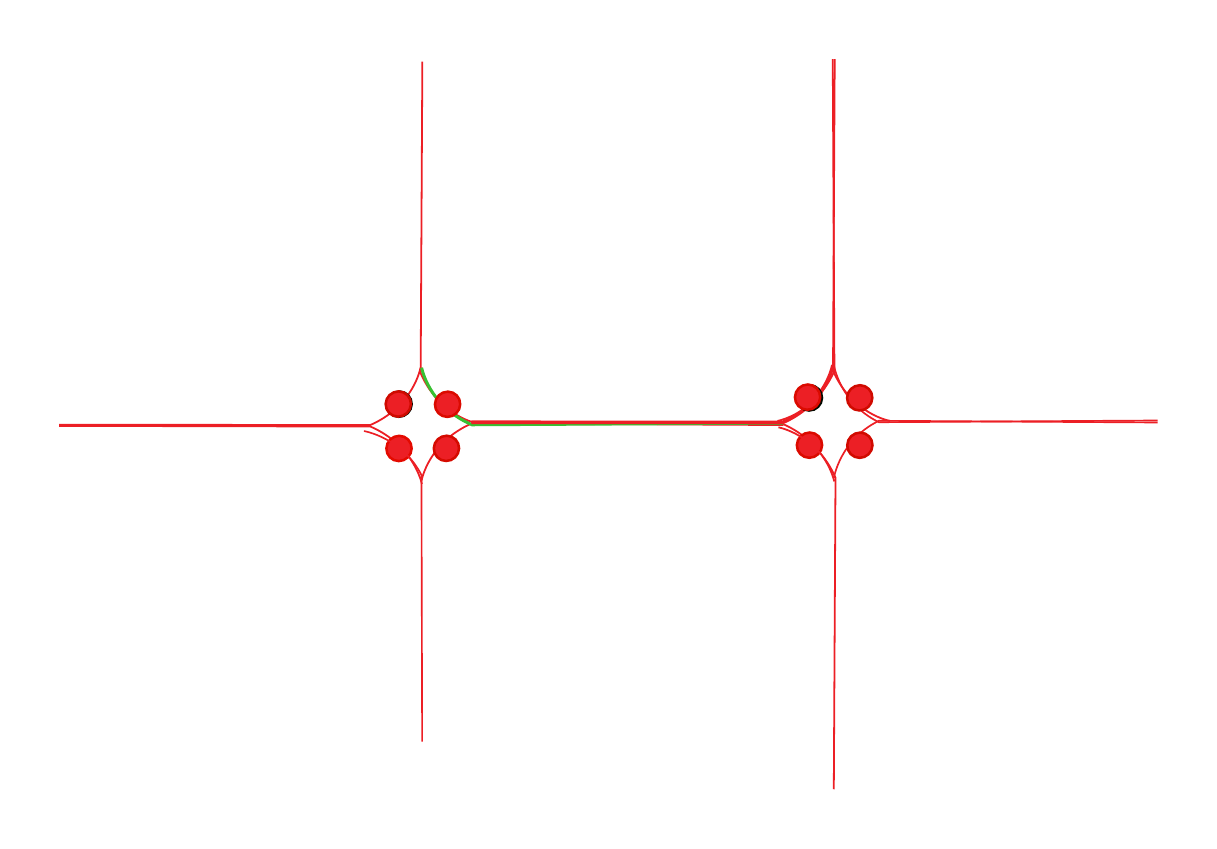} \caption{An $N_2$-piece}
	\end{figure} 
	\begin{figure}[h!] 
		\labellist 
		\small\hair 2pt 
		\endlabellist 
		\centering 
		\includegraphics[scale=0.155]{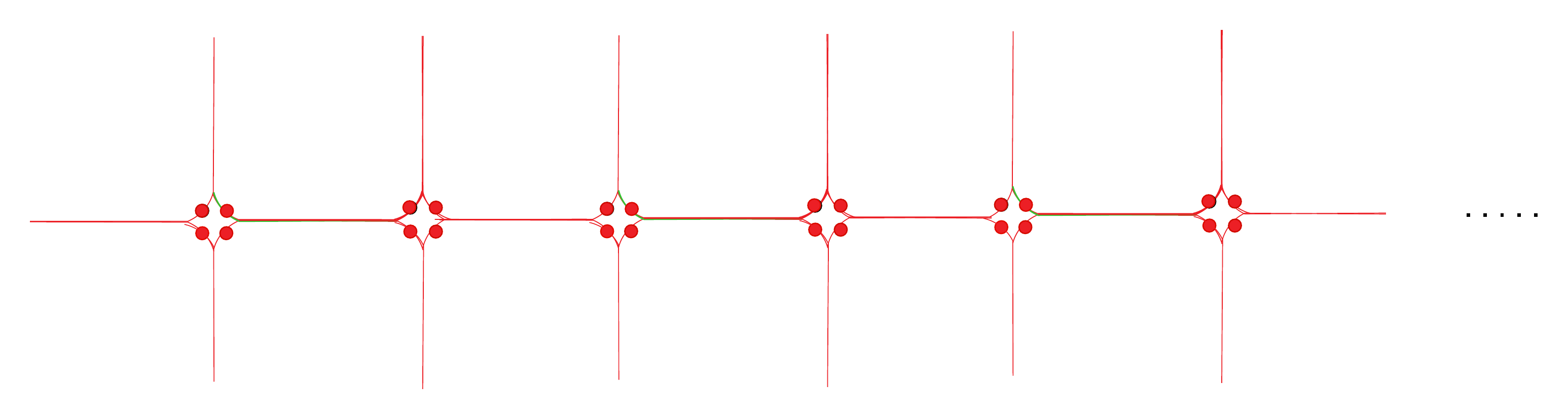} \caption{An $N_{\infty}$-piece}
	\end{figure} 
		\begin{figure}[h!]
			\labellist 
			\small\hair 2pt 
			\endlabellist 
			\centering 
			\includegraphics[scale=0.16]{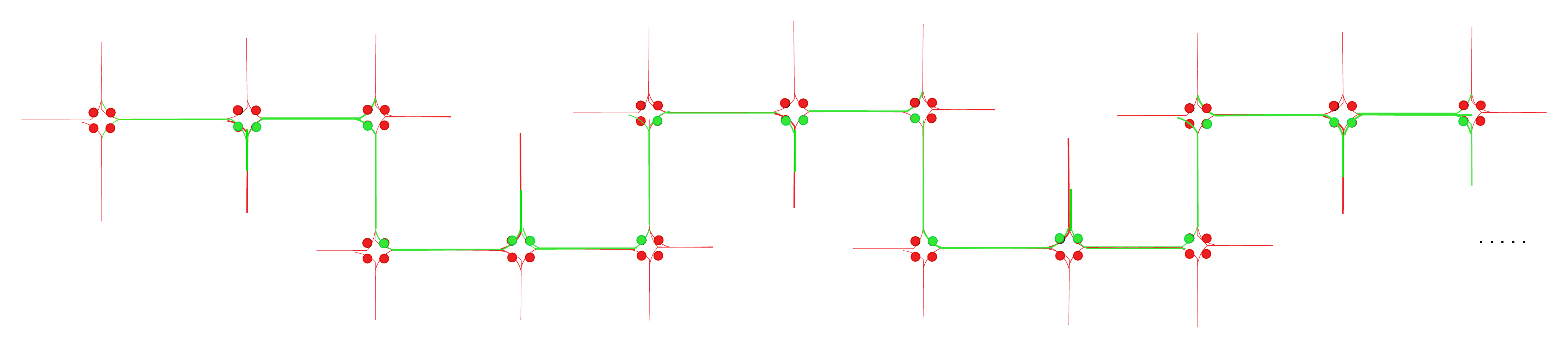} \caption{A non-geometric end containing infinitely many $N_2$-pieces}
		\end{figure}

\begin{remark}
1) A ray crossing an infinite chain of simple branchings of $\mathcal{E}$ does not define a geometric end. An end given by an infinite chain of branchings will be denoted by $\displaystyle C_{\infty}$. So if $\mathcal{E}$ has only geometric ends, it cannot contain a $\displaystyle C_{\infty}$-piece.  \\
2) A geometric end containing an infinite chain contains necessarily a semi-infinite chain that has only $N_2$-pieces, as in the Example ($\mathcal{G}$) of Figure 14. 
\end{remark}
	\begin{figure}[h!] 
		\labellist 
		\small\hair 2pt 
		\endlabellist 
		\centering 
		\includegraphics[scale=0.175]{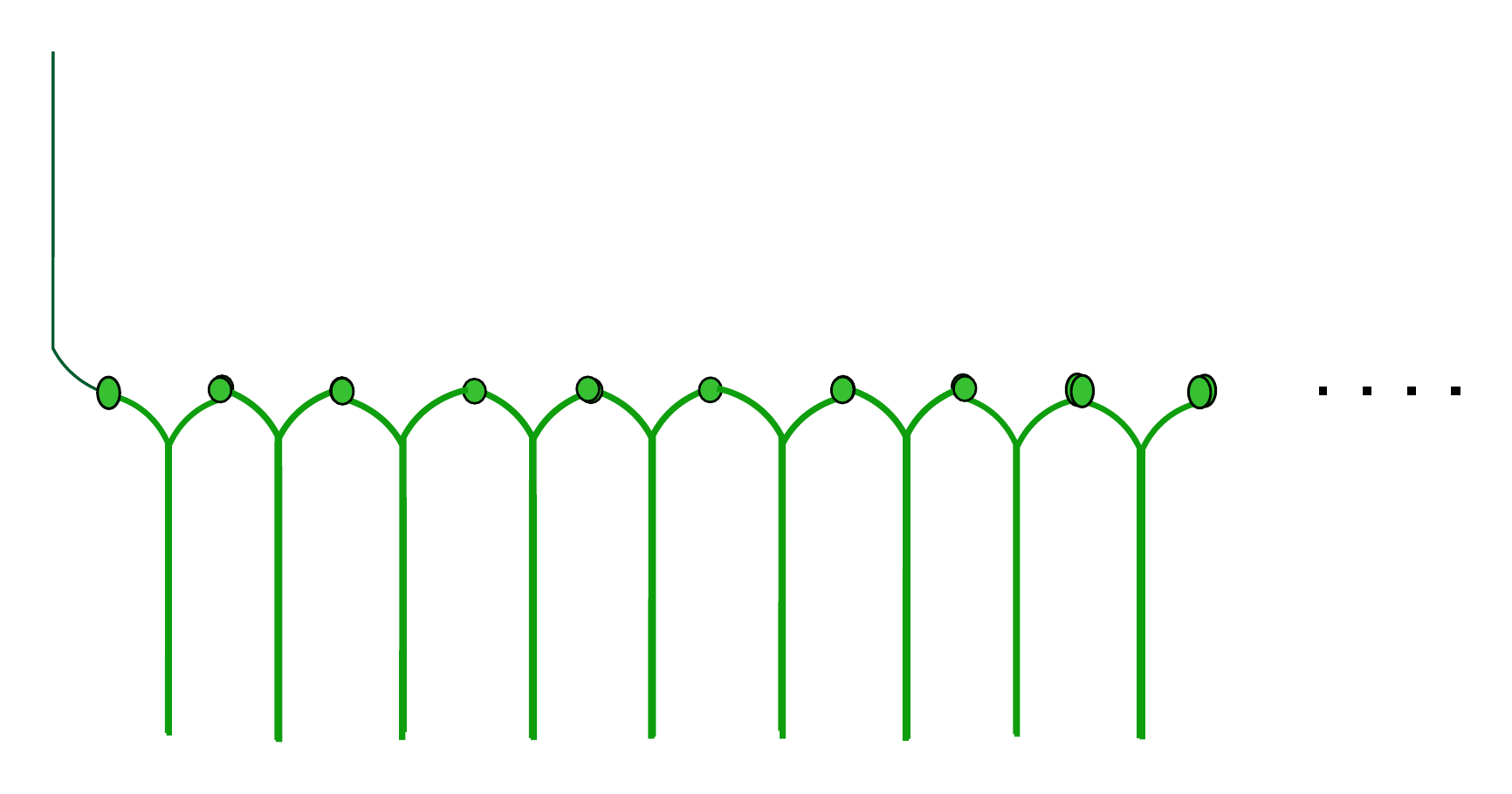} \caption{A $C_{\infty}$-piece}
	\end{figure} 

\begin{definition}\textbf{(Special end).}
Let $\mathcal{E}$ be a $1$-dimensional Riemannian manifold with topology $\mathfrak{T}$. 
A special end of $\mathcal{E}$ is either a geometric end, or a non-geometric end  containing only finite order branchings, and which is neither $N_{\infty}$ nor a chain with infinitely many $N_{2}$-pieces. 
\end{definition}
The end represented in Figure 17 is not a special end.
\begin{proposition}\label{space of leaves with condition ME}
Let $\mathcal{E}$ be a $1$-dimensional Riemannian manifold with topology $\mathfrak{T}$. 
Suppose that $\mathcal{E}$ satisfies the following condition 
\begin{center}
	$(SE)$: "all ends of $\mathcal{E}$ are special ends", 
\end{center}
then for any Lorentzian surface $X$ such that $\mathcal{E}_X \simeq \mathcal{E}$, the geodesics of $X$ remain in a ribbon after a certain while.  
\end{proposition}
\begin{proof}
Suppose there is a geodesic $\gamma$ in $X$ that leaves infinitely many ribbons; its projection on $\mathcal{E}_X$, that we denote by $\hat{\gamma}$, is a ray with infinitely many cusps that leaves any compact subset of $\mathcal{E}_X$, defining this way and end of $\mathcal{E}_X$. We claim that this end is a non-geometric end. To see this, suppose there is a geodesic ray $r$ that defines this end. Define $N(r)$ to be the sequence of elements of $\Sigma$ -taken in an increasing order- defined by $r$.   Fact \ref{Observation sur la nature des bouts --> chaines} above implies that $N(r)$ and $N(\hat{\gamma})$ admit a common subsequence. Furthermore, since  $\mathcal{G}(\mathcal{E}_X)$ is a tree and $r$ is a smooth semi-geodesic, we can say that $N(r)$ is a subsequence of  $N(\hat{\gamma})$, and the two sequences only differ along a $4$-order branching, with $\hat{\gamma}$ having two cusps on it. Since the ray $\hat{\gamma}$ comes from a geodesic of $X$, this behavior is impossible, which proves our claim. Now, since all ends of $\mathcal{E}_X$ are special ends, it follows that this end only contains $4$-order branchings, so this geodesic defines a chain in $\mathcal{E}_X$. If the geodesic lies in every maximal ribbon at most on $3$ bands (outside some compact subset of $X$), then it defines an $N_{\infty}$-piece in $\mathcal{E}_X$. Otherwise, using the fact that the norm of $K$ changes sign on two consecutive bands, we see that each time the geodesic crosses a maximal ribbon in more than $3$ bands, it defines $N_2$-pieces in the chain; so when it crosses infinitely many of them, the chain it defines contains infinitely many $N_2$-pieces.  Both cases are excluded by the "special end" condition $(SE)$, which ends the proof.
\end{proof}
\begin{remark}
The $1$-dimensional Riemannian manifolds with topology $\mathfrak{T}$ containing only geometric ends are special cases of manifolds involved in Proposition \ref{space of leaves with condition ME}. For example, a finite gluing of ($\mathcal{G}$) manifolds gives many examples of them. 
\end{remark}
\begin{theorem}\label{Complétudes: conditions sur les bouts spéciaux}
Let $(X,K)$ be a simply connected null-complete Lorentzian surface with a Killing field $K$. Suppose that	$\mathcal{E}_X$ satisfies the condition $(SE)$ in Theorem \ref{space of leaves with condition ME} above. Then  $X$ is complete if and only if all the ribbons of $X$ satisfy the (SRC) condition of Lemma \ref{Complétude dans un ruban}.
\end{theorem}
\begin{proof}
According to Proposition \ref{space of leaves with condition ME}, any geodesic of $X$ remains in a ribbon after a certain while, hence $X$ is complete if and only if all the ribbons  in $X$ satisfy the (SRC) condition, by Lemma \ref{Complétude dans un ruban}.
\end{proof}
\begin{corollary}
Let $(X,K)$ be a simply connected null-complete Lorentzian surface with a Killing field $K$. Suppose that 
\begin{itemize}[topsep=0pt, itemsep=0pt, parsep=0pt]
	\item $\mathcal{E}_X$ satisfies the condition $(SE)$ in Theorem \ref{space of leaves with condition ME} above,
	\item the norm of $K$ is bounded.
\end{itemize}
Then $X$ is complete.
\end{corollary}
\begin{proof}
Since $X$ is null complete and the norm $K$ is bounded, the ribbons in $X$ satisfy the (SRC) condition (see the last conclusion of Lemma \ref{Complétude dans un ruban}). We conclude using Theorem \ref{Complétudes: conditions sur les bouts spéciaux} above.
\end{proof}
\subsection{Completeness results with mixed topological and geometrical conditions}
\paragraph{Definition of finite-size ends in terms of non-convergent Cauchy sequences.}
Let $\mathcal{E}$ a $1$-dimensional Riemannian manifold with topology $\mathfrak{T}$. The Riemannian structure on this space defines a pseudometric topology on it; we denote by $\delta$ the induced pseudometric on $\mathcal{E}$. In the following, we say that $(\mathcal{E},\delta)$ is complete if any Cauchy sequence on this space converges (for the induced pseudometric topology). \\
The relation $p \sim q$ if $\delta(p,q) =0$ is an equivalence relation on the set $B$ of branch points; in what follows, the class of a branch point of $ \mathcal{E}$ is considered with respect to this relation. 

\begin{definition}[\textbf{Finite-size end}]
	Let $\mathcal{E}$ a $1$-dimensional Riemannian manifold with topology $\mathfrak{T}$. An end of $\mathcal{E}$ is said to be of finite size if there is a ray in the equivalence class of rays defining this end (see Definition \ref{Définition des bouts avec les géodésiques brisées} and Fact \ref{Observation sur la nature des bouts --> chaines}), which is of finite length. 
\end{definition}

\begin{lemma} Let $\mathcal{E}$ a $1$-dimensional Riemannian manifold with topology $\mathfrak{T}$.\\
	1) A non-convergent injective Cauchy sequence in $\mathcal{E}$ defines an end in this space.\\
	2) If $\mathcal{E}$ has an end of finite size different from $\displaystyle C_{\infty}$, then there is a non-convergent Cauchy sequence  in $\mathcal{E}$.
\end{lemma}
\begin{proof}
	1) Denote by $\mathcal{T}_1$ the topology of $\mathcal{E}$, and by $\mathcal{T}_2$ the topology induced by the pseudometric $\delta$. Since the latter is continuous with respect to the $\mathcal{T}_1$ topology, then any compact set for the $\mathcal{T}_1$ topology is compact for the $\mathcal{T}_2$ topology. Let $(b_n)_n$ be a non-convergent injective Cauchy sequence in $\mathcal{E}$. If $(b_n)_n$ is contained in a $\mathcal{T}_1$-compact set $V$ then it converges, for $V$ is also $\mathcal{T}_2$-compact, a contradiction. Consequently, this sequence leaves any $\mathcal{T}_1$-compact set in $\mathcal{E}$. Now, we have to see that for any $\mathcal{T}_1$-compact set $K$, this sequence gives a well defined connected component of $\mathcal{E} \smallsetminus K$. First, observe that since $(b_n)_n$ does not admit a limit point, for any branch point $z$ belonging to a $C_{\infty}$-piece in $\mathcal{E}$, there exists $r>0$ such that   $b_n$  is contained in $\mathcal{E} \smallsetminus  B_r(z)$ for all $n \in \N$, where $B_r(z) $ is the open ball of center $z$ and radius $r > 0$.  

    Suppose now that there are two subsequences $(b_{\phi(n)})_{_n}$ and $(b_{\psi(n)})_{_n}$, such that $b_{\phi(n)} \in E_1$ and $b_{\psi(n)} \in E_2$, $\forall n \in \N$, where $E_1$ and $E_2$ are two different connected components of $\mathcal{E} \smallsetminus K$. Since $K$ is compact, it contains finitely many $\mathcal{Y}$-pieces, hence finitely many branch points $z_1,...,z_m$ from $C_{\infty}$-pieces.   For all $i=1,..,m$, there exists $r_i>0$, such that the ball $B(z_i,r_i)$ does not contain $b_n$ for all $n \in \N$. Consider the set $K':= K \cup \bigcup_{i=1..m} B(z_i,r_i/2)$. We know that for $n$ large enough, $b_n$ is not in $K'$, so there exist two connected components $E_1'$ and $E_2'$ of $\mathcal{E} \smallsetminus K'$ such that $E_1' \subset E_1$ and $E_2' \subset E_2$, containing  $(b_{\phi(n)})_{_n}$ and $(b_{\psi(n)})_{_n}$ respectively. Let $p: \mathcal{E} \to \mathcal{G}(\mathcal{E})$ be the natural projection. The sets $p(E_1')$ and $p(E_2')$ are disjoint by construction, so a ray connecting any two points from these two sets intersects $p(K')$. Consider the (unique) geodesic ray in the real tree $\mathcal{G}(\mathcal{E})$, joining $p(b_{\phi(n)})$ and $p(b_{\psi(n)})$; it intersects  $p(K')$ in some point $a_n$. We have
    \begin{align*}
    d(b_{\phi(n)}, b_{\psi(n)})&=d(b_{\phi(n)},a_n)+d(b_{\psi(n)},a_n)\\
    &\geq d(b_{\phi(n)},K') + d(b_{\psi(n)},K').
    \end{align*}
   Since $p(K)$ is compact, there exist $N, N' \in \N$ such that 
   \begin{align*}
\epsilon_1 := d( (b_{\phi(n)})_{_{n \geq N}}, p(K)) >0  \textrm{\;and\;} \epsilon_2 := d( (b_{\psi(n)})_{_{n \geq N'}}, p(K)) >0.
    \end{align*}
In consequence, $d(b_{\phi(n)}, b_{\psi(n)}) \geq \epsilon >0, \forall n \geq \sup(N,N')$, with $\epsilon := \inf\{ \displaystyle \min_{i=1..m} r_i/2;\, \epsilon_1;\, \epsilon_2\}$.
     
	2) Consider a finite-size end in $\mathcal{E}$ which is not a $\displaystyle C_{\infty}$-piece, and fix a ray $r$ that defines this end. If $r$ has finitely many cusps, then the end of it is diffeomorphic to an open and bounded interval $I$ of $\R$, so it is does contain a non-convergent Cauchy sequence. If $r$ has infinitely many cusps,  it's clear that the sequence of branch points $(b_n)_n$ contained in $r$ gives an injective Cauchy sequence in $\mathcal{E}$. We claim that this sequence does not converge. To prove this, suppose contrary to our claim that there exists $b_{\infty} \in \mathcal{E}$ such that $(b_n)_n$ converges to $b_{\infty}$.  The sequence $(p(b_n))_n$, the projection of $(b_n)_n$ in $\mathcal{G}(\mathcal{E})$, does not contain any constant subsequence, for otherwise $r$ defines a $C_{\infty}$-piece, contrary to our assumption. It follows that $(p(b_n))_n$ does not contain any constant subsequence, hence the minimal ray contained in $p(r)$ is a smooth geodesic in $\mathcal{G}(\mathcal{E})$ that converges to $p(b_{\infty})$. Consequently, this geodesic can be extended beyond this point; lifting this geodesic to a continuous path in $\mathcal{E}$ gives a ray that extends $r$, which is excluded since $r$ is maximal.  
\end{proof}
\begin{theorem}\label{Complétude à f bornée et espace des feuilles complet, général}
Let $(X,K)$ be a simply connected and null complete Lorentzian surface with Killing field $K$.  Assume that $F$ is bounded, and that $\mathcal{E}_X$ is complete and contains no $C_{\infty}$-piece. Then $X$ is complete.
\end{theorem}
\begin{proof}
Geodesics remaining in a maximal ribbon after a certain while are complete by Lemma \ref{Complétude dans un ruban}. Let $\gamma$ be a geodesic that leaves any maximal ribbon. 
Denote by $(\Delta_i)_i$ the sequence of elements of $\Sigma$ crossed by $\gamma$. If $ \sum_i |\Delta_i| < + \infty$, then  $\gamma$ defines an end of $\mathcal{E}_X$ of finite size different from $C_{\infty}$, hence gives rise to a non convergent Cauchy sequence in $\mathcal{E}_X$ by the previous lemma, contrary to our assumption. It follows that  $ \sum_i |\Delta_i|$ is infinite, and since $F$ is bounded,  one can write
\begin{align}\label{Minorer le temps, ici seulement, f bornée}
	t_i = \int_{\Delta_{i}} \frac{dx}{\sqrt{C^2- \epsilon f(x)}} \geq \frac{|\Delta_{i}|}{\sqrt{N+C^2}},
\end{align}
where $N$ is an upper bound on $|F|$, proving that the geodesic is complete. 
\end{proof}
\begin{remark}
In the previous theorem, the hypothesis on the absence of $C_{\infty}$-pieces can be weakened by assuming instead that $F$ is bounded outside of $C_{\infty}$-pieces, and that for any $C_{\infty}$-piece $\mathfrak{C}$ in $\mathcal{E}_X$:
\begin{itemize}[topsep=0pt, itemsep=0pt, parsep=0pt]
	\item $F$ is bounded on the geometric ends of $\mathfrak{C}$,
	\item if $F$ has constant sign on the non-geometric end of $\mathfrak{C}$, then $F$ is not bounded on $\mathfrak{C}$. 		
\end{itemize}
In this case, a geodesic contained in some $C_{\infty}$-piece after a certain while remains necessarily in a maximal ribbon.
\end{remark}
In Example \ref{Exemple 2: f bornée, espace des feuilles complet, X incomplète}, we obtain a simply connected and null complete Lorentzian surface $(X,K)$ with a Killing field $K$,  where $\mathcal{E}_X$ is complete and the norm of $K$ is bounded, but $X$ is incomplete. In this example, $\mathcal{E}_X$ contains an infinite chain of simple branchings, and the incompleteness of $X$ comes from those geodesics that cross the infinite chain. 
\subsection{Examples and non-examples}\label{section: examples and non-examples}
In this section, we illustrate various completeness behaviours for the geodesics of a null complete surface (even on the same surface), which may appear when omitting certain conditions in the completeness results of Section~\ref{Section: complétude}. This illustrates the complexity of the question. All the surfaces obtained in this section are L-complete, hence maximal.

We call an "infinite branch" of $\mathcal{E}$ an element of $\Sigma$ near infinity.\\
 
A symmetric saddle is obtained in \cite{BM}, Proposition 2.29, as the extension of a domino (whose unique null orbit of $K$ is incomplete) by a simply connected surface containing a unique zero of $K$, whose metric is symmetric with respect to $p$. This extension is unique up to isometry (\cite[Proposition 2.37]{BM}). This extends the null orbits of $K$ which are geodesically incomplete into complete geodesics.
\begin{example}[\textbf{A complete saddle}] 
\end{example}

An easy example of a complete saddle is given by the Minkowski plane. To obtain a non-flat saddle, consider a symmetric saddle defined as the extension of a domino $(U,K)$, such that the norm of the Killing field on $U$ induces a bounded function $f$ defined on $\R$; for instance: $f(x)= \arctan x$. The saddle we obtain is complete (see Corollary \ref{Complétude d'une selle L-complète}). 

\begin{example}[\textbf{A null complete saddle -hence maximal- but not complete}]\label{Exemple 2.3.19}
\end{example}
If we consider now a symmetric saddle defined as the extension of a domino $(U,K)$, such that the norm of $K$ on $U$ induces a function $f$ defined on $\R$ by $f(x):=e^x-1$, it follows from Corollary \ref{Complétude d'une selle L-complète} that the saddle is incomplete. In this example, spacelike geodesics are complete and timelike geodesics are incomplete; among the latter, some are semi-complete. If now $f$ is defined by $f(x):=\sinh x$, then we get a null complete saddle whose non-null geodesics are all  incomplete.

\begin{example}\label{Exemple 0}
\end{example}
Let $(X,K)$ be a small surface. Distinguished geodesics of $\mathcal{E}_X$ have an infinite branch from both sides. Since any geodesic of  $\mathcal{E}_X$ is contained in an infinite branch after a certain while, the null completeness of $\mathcal{E}_X$ implies its completeness. 
However $X$ could contain incomplete geodesics: it is sufficient, by Lemma \ref{Complétude dans un ruban} to set  $ f(x) = x^{\alpha}, \alpha > 2 $ on one of the infinite branches.

\begin{example}[\textbf{A simply connected null complete surface, all of whose timelike and spacelike geodesics are incomplete, except orthogonal geodesics}]\label{exemple selle incomplète et orthogonal complète}   
\end{example}
Let $(I_n)_{n \geq 0}$ be a sequence of pairewise disjoint intervals of $\R^{+}$ such that $I_0=]1,2[$, $ \sum_{n} |I_n| = d < \infty$, and the distance between $I_n$ and $I_{n+1}$ is equal to $1$, for all $n$. Denote by $(J_n)_{n\geq0}$ the sequence of intervals given by the connected components of $\R^+ \smallsetminus \cup I_n$. We have $ \forall n \in \N, |J_n|=1$. Consider the following piecewise constant function $\phi$:

$$
\phi(x) = \left\{
\begin{array}{ll}
n^3 & \mbox{if\;} x \in J_n, \\
\frac{1}{n^2} &\mbox{if\;} x \in I_n.
\end{array}
\right.
$$
We have 
\begin{enumerate}[itemsep=0pt, topsep=0pt, parsep=0pt]
	\item for all $M>0$, $\mu \{ \phi \leq M\} <  \infty$;
	\item for all $\alpha >0, \exists \delta_{\alpha} >0$, $\displaystyle \int_{\{\phi \geq \alpha\} \cap \R^{+}} \frac{dx}{\sqrt{\phi(x)}} = \delta_{\alpha} + \displaystyle \sum_n \frac{1}{n^{3/2}} < \infty$, ;
	\item $\displaystyle \int_{\R^{+}} \frac{dx}{\sqrt{\phi(x)}} = \displaystyle \sum_n \frac{1}{n^{3/2}} + d \displaystyle \sum_n n = \infty$.
\end{enumerate}
Let $\tilde{\phi}$ be a smooth function sufficiently close to $\phi$, which satisfies these three properties. Consider now a symmetric saddle defined as the extension of a domino $(U,K)$, such that the norm of the Killing field on $U$ induces a function $f$ defined on $\R$ by: 
$$
f(x) = \left\{
\begin{array}{ll}
\tilde{\phi}(x) & \mbox{if\;} x \geq 0, \\
-\tilde{\phi}(-x) &\mbox{if\;} x <0.
\end{array}
\right.
$$
The assertions (1) and (2) imply that timelike and spacelike geodesics of $S$ not orthogonal to the Killing field are incomplete, by Corollary \ref{Complétude d'une selle L-complète} and \ref{C critique selle}, whereas geodesics orthogonal to $K$ are complete by assertion (3). 
\begin{example}[\textbf{A simply connected surface all of whose geodesics are complete, except timelike geodesics from a certain point}]\label{Exemple dérivée bornée 1er}
\end{example}
Let $\mathcal{E}$ be a $1$-dimentional manifold homeomorphic to $N_{\infty}$.  Denote by $(\sigma_i)_i$ the (unique) infinite normal sequence of $\mathcal{E}$. Consider a local diffeomorphism $\textbf{x} \in C^{\infty}(\mathcal{E},\R)$ such that:
\begin{itemize}[itemsep=0pt, topsep=0pt, parsep=0pt]
	\item   $\textbf{x}(\sigma_0)=]-\infty,1[$, $\textbf{x}(\sigma_1)=]0,1[$,
	\item  $\forall i \in \N^*, |\sigma_{2i}|=|\sigma_{2i+1}|=1/(2i)^4$, so that $\textbf{x}(\sigma_{2i})=\textbf{x}(\sigma_{2i+1})=]0,1/(2i)^4[$,
\item $\textbf{x}$ is unbounded on infinite branches of $\mathcal{E}$.
\end{itemize}
Consider $F \in C^{\infty}(\mathcal{E},\R)$, which is a negative function on infinite branches of $\mathcal{E}$, except on $\sigma_0$, and with bounded transverse derivative on them. Denote by $f_i$ the restriction of $F $ to $\sigma_i$ composed with $ \textbf{x}^{-1}$, and assume further that 
\begin{itemize}[itemsep=0pt, topsep=0pt, parsep=0pt]
	\item $f_{0}$ has bounded derivative,
	\item $\forall i \in \N^*, \forall x \in ]0,1/(2i)^4[,  f_{2i}(x)=f_{2i+1}(x) := \frac{1}{(2i)^4}f_1((2i)^4x)$.
\end{itemize}
Let $(X,K)$ be a simply connected Lorentzian surface such that $\mathcal{E} \simeq \mathcal{E}$; it is null complete. It's clear that $F$ has bounded transverse derivative, and in consequence, all geodesics not orthogonal to $K$ are complete (Theorem \ref{Complétude à dérivée bornée}). Spacelike geodesics orthogonal to $K$ are contained in a ribbon after a certain while, so they are complete by the last conclusion of Theorem \ref{Complétude à dérivée bornée}. Finally, we see that the series defined in (\ref{Complétude à dérivée bornée, condition sur les géod orthogonales}) involving the infinite normal sequence of $\mathcal{E}_X$ converges in this case, so that timelike geodesics orthogonal to $K$ are incomplete. Note that the latter are all equivalent by the isometric flow.  We see that these conslusions are independent of the choice of the linking structure on $\mathcal{E}$. 
\begin{example}[\textbf{A simply connected surface $(X,K)$ all of whose geodesics are complete, except timelike and spacelike geodesics orthogonal to $K$}]\label{Exemple dérivée bornée 2ème}
\end{example}
Let $\mathcal{E}$ be a $1$-dimensional manifold such that $\mathcal{G}(\mathcal{E})$ is the Cayley graph of the free group $G$ on two generators ${\displaystyle a}$ and ${\displaystyle b}$ corresponding to the set ${\displaystyle S=\{a,b,a^{-1},b^{-1}\}}$. Travelling along an edge to the right represents right multiplication by ${\displaystyle a}$, while travelling along an edge upward corresponds to the multiplication by ${\displaystyle b}$. Since the free group has no relations, the Cayley graph is a tree. The horizontal edges are $\{g,ga\}$ and the vertical ones are $\{g,gb\}$, $g \in G$. Fix a sequence $\alpha_n=1/n^4, n \geq 1$. For every $g \in G$, we denote by $\#[a,a^{-1}]_g$ the number of elements $a$ and $a^{-1}$ in $g$, and $\#[b,b^{-1}]_g$ the number of elements $b$ and $b^{-1}$ in $g$. Define $N_g:=\#[a,a^{-1}]_g - \#[b,b^{-1}]_g +1$, and suppose that the length of the edges $\{g,ga\}$ and $\{g,gb\}$ is given by $\alpha_{|N_g|}$. We have $ \forall i \in \N, N_{ga^{\pm i}} = N_g +i$ and $N_{gb^{\pm i}} = N_g -i$.
A distinguished maximal semi-geodesic in $\mathcal{E}$, for some linking structure on it, corresponds to a word in $G$ of the form $g_0\, z^{\epsilon_1}\, \nu(z)^{\epsilon_2}\, z^{\epsilon_3}\, \nu(z)^{\epsilon_4}...$, where $g_0 \in G, z \in \{a,b\}, \epsilon_i \in \{-1,+1\}$, and $\nu$ is the non-trivial permutation of the set $\{a,b\}$. 
So we see that the sequence $|N_g|$ on that geodesic is either stationary or has two stationary subsequences. Therefore, if $(\Delta_i)_i$ is the sequence of elements of  $\Sigma$ belonging to some distinguished maximal semi-geodesic on $\mathcal{E}$, for some linking structure on it, then
\begin{align}\label{Exemple dérivée bornée 2ème, divergence null complète}
\displaystyle \sum_i |\Delta_i| = +\infty.
\end{align} 
Now let $(X,K)$ be a simply connected Lorentzian surface with Killing field $K$, such that $\mathcal{E}_X = \mathcal{E}$, and the pseudometric induced by the Riemannian structure on $\mathcal{E}_X$ coincides with that defined above. The surface is null complete by (\ref{Exemple dérivée bornée 2ème, divergence null complète}). \\
\\
For every $g \in G$, denote by $\sigma^a_g$ (resp. $\sigma^b_g$) the element in $\Sigma$ corresponding to the edge $\{g,ga\}$ (resp. $\{g,gb\}$). 
Denote by $f_1$ some $C^{\infty}$ function on $[0,1]$ with bounded derivative, such that derivatives of all orders at $0$ and $1$ vanish.
For $ g \in G$, set $\lambda_g:=\textbf{x}(\{g\})$, where $\{g\}$ is the vertex of $\mathcal{G}(\mathcal{E})$ corresponding to $g$. Define
\begin{align*}
&\forall x \in \textbf{x}(\sigma^{a}_{g})=]\lambda_g, \lambda_g+\frac{1}{\alpha_{|N_g|}^{4}}[,\;\; f^{a}_{g}(x) := \frac{1}{\alpha_{|N_g|}^{4}}f_1(\alpha_{|N_g|}^{4} (x-\lambda_g)), \\
&\forall x \in \textbf{x}(\sigma^{b}_{g})=]\lambda_g-\frac{1}{\alpha_{|N_g|}^{4}}, \lambda_g[, \;\;f^{b}_{g}(x) := \frac{1}{\alpha_{|N_g|}^{4}}f_1(\alpha_{|N_g|}^{4} (x+\lambda_g)).
\end{align*}
 
Suppose that the norm of $K$ is given by a function $F \in C^{\infty}(\mathcal{E},\R)$, whose restriction to $\sigma_g^a$ (resp. $\sigma_g^b$), composed with $\textbf{x}^{-1}$, is given by $f_g^a$ (resp. $f_g^b$). 

$F$ has bounded transverse derivative, so as in the Example \ref{Exemple dérivée bornée 1er} above, all geodesics not orthogonal to $K$ are complete, whereas spacelike and timelike orthogonal geodesics are all incomplete. 
\begin{example}[\textbf{complete} $X$\textbf{, incomplete} $\mathcal{E}_X$]\label{Exemple: l'importance du choix de linking structure}
\end{example}
Let $\mathcal{E}$ be a $1$-dimensional smooth manifold with topology $\mathfrak{T}$, and let $\eta$ (represented in green in the figure below) be a geodesic ray in $\mathcal{E}$ homeomorphic to an interval $[a,b[$, such that $\eta \cap B = \{b_n, n \in \N\}$, with $(b_n)_n$ converging to $b$. Suppose that $B$ is locally finite except along $\eta$. Let $\mathcal{A}$ be a linking structure on $\mathcal{E}$ such that any distinguished geodesic shares with $\eta$ at most one branch point. Take $\textbf{x} \in C^{\infty}(\mathcal{E},\R)$ a local diffeomorphism. Since $B$ is locally finite except on $\eta$, it is possible to change $\textbf{x}$ so that $\mathcal{A}$-distinguished geodesics are complete, and $\eta$ is incomplete, so henceforth we assume that $\textbf{x}$ satisfies this property.  Corollary \ref{Corollaire 1 (cas complet): caractérisation top de l'espace des feuilles} ensures the existence of a smooth and simply connected Lorentzian surface $(X,K)$ such that $\mathcal{E}_X \simeq \mathcal{E}$, which is complete. It appears that among complete surfaces, the space of leaves may be incomplete. 
 
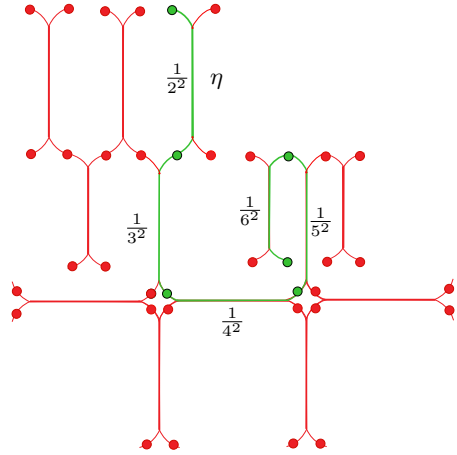
\begin{figure}[h!] 
	\labellist 
	\small\hair 2pt 
	\pinlabel $\eta$ at 276 497 
	\pinlabel $\frac{1}{2^2}$ at 223 500 
	\pinlabel $\frac{1}{3^2}$ at 169 298 
	\pinlabel $\frac{1}{4^2}$ at 297 172 
	\pinlabel $\frac{1}{5^2}$ at 416 307 
	\pinlabel $\frac{1}{6^2}$ at 321 322 
	\endlabellist 
	\centering 
	\includegraphics[scale=0.29]{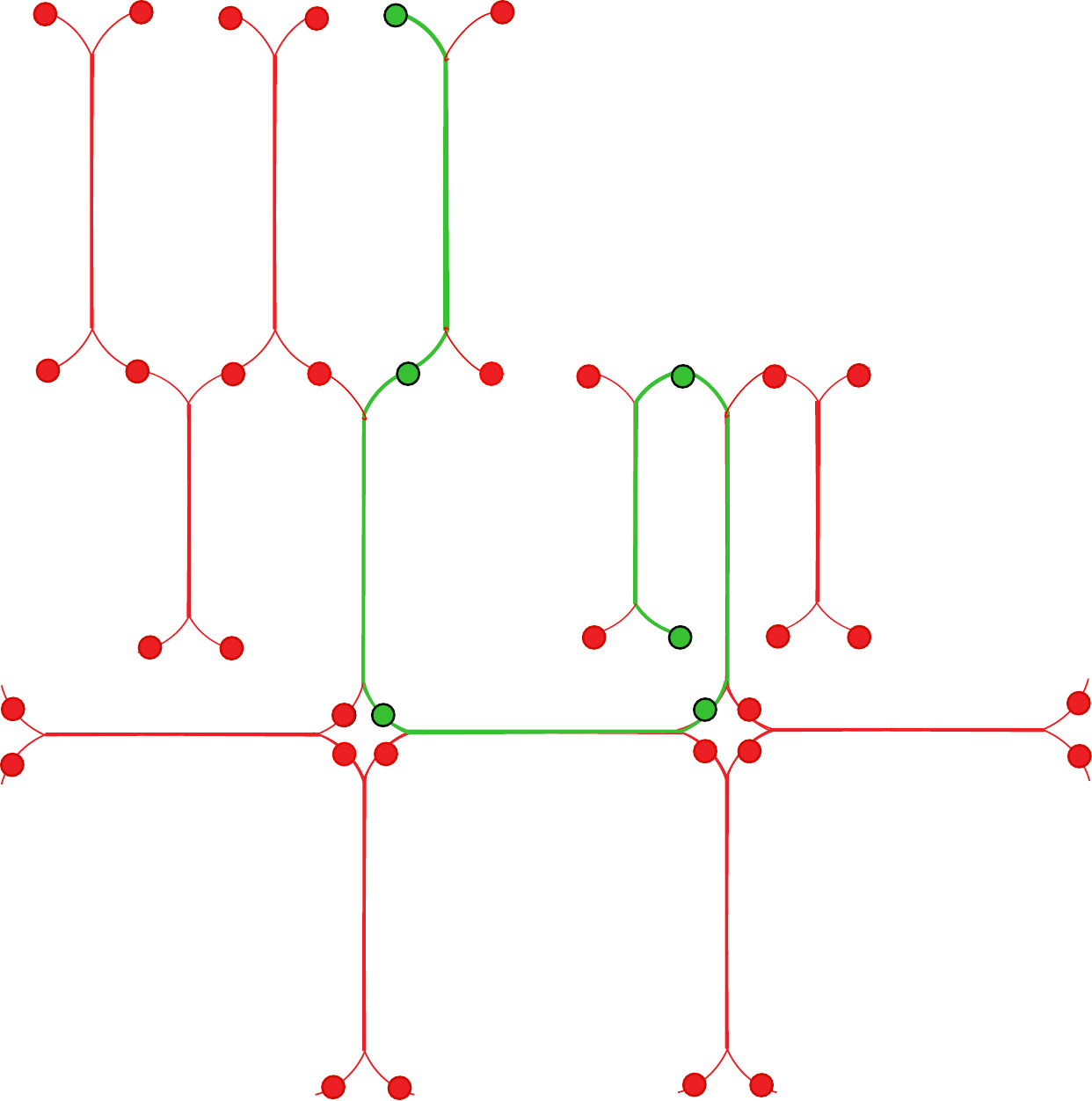} \caption{Example \ref{Exemple: l'importance du choix de linking structure}: the space of Killing orbits $\mathcal{E}_X$ is incomplete}
\end{figure} 

The closure in $\eta$ of every element of $\Sigma$ is the space of leaves of a type II band in $X$, so a geodesic of $X$ goes through a piece of $\eta$ at most on two bands. So it is clear that the completeness of $ \eta $ has no consequence on that of the surface. 
It is then possible to modify the geodesic completeness of $ \mathcal{E} $ by changing $\textbf{x}$ on $\eta$ without affecting the completeness of $ X $. \\
Another interesting fact is that according to the linking structure we choose  on $\mathcal{E}$, the surface can be null complete or not. So the hypothesis of the existence of a linking structure such that the distinguished geodesics are complete in Corollary \ref{Corollaire 1 (cas complet): caractérisation top de l'espace des feuilles} is not superfluous.\\

\begin{example}[\textbf{incomplete} $X$\textbf{, complete} $\mathcal{E}_X$]\label{Exemple 2: f bornée, espace des feuilles complet, X incomplète}
\end{example}
Suppose that $\mathcal{E}_X$ contains a $C_{\infty}$-piece, as in the figure below. Assume further that on each element of $\Sigma$ contained in $C_{\infty}$, the norm of the Killing field is given by $$f_i(x)= \sin^2(\frac{1}{\sqrt{x-x_i}}), \mathrm{\;\;for\;\;} x \in I_i:=[\frac{1}{(\pi i+\pi)^2}-x_i,\frac{1}{(\pi i)^2}-x_i], \mathrm{\;\;with\;\;} x_i:=\frac{1}{(\frac{\pi}{2}+\pi i)^2}.$$
The sequence of the lengths of the bands is thus given by $(\Delta_{x_i})_i =(\frac{1}{(\pi i)^2} -\frac{1}{(\pi i+\pi)^2})_i$, whose sum converges.  Finally, assume that the length of the bands near infinity is infinite, and that outside these bands and those of $ C_{\infty} $, the length of the bands is equal to $ 1 $. The space of Killing orbits thus obtained is complete.

Now let $ \gamma $ be a spacelike  geodesic that goes through the space of Killing orbits $ \mathcal{E}_X $ along the green piece represented in Figure 21 above. When $ \gamma $ turns around in $ \mathcal{E}_X $, it means that it is tangent to a leaf of the Killing field (see the figure above).
\begin{figure}[h!] 
	\labellist 
	\small\hair 2pt 
	\endlabellist 
	\centering 
	\includegraphics[scale=0.26]{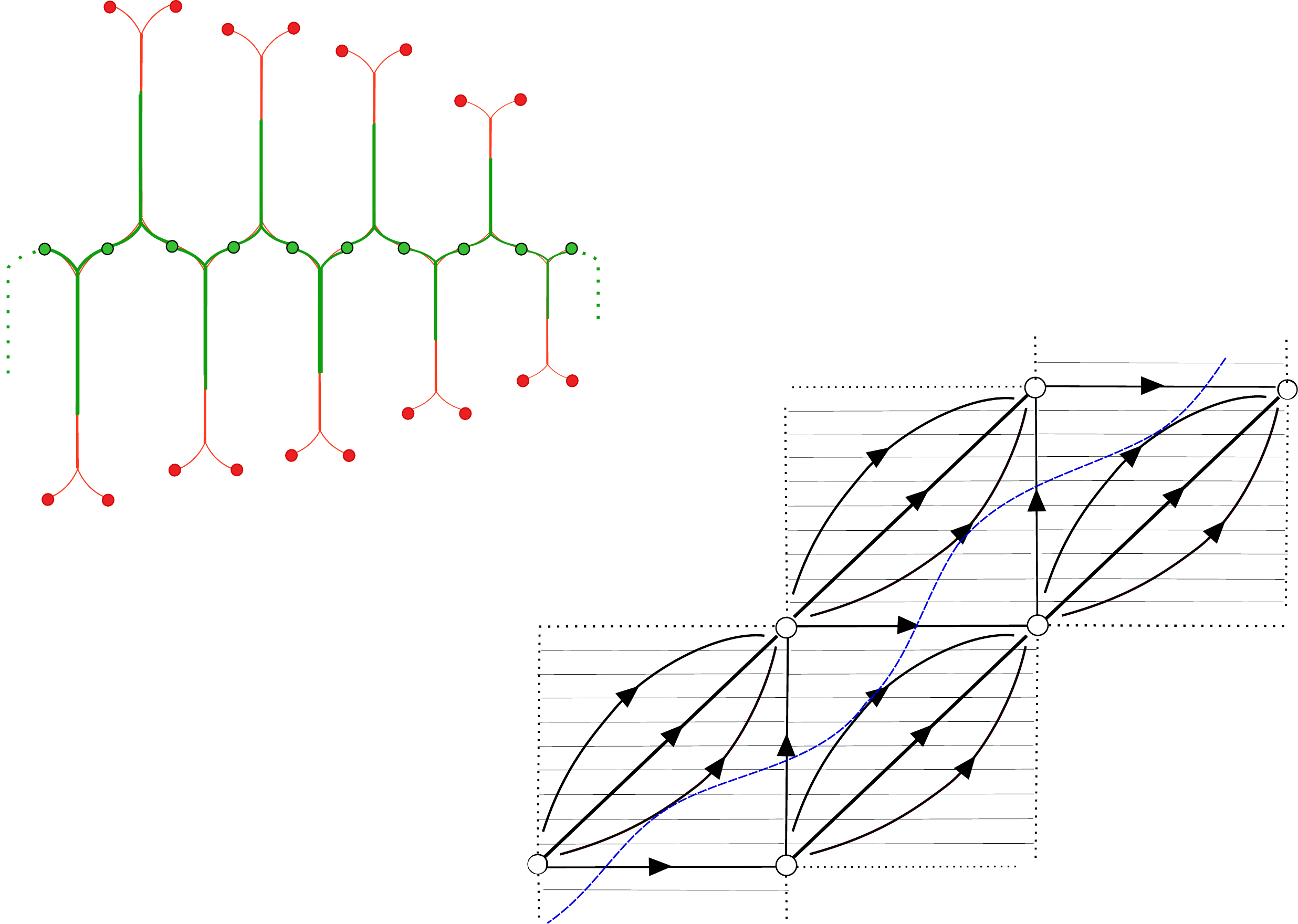} \caption{$C_{\infty}$:  an infinite chain of simple branchings in the space of Killing orbits}
\end{figure}

We claim that $\gamma$ is incomplete. Indeed, denote by $t_i$ the time $\gamma$ takes to cross each band of $C_{\infty}$; we have $t_i \leq 2 \int_{y_i}^{x_i} \frac{dx}{\sqrt{g_i(x)}}$, where $y_i =  \frac{1}{(\pi i+\pi)^2}$, $x_i =  \frac{1}{(\pi i+\pi/2)^2}$ and $g_i(x)=C^2-f_i(x)$. Now write  $|g_i(x)| = |g_i^{'}(a_i)||x-z_i|$, with $a_i \in ]y_i,x_i[$ and $g_i(z_i)=0$. It follows that for all $i$, $$\int_{y_i}^{x_i} \frac{dx}{\sqrt{g_i(x)}} \leq \frac{2 \sqrt{x_i-y_i}}{\sqrt{g_i^{'}(a_i)}}.$$
Since $g_i$ is uniformly bounded from below on a neighborhood of the  null orbits of the Killing field, $\gamma$ crosses this neighborhood with finite length; taking this into account, we are reduced to considering   a sequence $ | g_i^{'}(z_i) | $ which tends to infinity, by isolating the neighborhoods of the zeros of the norm. There is therefore a constant $ M> 0 $ such that $ \int_{y_i}^{x_i} \frac{dx}{\sqrt{g_i(x)}} \leq 2 M \sqrt{x_i-y_i} $. As a result, the sum of the $ t_i$'s is bounded from above  by the series $ \sum_i 4 M \sqrt{x_i-y_i} $ which converges by definition; this proves that $ \gamma $ is incomplete.
\newpage
\bibliographystyle{abbrv}
\bibliography{Bibliographie}

\begin{thebibliography}{10}

\bibitem{BM}
C.~Bavard and P.~Mounoud.
\newblock Extensions maximales et classification des tores lorentziens munis
  d'un champ de killing.
\newblock {\em arXiv: 1510.01253; to appear in Ann. Inst. Fourier (Grenoble)}.

\bibitem{MR3039767}
C.~Bavard and P.~Mounoud.
\newblock Sur les surfaces lorentziennes compactes sans points conjugu\'{e}s.
\newblock {\em Geom. Topol.}, 17(1):469--492, 2013.

\bibitem{MR979369}
Y.~Carri\`ere.
\newblock Autour de la conjecture de {L}. {M}arkus sur les vari\'{e}t\'{e}s
  affines.
\newblock {\em Invent. Math.}, 95(3):615--628, 1989.

\bibitem{MR1306563}
Y.~Carri\`ere and L.~Rozoy.
\newblock Compl\'{e}tude des m\'{e}triques lorentziennes de {${\bf T}^2$} et
  diff\'{e}ormorphismes du cercle.
\newblock {\em Bol. Soc. Brasil. Mat. (N.S.)}, 25(2):223--235, 1994.

\bibitem{MR1658468}
D.~J. Collins, R.~I. Grigorchuk, P.~F. Kurchanov, and H.~Zieschang.
\newblock {\em Combinatorial group theory and applications to geometry}.
\newblock Springer-Verlag, Berlin, 1998.
\newblock Translated from the 1990 Russian original by P. M. Cohn, Reprint of
  the original English edition from the series Encyclopaedia of Mathematical
  Sciences [{{\i}t Algebra. VII}, Encyclopaedia Math. Sci., 58, Springer,
  Berlin, 1993; MR1265269 (95g:57004)].

\bibitem{MR0336755}
C.~Godbillon.
\newblock Fibr\'{e}s en droites et feuilletages du plan.
\newblock {\em Enseignement Math. (2)}, 18:213--224 (1973), 1972.

\bibitem{MR0219084}
C.~Godbillon and G.~Reeb.
\newblock Fibr\'{e}s sur le branchement simple.
\newblock {\em Enseignement Math. (2)}, 12:277--287, 1966.

\bibitem{MR1310948}
M.~Guediri and J.~Lafontaine.
\newblock Sur la compl\'{e}tude des vari\'{e}t\'{e}s pseudo-riemanniennes.
\newblock {\em J. Geom. Phys.}, 15(2):150--158, 1995.

\bibitem{MR0089412}
A.~Haefliger and G.~Reeb.
\newblock Vari\'{e}t\'{e}s (non s\'{e}par\'{e}es) \`a une dimension et
  structures feuillet\'{e}es du plan.
\newblock {\em Enseignement Math. (2)}, 3:107--125, 1957.

\bibitem{MR881799}
G.~Hector and U.~Hirsch.
\newblock {\em Introduction to the geometry of foliations. {P}art {A}},
  volume~1 of {\em Aspects of Mathematics}.
\newblock Friedr. Vieweg \& Sohn, Braunschweig, second edition, 1986.
\newblock Foliations on compact surfaces, fundamentals for arbitrary
  codimension, and holonomy.

\bibitem{MR1411352}
B.~Klingler.
\newblock Compl\'{e}tude des vari\'{e}t\'{e}s lorentziennes \`a courbure
  constante.
\newblock {\em Math. Ann.}, 306(2):353--370, 1996.

\bibitem{MR977043}
J.~Lafuente~L\'{o}pez.
\newblock A geodesic completeness theorem for locally symmetric {L}orentz
  manifolds.
\newblock {\em Rev. Mat. Univ. Complut. Madrid}, 1(1-3):101--110, 1988.

\bibitem{LM}
L.~Mehidi.
\newblock On the existence and stability of two-dimensional lorentzian tori
  without conjugate points.
\newblock {\em arXiv:1902.04505}, 2019.

\bibitem{MR1230550}
A.~Romero and M.~S\'{a}nchez.
\newblock On the completeness of geodesics obtained as a limit.
\newblock {\em J. Math. Phys.}, 34(8):3768--3774, 1993.

\bibitem{MR1267937}
A.~Romero and M.~S\'{a}nchez.
\newblock New properties and examples of incomplete {L}orentzian tori.
\newblock {\em J. Math. Phys.}, 35(4):1992--1997, 1994.

\bibitem{MR1376554}
M.~S\'{a}nchez.
\newblock Structure of {L}orentzian tori with a {K}illing vector field.
\newblock {\em Trans. Amer. Math. Soc.}, 349(3):1063--1080, 1997.

\bibitem{MR532830}
M.~Spivak.
\newblock {\em A comprehensive introduction to differential geometry. {V}ol.
  {I}}.
\newblock Publish or Perish, Inc., Wilmington, Del., second edition, 1979.

\end{thebibliography}
\end{document}